\documentclass[multixcb]{amstext-l}
\usepackage{url}
\usepackage{mathrsfs}
\usepackage{hyperref}

\DeclareMathAlphabet{\mathzc}{OT1}{pzc}{m}{it}

\theoremstyle{plain}
\newtheorem{theorem}{Theorem}[chapter]

\theoremstyle{definition}

\theoremstyle{definition}
\newtheorem{definition}[theorem]{Definition}

\theoremstyle{remark}
\newtheorem{remark}[theorem]{Remark}

\numberwithin{equation}{chapter}
\DeclareMathOperator{\Tr}{Tr}
\newcommand{\diag}{\mathop{\rm diag}}

\def\<#1>{$\langle$\textit{#1}$\rangle$}

\begin{document}

\frontmatter

\makeatletter
\def\currversion{\csname ver@amstext-l.cls\endcsname}
\makeatother

\title{Approximation of a Multivariate Function of Bounded Variation from its Scattered Data}
\author{Rajesh Dachiraju\\Hyderabad, India\\rajesh.dachiraju@gmail.com\\[12pt]
        \smaller \currversion}
%        \smaller Version 0.92 beta, March 2010}

\maketitle
\begin{abstract}
    In this paper, we address the problem of approximating a function of bounded variation from its scattered data. Radial basis function(RBF) interpolation methods are known to approximate only functions in their native spaces, and to date, there has been no known proof that they can approximate functions outside the native space associated with the particular RBF being used. In this paper, we describe a scattered data interpolation method which can approximate any function of bounded variation from its scattered data as the data points grow dense. As the class of functions of bounded variation is a much wider class than the native spaces of the RBF, this method provides a crucial advantage over RBF interpolation methods.
\end{abstract}
\tableofcontents

\mainmatter

\chapter{Introduction}\label{Intro_chap}

\section{Introduction}

\subsection{Approximation via Interpolation}
One method of approximating a function when a finite number of its samples are given is by interpolating the sample data points. If the interpolation method is that the interpolant converges to the function as the data points become dense in the domain, then approximation via interpolation is assumed to be achieved, and the interpolation method is assumed to have the approximation property. Let $\Omega\subset\mathbb{T}^m$ be a bounded Lipschitz domain, and let $\psi:\Omega\to\mathbb{R}$ be a function from a specified function space. $\psi$ is the function to be approximated. Let $D$ be a countable dense subset of $\Omega$. The scattered data constitute a set of $n$ distinct points $\{\boldsymbol{p}_i/\boldsymbol{p}_i\in D, i=1,2,\ldots n\}$ chosen from $D$ without assumptions on their geometry and the corresponding values of $\psi$ evaluated at these points $\psi(\boldsymbol{p}_i)$. The set of data points is denoted as $E_n = \{\boldsymbol{p}_i/\boldsymbol{p}_i\in D, i = 1,2,3,..n\}$. Scattered data interpolation aims to obtain a function $f^n$ that interpolates the data in $E_n$; in other words,  $f^n(\boldsymbol{p}_i) = \psi(\boldsymbol{p}_i)\mbox{   }i=1,2,\ldots n$. The interpolation method is assumed to have an approximation property if the interpolation method is such that as $n\to\infty$, $f^n\to \psi$ under some suitable norm.

\subsection{Scattered Data Interpolation}

Given a set of data points in a domain and the corresponding values to be attained at those points, a method of selecting a function from a known class of functions that attains those values at the corresponding data points is referred to as interpolation. Interpolation is a very old topic in mathematics with wide practical applications in many fields. It has been widely studied, and there is an enormous corpus of literature on this topic. This paper undertakes a brief exposure, focusing only on the methods that are relevant and comparing them with methods proposed in this study. We mainly concentrate on multivariate scattered data interpolation methods and their ability to approximate functions over a bounded domain. Scattered data interpolation methods can be broadly classified into two categories: (a) polynomial interpolation methods \cite{gasca2000polynomial, de1990multivariate} (b) non-polynomial interpolation methods. The easiest of the polynomial interpolation methods are tensor product methods, but they require a prescribed geometry on the given data points, rendering them not very useful in the context of this study. Piece-wise polynomial approaches, such as multivariate spline interpolation \cite{de1983approximation}, exist, but they require the triangulation of data points, which is non-trivial and the methods are too specific to the dimension of the Euclidean space in which interpolation is being performed. Consequently, they are not very useful when seamless mobility across dimensions is required. For example, if one designs a method in two-dimensional space, it will not be readily useful for a seven-dimensional space without a substantial amount of work. In addition, the accuracy of the approximation substantially depends on the triangulation method used. General, non-polynomial methods date back to Shepard's method \cite{10.1145/800186.810616, gordon1978shepard} that provides easy methods to generate interpolants in any dimensional space. However, this method generally does not provide good interpolation accuracy, suffers from the interpolant having stationary points or vanishing gradients (flat regions) near all the data points, and is unduly influenced by distant points. There are recently developed, mathematically well-studied methods that are widely used in both higher and lower dimensions, which are referred to as radial basis function interpolation (RBF) methods. They have the advantage of being generic for any dimension and are the closest to the method presented in this paper. We provide a brief exposition to the RBF interpolation methods and discuss their advantages, the functions they can approximate, and their disadvantages. For a brief overview of the scattered data interpolation methods, one can refer to the review paper by Franke et al. \cite{franke1991scattered}.

\subsection{Approximation Using Radial Basis Function Interpolation Methods}
Let $\psi$ be the function to be approximated. When the data are scattered, the common choice for such an approximation is the radial basis-function interpolation method. We provide a brief exposition of the radial basis-function interpolation methods. The standard radial basis function interpolant is generally of the form 
\begin{equation}\label{eq1}
%\label{rbf}
    f^n(\boldsymbol{x}) = \sum_{i=1}^n c_i\phi(\|\boldsymbol{x}-\boldsymbol{p}_i\|), x\in\mathbb{R}^m.
\end{equation}
The function $\phi:\mathbb{R}_+\to\mathbb{R}$ is called the radial basis function. The coefficients $c_i$ can be determined uniquely from the interpolation requirements $f^n(\boldsymbol{p}_i) = \psi(\boldsymbol{p}_i)$, which involves solving a set of linear equations, by inverting a matrix $A = [a_{i,j}]_{n\times n}$, where $a_{i,j} = \phi(\|\boldsymbol{p}_i-\boldsymbol{p}_j\|)$. We refer to this matrix $A$ as the interpolation matrix. The radial basis function $\phi$ is sometimes strictly positive definite; for example, the Gaussian radial basis function $\phi(r) = e^{-\sigma^2 r^2}$ and inverse multiquadric function $\phi(r) = 1/\sqrt{r^2+\sigma^2}$ render the interpolation matrix $A$ positive definite, thereby rendering the coefficients $c_i$ uniquely solvable. Sometimes, $\phi(r)$ is only conditionally positive definite of some order $k$ on $\mathbb{R}^m$. Examples of such radial basis functions are the so-called thin plate splines. For these types of radial basis functions, polynomials $p(\boldsymbol{x}) \in \mathbb{P}^{k-1}_m(\boldsymbol{x})$ of degree $k-1$ in $m$ unknowns are augmented to equation \ref{eq1} to render the interpolation problem uniquely solvable. Thus, we obtain the interpolant as
\begin{equation}\label{eq2}
%\label{approximant}
    f^n(\boldsymbol{x}) = \sum_{i=1}^n c_i\phi(\|\boldsymbol{x}-\boldsymbol{p}_i\|) + p(\boldsymbol{x}), \boldsymbol{x}\in \mathbb{R}^m.
\end{equation} 
The extra degrees of freedom are obtained by requiring that the coefficient vector $\boldsymbol{c} = [c_1,c_2,...c_n]^T \in \mathbb{R}^n$ satisfy 
\begin{equation}\label{eq3}
    \sum_{i=1}^nc_iq(\boldsymbol{p}_i) = 0, \forall q\in\mathbb{P}^{k-1}_m.
\end{equation}
To ensure that the problem is solvable, the set of data points $E$ should contain a $\mathbb{P}^{k-1}_m$-unisolvent subset. This is the only mild assumption about the data that are required by the radial basis function interpolation methods using radial basis functions that are conditionally positive definite of the order $k$. For the special case of the linear radial basis functions $\phi(r) = r$, the interpolation matrix is non-singular even without augmentation of the polynomials to the interpolant. Does the radial basis-function interpolation method have the approximation property as the data points become dense in the domain? The answer to this question is ``yes'' when the function $\psi$ (the one to be approximated) is in the  reproducing kernel Hilbert space(RKHS space)(also known as the native space) corresponding to the radial basis function $\phi(r)$ used. The question of convergence was first answered by Powell \cite{powell1994uniform} in two dimensions and general dimensions by Duchon for the special case of the RBF being a thin plate spline or any of its siblings, under the unisolvency assumption on the scattered data and the domain of the function $\psi$. The use of native space methods was pioneered by Duchon \cite{duchon1976interpolation, duchon1977splines, duchon1978erreur}, where he derived thin plate spline-type RBF's using the variational principle. They are of the form 
\begin{equation}\label{eq4}\begin{aligned}
%\label{tp_splines_eq}
%\begin{align}
\phi(r) &= r^{2k-m}\log{r},&\mbox{ if } 2k-m \mbox{ is an even integer,}\\
        &= r^{2k-m},&\mbox{ if } 2k-m \mbox{ is not an even integer.}
        \end{aligned}.
\end{equation}
By defining the fill distance(or mesh norm) $h_n := \sup_{x\in\Omega}\inf_{p\in E_n}\|x-p\|$, one can state the approximation property of the thin plate spline-type radial basis function interpolation in the following theorem due to Bejancu \cite{bejancu1997uniform}.
\begin{theorem}
Let $\phi$ be from class (\ref{eq4}) and $\Omega$ be bounded and contains a $\mathbb{P}^k_m$-unisolvent subset. Let $f^n$ be the RBF interpolant as in equation (\ref{eq2}); then, there is a constant $C$ that is independent of $h$ such that
\begin{equation}\label{eq5}
\begin{aligned}
%\label{tp_splines_conv}
%\begin{align}
\|f^n - \psi\| &\le C h_n\sqrt{\log{1/h_n}}, \mbox{   if } 2k-m=2,\\% \mbox{ is an even integer,}\\
            &\le C\sqrt{h_n}, \mbox{   if } 2k-m=1 \mbox{ and } 0<h<1.\\
            &\le C h_n \mbox{ in all other cases.}
        \end{aligned}
\end{equation}
\end{theorem} 
It should be noted that $\lim\limits_{n\to\infty}h_n = 0$, so the above theorem ensures convergence and approximates the function $\psi$ when the points become dense in $\Omega$.
A few years after Duchon's paper was published, Madych and Nelson \cite{madych1990multivariate} developed an approach dealing with general RBF interpolation and provided related error estimates using a variational principle. Wu and Schaback \cite{wu1993local}, using Kriging methods, obtained many explicit, useful error estimates for interpolation via RBFs. More improvements in error estimates were achieved by Schaback \cite{schaback1999improved}. Some of these results were improved in terms of convergence rates, for which one can refer to the review paper by Buhmann \cite{buhmann_2000}.

In many applications, for example, in mesh-less methods for solving PDEs or in the field of statistical learning theory, the function $\psi$ generating the data may not be sufficiently smooth or have the right properties to be in the native space of the RBF. In summary, there is not enough space in the native space. For approximating a function that is outside the native space of the RBF, or, for example, any continuous function $\psi$, to date, the convergence results are unknown, and efforts are being made in this direction, for example   Narcowich et al, \cite{narcowich2004scattered, narcowich2005sobolev} and Brownlee et al, \cite{brownlee2004approximation} with some partial progress in the form of obtaining some bounds on the error, which are not strong enough to establish convergence. Motivated by the problem of escaping the native space, Yoon \cite{yoon2001spectral} used thin plate splines that depend on a parameter $\lambda$ that is scaled with the spacing of the data and investigated interpolation errors associated with data-dependent RBFs. Again the bounds obtained were not strong enough to prove convergence. Some of these efforts to approximate functions outside the native space are summarized in the review paper by Narcowich \cite{narcowich2005recent}. For a detailed overview of the RBF methods, one can refer to Buhmann's book on radial basis functions \cite{buhmann_2003}. 

One major disadvantage of the RBF methods is that there is no known upper bound on the condition number of the interpolation matrix, which is independent of the data. The only known bounds depend directly on the number of data points $n$ or implicitly on the minimum separation distance of the data. Both these bounds become unbounded as the data points increase in density, as given in \cite{narcowich1991norms}. There has been an experimental study \cite{boyd2011numerical} demonstrating that the condition number becomes unbounded as the data points become dense.

In this paper we introduce a scattered data interpolation method using trigonometric polynomials. It is shown that this interpolation method can be used for approximation of functions of bounded variation defined on a torus $\mathbb{T}^m$, from their scattered data. We began with an introduction and a brief overview of radial basis functions (RBF) interpolation methods in Chapter \ref{Intro_chap}. In Chapter \ref{chap_2} we describe a scattered data interpolation method for approximating continuous Sobolev functions from their scattered data. In Chapter \ref{chap_3} we adopt this method for scattered data interpolation using trigonometric polynomials. In Chapter \ref{chap_4} we show that this method can be used to approximate functions of bounded variation from their scattered data. (It has to be noted that, in this paper, by functions of bounded variation we mean that these functions have  a finite but non zero total variation(in the Vitali sense) and also they do not contain removable discontinuities).

%\section{Variational Approach}

%The thin plate spline radial basis function approach has roots in a variational approach. %This is a limiting case of least square data fitting with a penalty/regularization term. %This was first proposed by Duchon \cite{duchon1977splines}. The functional involved a %rotational invariant semi norm consisting of the norm of homogeneous derivatives of the %function. One fundamental question is ``why the norm of the derivative is minimized over %the entire $\mathbb{R}^m$ while we are interested in interpolating over a bounded %domain?.'' Another question is ``why not consider functions defined over Torus?'' (compact %support with periodic boundary conditions). Defining over Torus also discretizes the %Fourier spectrum of the function, which enables easier analysis. In this study, we used a %variational approach involving functions defined over a torus and showed that, for an %extremal value of a parameter, we achieve interpolation.

%\section{Motivation for the functional}
\chapter{Interpolation using Sobolev Functions}\label{chap_2}

\section{Minimization Problem}
\label{mnmz}
\subsection{Definitions}

Let $H^k(\Omega)$ denote the Sobolev Hilbert space of the functions defined on the set $\Omega$, $\mathbb{T}^m$ denote the $m $ - dimensional Torus. Defining a function on a torus $\mathbb{T}^m$ means the function is defined over $(0,1)^m$ and it is periodic with a period $T = (1,1\ldots 1) \in \{1\}^m$. Let $\mathbb{Z}^m$ denote the set containing all the $m$-tuples of integers.
 \begin{definition}
\label{kgradient}
We define the $k$-gradient as 
\begin{equation}
\nabla^kf = (\frac{\partial^{k}f}{\partial x_1^{k}},\frac{\partial^{k}f}{\partial x_2^{k}},...\frac{\partial^{k}f}{\partial x_m^{k}}   ). 
\end{equation}
\end{definition}
We define the functional $C_{\lambda}$ as 
\begin{equation}
\label{eq7}
%\label{functional}
 C_{\lambda}(f) = \frac{\lambda^2}{n}\sum\limits_{i=1}^{n}(f(\boldsymbol{p}_i)-q_i)^2 + \lambda\|\nabla^kf\|_{L^2(\mathbb{T}^m)}^2 + \|f\|_{L^2(\mathbb{T}^m)}^2, \end{equation}
where $k,m \in \mathbb{N}, k>\frac{m}{2}, 
\lambda \in \mathbb{R}^+ \mbox{ and } f \in C^0(\mathbb{T}^m)\cap H^k(\mathbb{T}^m)$. 

The minimization problem is minimizing the functional defined in Equation \ref{eq7} over the space $ C^0(\mathbb{T}^m)\cap H^k(\mathbb{T}^m)$ .

In \cite{Afunctionfittingmethod}(in a sequence of theorems (Theorems 1-4)), a similar functional \begin{equation}\label{sim_func}
B_{\lambda}(f) =  \sum\limits_{i=1}^{n}(f(\boldsymbol{p}_i)-q_i)^2 + \lambda\|\nabla^kf\|_{L^2(\mathbb{T}^m)}^2 + \|f\|_{L^2(\mathbb{T}^m)}^2
\end{equation}defined on the same space $C^0(\mathbb{T}^m)\cap H^k(\mathbb{T}^m)$ has been considered, and it has been shown that this functional $B_{\lambda}$ has a unique minimizer in $C^0(\mathbb{T}^m)\cap H^k(\mathbb{T}^m)$. Using similar technique it can be shown that the present functional in Equation \ref{eq7} also has a unique minimizer in the space $C^0(\mathbb{T}^m)\cap H^k(\mathbb{T}^m) $, the proof of which is given in Appendix \ref{unique_minimizer} .

\subsection{Euler--Lagrange (E--L) Equation }

We now derive the Euler--Lagrange (E--L) equation of the minimization problem posed in the previous section and show that it is a linear weak PDE with some global terms.

We minimize in $C^0(\mathbb{T}^m)\cap H^k(\mathbb{T}^m)$, \begin{equation}\label{eq8}C_{\lambda}(f) = \frac{\lambda^2}{n}\sum\limits_{i=1}^{n}(f(\boldsymbol{p}_i)-q_i)^2 + \lambda\|\nabla^kf\|_{L^2(\mathbb{T}^m)}^2 + \|f\|_{L^2(\mathbb{T}^m)}^2. \end{equation}

We derive the Euler--Lagrange equation for the above problem by steps for each term separately.
For any $\phi \in C^{\infty}(\mathbb{T}^m)\cap H^k(\mathbb{T}^m)$.

\begin{equation}\label{eq9}
\begin{aligned}
\frac{d}{ds}|_{s=0} \|f(\boldsymbol{x})+s\phi(\boldsymbol{x})\|_{L^2(\mathbb{T}^m)}^2 &= \frac{d}{ds}|_{s=0} \int_{\mathbb{T}^m} \left|f(\boldsymbol{x})+s\phi(\boldsymbol{x})\right|^{2} \mathop{}\!\mathrm{d}^m\boldsymbol{x}\\
&\stackrel{*}{=} 2\int_{\mathbb{T}^m} \phi(\boldsymbol{x}) f(\boldsymbol{x})\mathop{}\!\mathrm{d}^m\boldsymbol{x}, 
\end{aligned}
\end{equation} where $*$ can be justified by using the dominated convergence theorem

\begin{equation}\label{eq10}
\begin{aligned}
 \frac{d}{ds}|_{s=0} \lambda\| \nabla^k(f(\boldsymbol{x})+s\phi(\boldsymbol{x}))\|_{L^2(\mathbb{T}^m)} &= \frac{d}{ds}|_{s=0} \int_{\mathbb{T}^m} \lambda\left|\nabla^k f(\boldsymbol{x})+s \nabla^k \phi(\boldsymbol{x})\right|^2\mathop{}\!\mathrm{d}^m\boldsymbol{x}\\
&= 2\lambda\int_{\mathbb{T}^m} \nabla^k \phi(\boldsymbol{x}) \cdot \nabla^k f(\boldsymbol{x})\mathop{}\!\mathrm{d}^m\boldsymbol{x}
\end{aligned}
\end{equation}

\begin{equation}\label{eq11}\frac{d}{ds}|_{s=0} \frac{\lambda^2}{n}\sum\limits_{i = 1}^n|f(\boldsymbol{p}_i)+s\phi(\boldsymbol{p}_i)-q_i|^2 =  -\frac{2\lambda^2}{n}\sum\limits_{i = 1}^n (q_i-f(\boldsymbol{p}_i))\phi(\boldsymbol{p}_i)
\end{equation}

and by combining all terms, we obtain the following PDE as the Euler-Lagrange equation for the minimization problem.

\begin{equation}\label{eq12}
\begin{aligned}
%\label{E-L}
-\frac{\lambda^2}{n}\sum\limits_{i = 1}^{n}(q_i-f(\boldsymbol{p}_i))\phi(\boldsymbol{p}_i) + \lambda\int_{\mathbb{T}^m}\nabla^k\phi(\boldsymbol{x})\cdot \nabla^k f(\boldsymbol{x}) \mathop{}\!\mathrm{d}^m\boldsymbol{x}+  \int_{\mathbb{T}^m} \phi(\boldsymbol{x}) f(\boldsymbol{x})\mathop{}\!\mathrm{d}^m\boldsymbol{x}  = 0 \\
\forall \phi \in  C^{\infty}(\mathbb{T}^m)\cap H^k(\mathbb{T}^m)
\end{aligned}
\end{equation}

%\begin{equation}\label{eq13}\cancel{\int_{\Omega}\phi(\boldsymbol{x})\Delta^k f(\boldsymbol{x}) \mathop{}\!\mathrm{d}^m\boldsymbol{x}}+ \lambda \int_{\Omega} \phi(\boldsymbol{x}) f(\boldsymbol{x})\mathop{}\!\mathrm{d}^m\boldsymbol{x}   + \sum\limits_{i = 1}^{n}(f(\boldsymbol{p}_i)-q_i)\phi(\boldsymbol{p}_i) = 0 \forall \phi \in C^{\infty}(\Omega) \cap M\end{equation} 

The equation is not a PDE in the strict sense owing to the appearance of global terms, such as $f(\boldsymbol{p}_i) = \int_{\Omega} f(\boldsymbol{x})\delta(\boldsymbol{x}-\boldsymbol{p}_i)\mathop{}\!\mathrm{d}^m\boldsymbol{x}$.

\section{Solution to the E--L Equations}

In this section, we solve the E--L equation of the minimization problem. The existence and uniqueness of the minimizer have already been established, and since the functional is convex, it has only one stationary point. Hence, the existence and uniqueness of the solution to this E--L equation hold. We now derive the solution of the E--L equation.

\begin{theorem}\label{E-L_solution}
The solution to the PDE in Equation \ref{eq12} is $f_{\lambda}$, which is given as
\begin{equation}
    \label{exp_th}  
  f_{\lambda}(\boldsymbol{x}) = \sum\limits_{i=1}^n \frac{c_i}{n}g_{\lambda}(\boldsymbol{x}-\boldsymbol{p}_i),
 \end{equation}
 where
 \begin{equation}\label{eq22_th}
%\label{greenf}
    g_{\lambda}(\boldsymbol{x}) =  \sum_{\boldsymbol{l}\in\mathbb{Z}^m} \frac{1}{1+\lambda\|\boldsymbol{l}\|_{2k}^{2k}} \cos{\left(2\pi\boldsymbol{l}\cdot\boldsymbol{x}\right)}.
\end{equation}

 $\pmb{c} = [c_1,c_2,...c_n]^T$ is given as \begin{equation}\label{eq28_th}
%\label{iq2}
 \pmb{c} = (\frac{1}{n}G_{\lambda}+\frac{1}{\lambda^2}I)^{-1}L,\end{equation} where the matrix $G_{\lambda}$ is given as 
\begin{equation}\label{eq29_th}
% \label{iqd}
  G_{\lambda} = [\gamma_{ij}(\lambda)]_{n\times n},\gamma_{ij}(\lambda) = g_{\lambda}(\boldsymbol{p}_i-\boldsymbol{p}_j)\end{equation} and $$L = [q_1,q_2,\ldots q_n]^T.$$ 

\end{theorem}

\begin{proof}

Consider the following PDE equation:

\begin{equation}
\begin{aligned}
\label{sub_EL}
 -\int_{\mathbb{T}^m} \phi(\boldsymbol{x})\delta(\boldsymbol{x})\mathop{}\!\mathrm{d}^m\boldsymbol{x} + \lambda\int_{\mathbb{T}^m} \nabla^k\phi(\boldsymbol{x})\cdot\nabla^k f(\boldsymbol{x})\mathop{}\!\mathrm{d^m}x &+ \int_{\mathbb{T}^m} \phi(\boldsymbol{x})f(\boldsymbol{x})\mathop{}\!\mathrm{d}^m\boldsymbol{x}\\ &= 0 \forall \phi\in C^{\infty}(\mathbb{T}^m)\cap H^k(\mathbb{T}^m).    
\end{aligned}
\end{equation}
      
Let $g$ be its solution. Now, consider the equation 
\begin{equation}
\label{sub_EL_2}
    \begin{aligned}
        -\frac{c_i}{n}\int_{\mathbb{T}^m} \phi(\boldsymbol{x})\delta(\boldsymbol{x}-\boldsymbol{p}_i)\mathop{}\!\mathrm{d}^m\boldsymbol{x} + \lambda\int_{\mathbb{T}^m} \nabla^k\phi(\boldsymbol{x})\cdot\nabla^k f(\boldsymbol{x})\mathop{}\!\mathrm{d}^m\boldsymbol{x} &+ \int_{\mathbb{T}^m} \phi(\boldsymbol{x})f(\boldsymbol{x})\mathrm{d^m}x\\ &= 0 \forall \phi\in C^{\infty}(\mathbb{T}^m)\cap H^k(\mathbb{T}^m).  
    \end{aligned}
\end{equation}

Substituting $f = c_ig(\boldsymbol{x}-\boldsymbol{p}_i)$ in the LHS of the equation \ref{sub_EL_2} and  denoting it as $J$, we obtain 
\begin{equation}
\label{der_sub_EL_1}
\begin{aligned}
J(\phi) = -\frac{c_i}{n}\int_{\mathbb{T}^m} \phi(\boldsymbol{x})\delta(\boldsymbol{x}-\boldsymbol{p}_i)\mathop{}\!\mathrm{d}^m\boldsymbol{x} &+ \lambda\int_{\mathbb{T}^m} \nabla^k\phi(\boldsymbol{x})\cdot\nabla^k c_ig(\boldsymbol{x}-\boldsymbol{p}_i)\mathop{}\!\mathrm{d}^m\boldsymbol{x}\\
&+ \int_{\mathbb{T}^m} \phi(\boldsymbol{x})c_ig(\boldsymbol{x}-\boldsymbol{p}_i)\mathop{}\!\mathrm{d}^m\boldsymbol{x}.\\
\end{aligned}
\end{equation}
Substituting $\boldsymbol{t} = \boldsymbol{x}-\boldsymbol{p}_i$, we obtain
\begin{equation}
    \label{der_sub_EL_2}
    \begin{aligned}
     J(\phi) = -\frac{c_i}{n}\int_{\mathbb{T}^m} \phi(\boldsymbol{t}+\boldsymbol{p}_i)\delta(\boldsymbol{t})\mathop{}\!\mathrm{d}^m\boldsymbol{t} &+ \frac{c_i}{n}\lambda\int_{\mathbb{T}^m} \nabla^k\phi(\boldsymbol{t}+\boldsymbol{p}_i)\cdot\nabla^k g(\boldsymbol{t})\mathop{}\!\mathrm{d}^m\boldsymbol{t}\\
     &+ \frac{c_i}{n}\int_{\mathbb{T}^m} \phi(\boldsymbol{t}+\boldsymbol{p}_i)g(\boldsymbol{t})\mathop{}\!\mathrm{d}^m\boldsymbol{t}.\\
     \end{aligned}
\end{equation}

Let $\theta(\boldsymbol{t}) = \phi(\boldsymbol{t}+\boldsymbol{p}_i)$, so we have 

\begin{equation}
    \label{der_sub_EL_3}
     J(\phi) = \frac{c_i}{n}\left\{-\int_{\mathbb{T}^m} \theta(\boldsymbol{t})\delta(t)\mathop{}\!\mathrm{d}^m\boldsymbol{t} + \lambda\int_{\mathbb{T}^m} \nabla^k\theta(\boldsymbol{t})\cdot\nabla^k g(\boldsymbol{t})\mathop{}\!\mathrm{d}^m\boldsymbol{t} + \int_{\mathbb{T}^m} \theta(\boldsymbol{t})g(\boldsymbol{t})\mathop{}\!\mathrm{d}^m\boldsymbol{t}\right\}.\\
\end{equation}

For every $\phi \in C^{\infty}(\mathbb{T}^m)\cap H^k(\mathbb{T}^m)$, we have $\theta \in C^{\infty}(\mathbb{T}^m)\cap H^k(\mathbb{T}^m)$, and using the fact that $g(t)$ is the solution of the Equation \ref{sub_EL}, we have

\begin{equation}
    J(\phi) = 0 \forall \phi\in C^{\infty}(\mathbb{T}^m)\cap H^k(\mathbb{T}^m).
\end{equation}

Hence, $\frac{c_i}{n}g(\boldsymbol{x}-\boldsymbol{p}_i)$ is a solution to Equation \ref{sub_EL_2}.  Writing Equation \ref{sub_EL_2} with different $c_i$, $i = 1,2,3...n$ and substituting $f=\frac{c_i}{n}g(\boldsymbol{x}-\boldsymbol{p}_i)$ in the $i^{th}$ equation (as it the solution of that equation), and adding up all the $n$ equations, we obtain 

\begin{equation}
    \label{der2_sub_EL}
    \begin{aligned}
     \sum\limits_{i=1}^n\left\{-\frac{c_i}{n}\int_{\mathbb{T}^m} \phi(\boldsymbol{x})\delta(\boldsymbol{x}-\boldsymbol{p}_i)\mathop{}\!\mathrm{d}^m\boldsymbol{x}\right\} &+\\ \lambda\int_{\mathbb{T}^m} \nabla^k\phi(\boldsymbol{x})\cdot\nabla^k\left( \sum\limits_{i=1}^n \frac{c_i}{n}g(\boldsymbol{x}-\boldsymbol{p}_i)\right)\mathop{}\!\mathrm{d}^m\boldsymbol{x} &+ \int_{\mathbb{T}^m} \phi(\boldsymbol{x})\sum\limits_{i=1}^n \frac{c_i}{n}g(\boldsymbol{x}-\boldsymbol{p}_i)\mathop{}\!\mathrm{d}^m\boldsymbol{x} = 0\\ \forall \phi\in C^{\infty}(\mathbb{T}^m)\cap H^k(\mathbb{T}^m).
     \end{aligned}
\end{equation}

Denoting $f_{\lambda} = \sum\limits_{i=1}^n \frac{c_i}{n}g(\boldsymbol{x}-\boldsymbol{p}_i)$ and assuming $c_i = \lambda^2(q_i-f(\boldsymbol{p}_i))$ and noting that $\int_{\mathbb{T}^m}\phi(\boldsymbol{x})\delta(\boldsymbol{x}-\boldsymbol{p}_i)\mathrm{d}x = \phi(\boldsymbol{p}_i)$, we can rewrite Equation \ref{der2_sub_EL} as 

\begin{equation}
    \label{der3_sub_EL}
    \begin{aligned}
      -\frac{\lambda^2}{n}\sum\limits_{i=1}^n (q_i-f_{\lambda}(\boldsymbol{p}_i))\phi(\boldsymbol{p}_i) + \lambda\int_{\mathbb{T}^m} \nabla^k\phi(\boldsymbol{x})\cdot\nabla^k f_{\lambda}(\boldsymbol{x})\mathop{}\!\mathrm{d}^m\boldsymbol{x} &+ \int_{\mathbb{T}^m} \phi(\boldsymbol{x})f_{\lambda}(\boldsymbol{x})\mathop{}\!\mathrm{d}^m\boldsymbol{x}\\
     &= 0 \forall \phi\in C^{\infty}(\mathbb{T}^m)\cap H^k(\mathbb{T}^m),
     \end{aligned}
\end{equation}

which is same as the E--L equation, as in Equation \ref{eq12}. Hence, $$f_{\lambda}(\boldsymbol{x}) = \sum\limits_{i=1}^n \frac{c_i}{n}g(\boldsymbol{x}-\boldsymbol{p}_i)$$ is the solution of the E-L equation. However, we still have no expression for $c_i$ and $g(\boldsymbol{x})$. To determine $g$, we need to solve Equation \ref{sub_EL} as $g$ is its solution. Let $\boldsymbol{l} = (l_1,l_2,l_3,..l_m)\in \mathbb{Z}^m$. Let $\hat{g}_{\boldsymbol{l}}$ and $\hat{\phi}_{\boldsymbol{l}}$ be the Fourier series coefficients of $g$ and $\phi$. Using Parseval's theorem, we have the following equations

\begin{equation}\label{eq14}
\int_{\mathbb{T}^m}\nabla^k\phi(\boldsymbol{x})\cdot \nabla^k g(\boldsymbol{x}) \mathop{}\!\mathrm{d}^m\boldsymbol{x} = \sum\limits_{i = 1}^m( \sum\limits_{\boldsymbol{l} \in \mathbb{Z}^k}(2\pi l_i)^{2k}\hat{g}_{\boldsymbol{l}} \hat{\phi}_{\boldsymbol{l}}).
\end{equation}

\begin{equation}\label{eq15}
\int_{\mathbb{T}^m}\phi(\boldsymbol{x})g(\boldsymbol{x}) \mathop{}\!\mathrm{d}^m\boldsymbol{x} = \sum\limits_{\boldsymbol{l} \in \mathbb{Z}^k}\hat{g}_{\boldsymbol{l}} \hat{\phi}_{\boldsymbol{l}}.
\end{equation}

\begin{equation}\label{eq16}
    \int_{\mathbb{T}^m}\phi(\boldsymbol{x})\delta(\boldsymbol{x})\mathop{}\!\mathrm{d}^m\boldsymbol{x} = \frac{1}{N} \sum\limits_{\boldsymbol{l} \in \mathbb{Z}^k} \hat{\phi}_{\boldsymbol{l}}. 
\end{equation}

Combining these equations in Equation \ref{sub_EL}, we obtain 

\begin{equation}\label{eq17}
%\label{peq}
     -\sum\limits_{\boldsymbol{l} \in \mathbb{Z}^k} \hat{\phi}_{\boldsymbol{l}} + \lambda\sum\limits_{i = 1}^m( \sum\limits_{\boldsymbol{l} \in \mathbb{Z}^k}(2\pi l_i)^{2k}\hat{g}_{\boldsymbol{l}} \hat{\phi}_{\boldsymbol{l}}) + \sum\limits_{\boldsymbol{l} \in \mathbb{Z}^k}\hat{g}_{\boldsymbol{l}} \hat{\phi}_{\boldsymbol{l}} = 0 .
\end{equation}

Now consider the function $\theta(x) = \cos{(2\pi\boldsymbol{\eta}\cdot\boldsymbol{x})}+i\sin{(2\pi\boldsymbol{\eta}\cdot\boldsymbol{x})}$ and let $\hat{\theta}_{\boldsymbol{l}}$ be its Fourier series coefficients. Then, by substituting this $\hat{\theta}_{\boldsymbol{l}}$ for $\hat{\phi}_{\boldsymbol{l}}$ in Equation \ref{eq17}, we obtain 

\begin{equation}\label{eq19}
    -1 + \lambda\sum\limits_{i = 1}^m (2\pi \eta{_i})^{2k}\hat{g}_{\boldsymbol{\eta}} + \hat{g}_{\boldsymbol{\eta}} = 0, 
\end{equation}

which implies 
\begin{equation}\label{eq20}
    \hat{g}_{\boldsymbol{\eta}} = \frac{1}{1+2\pi\lambda\|\boldsymbol{\eta}\|_{2k}^{2k}}.  
\end{equation}

Applying this for each of $\boldsymbol{\eta} \in \mathbb{Z}^m$, we obtain the solution for Equation \ref{sub_EL} as $g$ whose Fourier series coefficients $\hat{g}_{\boldsymbol{l}}$ are given as \begin{equation}\label{eq21}
    \hat{g}_{\boldsymbol{l}} =  \frac{1}{1+\lambda\|\boldsymbol{l}\|_{2k}^{2k}}, %\boldsymbol{l}\in\mathbb{Z}^m . 
\end{equation} 

Let us denote this solution as $g_{\lambda}$. Thus, by Fourier series expansion, we obtain 

\begin{equation}\label{eq22}
%\label{greenf}
    g_{\lambda}(\boldsymbol{x}) =  \sum_{\boldsymbol{l}\in\mathbb{Z}^m} \frac{1}{1+\lambda\|\boldsymbol{l}\|_{2k}^{2k}} \cos{(2\pi\boldsymbol{l}\cdot\boldsymbol{x})}.
\end{equation}

Using $c_i = \lambda^2(q_i-f(\boldsymbol{p}_i))$ and that $f_{\lambda}(\boldsymbol{x}) = \frac{1}{n}\sum\limits_{i = 1}^n c_i g_{\lambda}(\boldsymbol{x}-\boldsymbol{p}_i)$ substituting the values of $f_{\lambda}(\boldsymbol{p}_i)$ from the later expression in the former equation, we obtain $n$ equations in $n$ unknowns $c_i$. Thus, we can solve for the $c_i$. Further, we obtain a matrix expression for $\pmb{c} = [c_1,c_2,...c_{n}]^T$ and is given as \begin{equation}\label{eq28}
%\label{iq2}
 \pmb{c} = (\frac{1}{n}G_{\lambda}+\frac{1}{\lambda^2}I)^{-1}L,\end{equation} where the matrix $G_{\lambda}$ is given as 
\begin{equation}\label{eq29}
% \label{iqd}
  G_{\lambda} = [\gamma_{ij}(\lambda)]_{n\times n},\gamma_{ij}(\lambda) = g_{\lambda}(\boldsymbol{p}_i-\boldsymbol{p}_j)\end{equation} and $L = [q_1,q_2,\ldots q_n]^T$. The solution exists and unique if and only if, in the Equation \ref{eq28}, the matrix $\frac{1}{n}G_{\lambda}+\frac{1}{\lambda^2}I$ is invertible.As we are solving the E-L equation of the minimization problem where the functional is convex, the minimizer if exists, needs to be solution of the E-L equation and as we have already established the existence and uniqueness of the minimizer to the minimization problem (stated in Section \ref{mnmz} and proof given in Appendix \ref{unique_minimizer}), the existence of a solution to the E-L equation follows. So it is already justified to assume that the matrix $\frac{1}{n}G_{\lambda}+\frac{1}{\lambda^2}I$ in Equation \ref{eq28} is indeed invertible(Otherwise there will not be any solution of the E-L equation, which contradicts and existence and uniqueness of the minimizer of the convex functional). Hence, the matrix $\frac{1}{n}G_{\lambda}+\frac{1}{\lambda^2}I$ is invertible, allowing us to determine the unique minimizer of the functional $C_{\lambda}(f)$ given in Equation \ref{eq7} over the set $ C^0(\mathbb{T}^m)\cap H^k(\mathbb{T}^m)$ as \begin{equation}
    \label{exp}  
  f_{\lambda}(\boldsymbol{x}) = \sum\limits_{i=1}^n \frac{c_i}{n}g_{\lambda}(\boldsymbol{x}-\boldsymbol{p}_i).
 \end{equation}\end{proof}
 
The matrix $G_{\lambda}$ is infact positive definite, which will be evident in the following section in the proof of Theorem \ref{eig_theorem}.
 
\section{Asymptotics of the Interpolation Matrix}\label{sec_asymp_exp_interp_matrix}

In this section we derive some asymptotics which we use in the proof of Theorem \ref{interpolation theorem}.

\subsection{Laurent Series of $\frac{1}{1+ay}$}\label{ls_sec}

Here, we state a Laurent series expansion of a function that we repeatedly use in the following sections. Consider the function $u(y) = \frac{1}{1+ay}$, where $a\in \mathbb{R}\setminus\{0\}$,$y\in\mathbb{R}^+$. We can expand this function using the Laurent series about $y=\infty$ as 
\begin{equation}\label{ls}
    \mbox{for  }  y > \frac{1}{|a|},\mbox{  } u(y) = \sum\limits_{r=1}^{\infty}(-1)^{r+1}(ay)^{-r}\mbox{   }.
\end{equation}

\subsection{Asymptotic Expansion of the Function $g_{\lambda}(\boldsymbol{x})$}

Here, we derive an asymptotic expansion for the function $g_{\lambda}$ as $\lambda\to\infty$.

The expression for $g_{\lambda}(\boldsymbol{x})$ is given in Equation \ref{eq22}. We restate it here
\begin{equation}\label{eq22r}
%\label{greenf}
    g_{\lambda}(\boldsymbol{x}) =  \sum_{\boldsymbol{l}\in\mathbb{Z}^m} \frac{1}{1+\lambda\|\boldsymbol{l}\|_{2k}^{2k}} \cos{2\pi\boldsymbol{l}\boldsymbol{x}},
\end{equation}

dividing the summation into two parts
\begin{equation}\label{eq22rs}
\begin{aligned}
    %\label{greenf}
    g_{\lambda}(\boldsymbol{x}) &=   \sum_{\boldsymbol{l}\in\{0\}^m} \frac{1}{1+\lambda\|\boldsymbol{l}\|_{2k}^{2k}} \cos{2\pi\boldsymbol{l}\boldsymbol{x}} + \sum_{\boldsymbol{l}\in\mathbb{Z}^m\setminus\{0\}} \frac{1}{1+\lambda\|\boldsymbol{l}\|_{2k}^{2k}} \cos{2\pi\boldsymbol{l}\boldsymbol{x}}\\
    &= 1 + \sum_{\boldsymbol{l}\in\mathbb{Z}^m\setminus\{0\}} \frac{1}{1+\lambda\|\boldsymbol{l}\|_{2k}^{2k}} \cos{2\pi\boldsymbol{l}\boldsymbol{x}}.
\end{aligned}
\end{equation}
Using the series expansion as stated in Section \ref{ls_sec}, we have 
\begin{equation}\label{eq22rs2}
\begin{aligned}
    %\label{greenf}
    g_{\lambda}(\boldsymbol{x}) &= 1 + \sum_{\boldsymbol{l}\in\mathbb{Z}^m\setminus\{0\}}\left( \cos{2\pi\boldsymbol{l}\boldsymbol{x}}\sum\limits_{r=1}^{\infty}\frac{(-1)^{r+1}}{(\lambda\|\boldsymbol{l}\|_{2k}^{2k})^r}\right) \\
    &= 1 +  \sum\limits_{r=1}^{\infty}\left(\frac{(-1)^{r+1}}{\lambda^{r}}\sum_{\boldsymbol{l}\in\mathbb{Z}^m\setminus\{0\}}\frac{1}{(\|\boldsymbol{l}\|_{2k}^{2k})^r}\cos{2\pi\boldsymbol{l}\boldsymbol{x}}\right).
\end{aligned}
\end{equation}

Denoting \begin{equation}\label{eq22rs3}
s_r(\boldsymbol{x}) = \sum_{\boldsymbol{l}\in\mathbb{Z}^m\setminus\{0\}}\frac{1}{(\|\boldsymbol{l}\|_{2k}^{2k})^r}\cos{2\pi\boldsymbol{l}\boldsymbol{x}},
\end{equation}

we obtain \begin{equation}\label{gasymp}
g_{\lambda}(\boldsymbol{x}) = 1 + \sum\limits_{r=1}^{\infty}\frac{(-1)^{r+1}}{\lambda^r}s_r(\boldsymbol{x}).
\end{equation}
 Using Equation \ref{eq22rs3} and Parseval's theorem, 
 \begin{equation}\label{prs_sr}
 \begin{aligned}
 \|\nabla^k s_r\|^2_{L^2(\mathbb{T}^m)}  &= \sum\limits_{\boldsymbol{l}\in\mathbb{Z}^m\setminus\{0\}^m}\frac{\|\boldsymbol{l}\|^{2k}_{2k}}{\|\boldsymbol{l}\|^{4kr}_{2k}}\\
 &= \sum\limits_{\boldsymbol{l}\in\mathbb{Z}^m\setminus\{0\}^m}\frac{1}{\|\boldsymbol{l}\|^{2k(2r-1)}_{2k}}\\
 &= P_r,
  \end{aligned}
 \end{equation}
where $P_r$ a finite positive constant for all $r\ge 1$. Hence,  $$s_r\in C^0(\mathbb{T}^m)\cap H^k(\mathbb{T}^m),\mbox{          }r=1,2,3\ldots$$
%Similar to the matrix $$G_{\lambda} = [g_{\lambda}(\boldsymbol{p}_i-\boldsymbol{p}_j)]_{1\le i,j\le n}$$ 
Define matrices $$T_r = [s_r(\boldsymbol{p}_i-\boldsymbol{p}_j)]_{1\le i,j\le n}\mbox{   r=1,2,3\ldots}.$$ then, using Equation \ref{gasymp}, we obtain the following asymptote \begin{equation}\label{Gasymp}
G_{\lambda} = T_0 + \sum\limits_{r=1}^{\infty}\frac{(-1)^{r+1}}{\lambda^r}T_r,    
\end{equation} 
where $T_0 = 1_{n\times n}$, an all-ones $n \times n$ matrix.

\subsection{Power Series for the eigenvalues} In this section, we estimate the asymptotes of the eigenvalues of all the matrices involved.

Let $\rho_l(G_{\lambda})$ be the $l^{th}$ eigenvalue of the matrix $G_{\lambda}$. From Equation \ref{Gasymp}, it is evident that $G_{\lambda}$ is an analytic perturbation of the all-ones matrix $T_0$ with the perturbation parameter $\epsilon = \frac{1}{\lambda}$. From \cite{kato2013perturbation}, we are aware that the eigenvalues of an analytically perturbed real symmetric matrix are also analytic perturbations of the eigenvalues of the original unperturbed real symmetric matrix. Thus, the eigenvalue $\rho_l(G_{\lambda})$ is an analytic function, and let it be of the form
\begin{equation}\label{eig_G}
\rho_l(G_{\lambda}) = a_{l0} + \frac{a_{l1}}{\lambda} +\frac{a_{l2}}{\lambda^2}+\ldots.
\end{equation}
In addition, owing to Weyl's inequality, as the perturbation parameter declines to zero, that is, as $\lambda\to\infty$, each eigenvalue of the perturbed matrix $G_{\lambda}$ converges to an eigenvalue of the unperturbed matrix $T_0$. The eigenvalues of the all-ones matrix $T_0$ are $n$ with multiplicity $1$ and $0$ with multiplicity $n-1$. Let the eigenvalue $\rho_1(G_{\lambda})$ converge to $n$, and the remaining eigenvalues($\rho_l(G_{\lambda})$,$l>1$) converge to zero. Therefore, in equation \ref{eig_G} assuming eigenvalues in descending order\begin{equation}\label{eig_ones}\begin{aligned}
a_{l0} &= n,\mbox{    } l = 1,\\
          &= 0, \mbox{    } l = 2,3,\ldots n.
          \end{aligned}
          \end{equation}

\begin{theorem}\label{eig_theorem}
 $$\mbox{for    }\lambda>1\mbox{,          } a_{l1} > 0 \mbox{,         } l= 2,3\ldots n.$$ 
\end{theorem}
\begin{proof}
Let $\boldsymbol{v}_l(\lambda) = [v_{l1}(\lambda),v_{l2}(\lambda),v_{l3}(\lambda),\ldots v_{ln}(\lambda)]^T$ be a normalized eigenvector of the matrix $G_{\lambda}$ corresponding to the eigenvalue $\rho_l(G_{\lambda})$ (normalized meaning $\|\boldsymbol{v}_l(\lambda)\|_2=1$). Now, we estimate the eigenvalue $\rho_l(G_{\lambda})$. Thus, by definition of eigenvalue 
\begin{equation}\label{eigest}
\begin{aligned}
    \rho_l(G_{\lambda}) &= \boldsymbol{v}_l^T(\lambda)G_{\lambda}\boldsymbol{v}_l(\lambda)\\
    &= \sum\limits_{i=1}^n\sum\limits_{j=1}^nv_{li}(\lambda)v_{lj}(\lambda)g_{\lambda}(\boldsymbol{p}_i-\boldsymbol{p}_j)\\
    &= \sum\limits_{i=1}^n\sum\limits_{j=1}^n\left\{v_{li}(\lambda)v_{lj}(\lambda)\sum_{\boldsymbol{\eta}\in\mathbb{Z}^m} \frac{1}{1+\lambda\|\boldsymbol{\eta}\|_{2k}^{2k}} \cos{(2\pi\boldsymbol{\eta}\cdot(\boldsymbol{p}_i-\boldsymbol{p}_j))}\right\}.
\end{aligned}
\end{equation} 

we know \begin{equation}\label{dmoiver}
    \cos{(2\pi\boldsymbol{\eta}\cdot(\boldsymbol{p}_i-\boldsymbol{p}_j))} = \frac{1}{2}\left(e^{2\pi i \boldsymbol{\eta}\cdot\boldsymbol{p}_i } e^{-2\pi i \boldsymbol{\eta}\cdot\boldsymbol{p}_j} + e^{-2\pi i \boldsymbol{\eta}\cdot\boldsymbol{p}_i } e^{2\pi i \boldsymbol{\eta}\cdot\boldsymbol{p}_j}\right). 
\end{equation}

Substituting Equation \ref{dmoiver} into Equation \ref{eigest},
\begin{equation}\label{eigenest_cont}
    \begin{aligned}
        \rho_l(G_{\lambda}) &= \boldsymbol{v}_l^T(\lambda)G_{\lambda}\boldsymbol{v}_l(\lambda)\\
        &= \sum\limits_{i=1}^n\sum\limits_{j=1}^n\left\{v_{li}(\lambda)v_{lj}(\lambda)\sum_{\boldsymbol{\eta}\in\mathbb{Z}^m} \frac{1}{1+\lambda\|\boldsymbol{\eta}\|_{2k}^{2k}} \frac{1}{2}\left(e^{2\pi i \boldsymbol{\eta}\cdot\boldsymbol{p}_i } e^{-2\pi i \boldsymbol{\eta}\cdot\boldsymbol{p}_j} + e^{-2\pi i \boldsymbol{\eta}\cdot\boldsymbol{p}_i } e^{2\pi i \boldsymbol{\eta}\cdot\boldsymbol{p}_j}\right)\right\}\\
        %&= \frac{1}{2}\sum\limits_{i=1}^n\sum\limits_{j=1}^n\left\{v_{li}(\lambda)v_{lj}(\lambda)\sum_{\boldsymbol{\eta}\in\mathbb{Z}^m} \frac{1}{1+\lambda\|\boldsymbol{\eta}\|_{2k}^{2k}}e^{2\pi i \boldsymbol{\eta}\cdot\boldsymbol{p}_i } e^{-2\pi i \boldsymbol{\eta}\cdot\boldsymbol{p}_j} + \sum_{\boldsymbol{\eta}\in\mathbb{Z}^m} \frac{1}{1+\lambda\|\boldsymbol{\eta}\|_{2k}^{2k}}e^{-2\pi i \boldsymbol{\eta}\cdot\boldsymbol{p}_i } e^{2\pi i \boldsymbol{\eta}\cdot\boldsymbol{p}_j} \right\}\\
        &= \frac{1}{2}\sum_{\boldsymbol{\eta}\in\mathbb{Z}^m}\left\{ \frac{1}{1+\lambda\|\boldsymbol{\eta}\|_{2k}^{2k}}\sum\limits_{i=1}^n\sum\limits_{j=1}^n v_{li}(\lambda)v_{lj}(\lambda)\left\{e^{2\pi i \boldsymbol{\eta}\cdot\boldsymbol{p}_i } e^{-2\pi i \boldsymbol{\eta}\cdot\boldsymbol{p}_j} + e^{-2\pi i \boldsymbol{\eta}\cdot\boldsymbol{p}_i } e^{2\pi i \boldsymbol{\eta}\cdot\boldsymbol{p}_j}
\right\} \right\}\\
        &= \sum_{\boldsymbol{\eta}\in\mathbb{Z}^m}\left\{ \frac{1}{1+\lambda\|\boldsymbol{\eta}\|_{2k}^{2k}}\left\{\left( \sum\limits_{i=1}^n v_{li}(\lambda)e^{2\pi i \boldsymbol{\eta}\cdot\boldsymbol{p}_i } \right) \left( \sum\limits_{j=1}^n v_{lj}(\lambda)e^{-2\pi i \boldsymbol{\eta}\cdot\boldsymbol{p}_j } \right) \right\} \right\}. 
\end{aligned}
\end{equation}

Denoting $z_l(\lambda,\boldsymbol{\eta}) = \sum\limits_{i=1}^n v_{li}(\lambda) e^{2\pi i \boldsymbol{\eta}\cdot\boldsymbol{p}_i }$, we have
\begin{equation}\label{eigenest_cont2}
\begin{aligned}
    \rho_l(G_{\lambda}) &= \boldsymbol{v}_l^T(\lambda)G_{\lambda}\boldsymbol{v}_l(\lambda)\\
    &= \sum_{\boldsymbol{\eta}\in\mathbb{Z}^m}\left\{ \frac{z_l(\lambda,\boldsymbol{\eta})\bar{z_l}(\lambda,\boldsymbol{\eta})}{1+\lambda\|\boldsymbol{\eta}\|_{2k}^{2k}}\right\}\\
    &= \sum_{\boldsymbol{\eta}\in\mathbb{Z}^m}\left\{ \frac{1}{1+\lambda\|\boldsymbol{\eta}\|_{2k}^{2k}}\left|z_l(\lambda,\boldsymbol{\eta})\right|^2 \right\}.
\end{aligned}
\end{equation}

$\left|z_l(\lambda,\boldsymbol{\eta})\right|^2 \ge 0$ and $z_l(\lambda,\boldsymbol{\eta})$ cannot vanish for all $\eta\in\mathbb{Z}^m$. Hence, we conclude that
\begin{equation}\label{psd}
\begin{aligned}
   \rho_l(G_{\lambda}) &= \boldsymbol{v}_l^T(\lambda)G_{\lambda}\boldsymbol{v}_l(\lambda)\\
        &= \sum_{\boldsymbol{\eta}\in\mathbb{Z}^m}\left\{ \frac{1}{1+\lambda\|\boldsymbol{\eta}\|_{2k}^{2k}}\left|z_l(\lambda,\boldsymbol{\eta})\right|^2 \right\}\\
        &> 0.
\end{aligned}
\end{equation}
This proves that all eigenvalues of the matrix $G_{\lambda}$ are positive; hence, $G_{\lambda}$ is a positive definite matrix. Now, we need to estimate the eigenvalues asymptotically as $\lambda\to\infty$. Let us consider Equation \ref{psd} and split the sum to obtain
\begin{equation}\label{psd2}
\begin{aligned}
   \rho_l(G_{\lambda}) &= \boldsymbol{v}_l^T(\lambda)G_{\lambda}\boldsymbol{v}_l(\lambda)\\
        &= \sum_{\boldsymbol{\eta}\in\mathbb{Z}^m}\left\{ \frac{1}{1+\lambda\|\boldsymbol{\eta}\|_{2k}^{2k}}\left|z_l(\lambda,\boldsymbol{\eta})\right|^2 \right\}\\
            &= \sum_{\boldsymbol{\eta}\in\{0\}^m}\left\{ \frac{1}{1+\lambda\|\boldsymbol{\eta}\|_{2k}^{2k}}\left|z_l(\lambda,\boldsymbol{\eta})\right|^2 \right\} + \sum_{\boldsymbol{\eta}\in\mathbb{Z}^m\setminus\{0\}^m}\left\{ \frac{1}{1+\lambda\|\boldsymbol{\eta}\|_{2k}^{2k}}\left|z_l(\lambda,\boldsymbol{\eta})\right|^2 \right\}\\
            &= \left|z_l(\lambda,\boldsymbol{0})\right|^2 + \sum_{\boldsymbol{\eta}\in\mathbb{Z}^m\setminus\{0\}^m}\left\{ \frac{1}{1+\lambda\|\boldsymbol{\eta}\|_{2k}^{2k}}\left|z_l(\lambda,\boldsymbol{\eta})\right|^2 \right\}\\
        &= \left|\sum\limits_{i=1}^nv_{li}(\lambda)\right| ^2 + \sum_{\boldsymbol{\eta}\in\mathbb{Z}^m\setminus\{0\}^m}\left\{ \frac{1}{1+\lambda\|\boldsymbol{\eta}\|_{2k}^{2k}}\left|z_l(\lambda,\boldsymbol{\eta})\right|^2 \right\}.   
\end{aligned}
\end{equation}
\begin{equation}\label{seq_1}
 \left|z_l(\lambda,\boldsymbol{\eta})\right|^2 \ge 0
 \end{equation}\begin{equation}\label{proof_lambda}
\begin{aligned}
    z_l(\lambda,\boldsymbol{\eta}) &= \sum\limits_{i=1}^n v_{li}(\lambda) e^{2\pi i \boldsymbol{\eta}\cdot\boldsymbol{p}_i }\\
    \implies \left|z_l(\lambda,\boldsymbol{\eta})\right|^2 &=  \left|\sum\limits_{i=1}^n v_{li}(\lambda) e^{2\pi i \boldsymbol{\eta}\cdot\boldsymbol{p}_i } \right|^2\\
    &\le\sum\limits_{i=1}^n \left| v_{li}(\lambda)\right|^2            \sum\limits_{i=1}^n \left|e^{2\pi i \boldsymbol{\eta}\cdot\boldsymbol{p}_i }  \right|^2\\ 
    &\le n\|\boldsymbol{v}_l\|_2^2\\
    &=n.
\end{aligned}
\end{equation}
\begin{equation}\label{seq_3}
\forall \lambda>0, \sup\limits_{\eta\in\mathbb{Z}^m\setminus\{0\}} \left| z_l(\lambda,\boldsymbol{\eta})  \right|^2 > 0.
\end{equation}
Thus, using Equations \ref{seq_1}, \ref{proof_lambda}, \ref{seq_3} \begin{equation}\label{eq_label}
\mbox{as   } \lambda\to\infty, \sup\limits_{\eta\in\mathbb{Z}^m\setminus\{0\}} \left| z_l(\lambda,\boldsymbol{\eta})  \right|^2 = \Theta(1)
\end{equation}
\begin{equation}\label{eq_proof_eig}
\begin{aligned}
\sum_{\boldsymbol{\eta}\in\mathbb{Z}^m\setminus\{0\}^m}\left\{ \frac{\lambda}{1+\lambda\|\boldsymbol{\eta}\|_{2k}^{2k}}\left|z_l(\lambda,\boldsymbol{\eta})\right|^2 \right\} &\le \lambda \left(\sup\limits_{\eta\in\mathbb{Z}^m\setminus\{0\}} \left| z_l(\lambda,\boldsymbol{\eta})  \right|^2\right) \sum_{\boldsymbol{\eta}\in\mathbb{Z}^m\setminus\{0\}^m}\frac{1}{1+\lambda\|\boldsymbol{\eta}\|_{2k}^{2k}} \\
&\le \left(\sup\limits_{\eta\in\mathbb{Z}^m\setminus\{0\}} \left| z_l(\lambda,\boldsymbol{\eta})  \right|^2\right) \sum_{\boldsymbol{\eta}\in\mathbb{Z}^m\setminus\{0\}^m}\frac{1}{\|\boldsymbol{\eta}\|_{2k}^{2k}}, \\
&\le K \sup\limits_{\eta\in\mathbb{Z}^m\setminus\{0\}} \left| z_l(\lambda,\boldsymbol{\eta})  \right|^2,
\end{aligned}
\end{equation} where $K = \sum\limits_{\boldsymbol{\eta}\in\mathbb{Z}^m\setminus\{0\}^m}\frac{1}{\|\boldsymbol{\eta}\|_{2k}^{2k}} \mbox{            }$ is a finite positive constant as $2k>1$ as it was already assumed in the beginning that $k>\frac{m}{2}$.

Using Equations \ref{eq_label} and \ref{eq_proof_eig},
\begin{equation}\label{forward_eq}
\lambda\sum_{\boldsymbol{\eta}\in\mathbb{Z}^m\setminus\{0\}^m}\left\{ \frac{1}{1+\lambda\|\boldsymbol{\eta}\|_{2k}^{2k}}\left|z_l(\lambda,\boldsymbol{\eta})\right|^2 \right\} < K_1 \mbox{   (a positive constant)}.
\end{equation}

As $\left|z_l(\lambda,\boldsymbol{\eta})\right|^2$ is always positive as long as $\lambda>0$ and does not vanish for all $\boldsymbol{\eta}\in\mathbb{Z}^m\setminus\{0\}$, there exists an $\boldsymbol{\eta}_0$, for which it is positive. Let \begin{equation}\label{delta}
\left|z_l(\lambda,\boldsymbol{\eta}_0)\right|^2 = \delta_0 > 0.
\end{equation}
Hence, \begin{equation}\label{reverse}
\begin{aligned}
\lambda\sum_{\boldsymbol{\eta}\in\mathbb{Z}^m\setminus\{0\}^m}\left\{ \frac{1}{1+\lambda\|\boldsymbol{\eta}\|_{2k}^{2k}}\left|z_l(\lambda,\boldsymbol{\eta})\right|^2 \right\} &\ge  \frac{\lambda\delta_0}{1+\lambda\|\boldsymbol{\eta}_0\|_{2k}^{2k}}\\
&> \frac{\delta_0}{1+\|\boldsymbol{\eta}_0\|_{2k}^{2k}}\mbox{               ( as  }\lambda >1 \mbox{ as stated in the theorem)       }, \\
&= K_0 > 0
\end{aligned}
\end{equation}
Combining Equations \ref{forward_eq} and \ref{reverse},
%we already know that $\sum_{\boldsymbol{\eta}\in\mathbb{Z}^m\setminus\{0\}^m}\left\{ \frac{1}{1+\lambda\|\boldsymbol{\eta}\|_{2k}^{2k}}\left|z_l(\lambda,\boldsymbol{\eta})\right|^2 \right\}$ is always positive; hence, there exists a $K_0>0$ such that $$\sum_{\boldsymbol{\eta}\in\mathbb{Z}^m\setminus\{0\}^m}\left\{ \frac{1} {1+\lambda\|\boldsymbol{\eta}\|_{2k}^{2k}}\left|z_l(\lambda,\boldsymbol{\eta})\right|^2 \right\}>K_0$$
\begin{equation}
\label{pp}
\begin{aligned}
\frac{K_0}{\lambda} < \sum_{\boldsymbol{\eta}\in\mathbb{Z}^m\setminus\{0\}^m}\left\{ \frac{1}{1+\lambda\|\boldsymbol{\eta}\|_{2k}^{2k}}\left|z_l(\lambda,\boldsymbol{\eta})\right|^2 \right\} &< \frac{K_1}{\lambda} 
\end{aligned}
\end{equation}
Thus, $$\rho_l(G_{\lambda}) = \left|\sum\limits_{i=1}^nv_{li}(\lambda)\right|^2 +\sum_{\boldsymbol{\eta}\in\mathbb{Z}^m\setminus\{0\}^m}\left\{ \frac{1}{1+\lambda\|\boldsymbol{\eta}\|_{2k}^{2k}}\left|z_l(\lambda,\boldsymbol{\eta})\right|^2 \right\} = \left|\sum\limits_{i=1}^nv_{li}(\lambda)\right|^2 +\Theta(\frac{1}{\lambda}).$$ 

Let \begin{equation}
    \begin{aligned}
        h(\lambda) &= \left|\sum\limits_{i=1}^nv_{li}(\lambda)\right|^2\\\mbox{  and }\\
        u(\lambda) &= \sum_{\boldsymbol{\eta}\in\mathbb{Z}^m\setminus\{0\}^m}\left\{ \frac{1}{1+\lambda\|\boldsymbol{\eta}\|_{2k}^{2k}}\left|z_l(\lambda,\boldsymbol{\eta})\right|^2 \right\}.
        \end{aligned}
\end{equation}
Hence, \begin{equation}\label{sum}
   \rho_l(G_{\lambda}) = h(\lambda) + u(\lambda).
\end{equation}

As both $u(\lambda)$ and $h(\lambda)$ are positive and that for $l>1$ $\rho_l(G_{\lambda}) \to 0$ as $1/\lambda\to 0$, we have \begin{equation}\label{gotozero}
    \begin{aligned}
        \lim\limits_{\frac{1}{\lambda}\to0}h(\lambda) &= 0\\  \mbox{  and  }\\
        \lim\limits_{\frac{1}{\lambda}\to0}u(\lambda) &= 0.
    \end{aligned}
\end{equation}

Let $h(\lambda) = h_a(\lambda) + h_n(\lambda)$, where $h_a(\lambda)$ is an analytic function of $\frac{1}{\lambda}$ in any neighborhood of $1/\lambda =\epsilon = 0$ and $h_n(\lambda)$ is not an analytic function of $\frac{1}{\lambda}$ in any neighborhood of $1/\lambda =\epsilon = 0$. Similarly, let $u(\lambda) = u_a(\lambda) + u_n(\lambda)$. From Equation \ref{pp}, $u(\lambda) = \Theta(1/\lambda)$. Hence, $u_a(\lambda)$ is of the form $u_a(\lambda) = \frac{w_1}{\lambda}+\frac{w_2}{\lambda^2}+\cdots$, where $w_1>0$ and $u_n(\lambda) = o(1/\lambda)$. As $\rho_l(G_{\lambda})$ is an analytic function of $1/\lambda$, by using Equation \ref{sum}, we have $h_n(\lambda) = - u_n(\lambda)$, and there is $\left|h_n(\lambda)\right| = o(1/\lambda)$. Using this and Equation \ref{gotozero}, we can say that $h_a(\lambda)$ is of the form $\frac{x_1}{\lambda}+\frac{x_2}{\lambda^2}+\cdots$. Thus, as $h(\lambda)$ is always positive and $|h_n(\lambda)| = o(1/\lambda)$, we have that $\lambda h_a(\lambda)$ is positive for sufficiently large $\lambda$. Thus, $\lambda(\frac{x_1}{\lambda}+\frac{x_2}{\lambda^2}+\ldots)$ is positive for sufficiently large $\lambda$, meaning $x_1\ge 0$. Thus, using Equation \ref{sum} and that $h_n(\lambda) = -u_n(\lambda)$, we have \begin{equation}\label{new_sum}
    \rho_l(G_{\lambda}) = \left(\frac{x_1}{\lambda}+\frac{x_2}{\lambda^2}+\cdots\right) + \left(\frac{w_1}{\lambda}+\frac{w_2}{\lambda^2}+\cdots\right).
\end{equation} As we have already shown that $w_1>0$ and $x_1\ge 0$, comparing Equations \ref{new_sum} and \ref{eig_G}, for $l>1$, as $a_{l0} = 0$, we have $a_{l1} = w_1+x_1 > 0$. Hence, it is proved.

%As $\|\boldsymbol{v}_l(\lambda)\|_2 = 1$, we have $\sum\limits_{i=1}^nv_{li}(\lambda) = O(1) $ Hence for $\lambda>1$,  $a_{l1} > 0 \mbox{,               } l= 1,2,3\ldots n$

\end{proof}

From linear algebra, we have the following results
\begin{equation}
    \rho_l(\frac{G_{\lambda}}{n}) = \frac{1}{n}\rho_l(G_{\lambda})
\end{equation}
Let $M_{\lambda} = \frac{G_{\lambda}}{n} + \frac{I_n}{\lambda^2}$.
Hence,
\begin{equation}\label{eig_rel0}
\rho_l(M_{\lambda}) = \frac{1}{n}\rho_l(G_{\lambda}) + \frac{1}{\lambda^2}\\
 \end{equation}
 
 Due to Equations \ref{eig_G} and \ref{eig_rel0} we have
 \begin{equation}\label{eig_rel}
  \rho_l(M_{\lambda}) = b_{l0} + \sum\limits_{r=1}^{\infty}\frac{b_{lr}}{\lambda^r}. 
 %\end{aligned}
\end{equation}
where,
\begin{equation}
\begin{aligned}
b_{l0} &= \frac{a_{l0}}{n}\\
b_{l1} &= \frac{a_{l1}}{n}\\
b_{l2} &=  \frac{a_{l2}}{n} + 1\\
b_{lr} &= \frac{a_{lr}}{n},\mbox{     } r=3,4,5\ldots.
\end{aligned}
\end{equation}
We now derive a power series expansion for the eigenvalues of matrix $M_{\lambda}^{-1}$. Let $\epsilon = \frac{1}{\lambda}$. As the matrix $G_{\lambda}$ is an analytic perturbation of the all ones matrix $S_0$ with the perturbation parameter $\epsilon$, we have \begin{equation}\label{egg1}
\begin{aligned}
\lim\limits_{\epsilon = \frac{1}{\lambda}\to 0}\rho_1(G_{\lambda})  &= a_{10} = 1\\
\lim\limits_{\epsilon = \frac{1}{\lambda}\to 0}\rho_l(G_{\lambda})  &= a_{l0} = 0 \mbox{,     } l = 2,3,4\ldots.\\
\end{aligned}
\end{equation}

From linear algebra, we know that \begin{equation}\label{eiginv}
\rho_l(M_{\lambda}^{-1}) = \frac{1}{\rho_l(M_{\lambda})}
\end{equation}
Case 1: The first eigenvalue is $l=1$. In this case, $b_{10} = \frac{a_{10}}{n} = 1\ne 0$, 
which implies that the first eigenvalue of the inverse matrix $M_{\lambda}^{-1}$ is also analytic. Let it be of the form \begin{equation}\label{eigf}
\rho_1(M_{\lambda}^{-1}) = d_{10} + \sum\limits_{r=1}^{\infty}\frac{d_{1r}}{\lambda^r}.
\end{equation}
              From equations \ref{eig_rel}, \ref{eiginv}, and \ref{eigf}, we can compute the coefficients $d_{1r}$ in terms of $b_{1r}$. We compute only the first two coefficients, as they are the only ones that are of interest to us.

Applying both the series in Equations \ref{eig_rel} and \ref{eigf} are the Taylor series of $\rho_1(M_{\lambda})$ and its reciprocal $\rho_1(M_{\lambda}^{-1}) = \frac{1}{\rho_1(M_{\lambda})}$, we obtain

\begin{equation}\label{first_eig}
\begin{aligned}
d_{10} &= \frac{1}{b_{10}} = \frac{n}{a_{10}} = 1\\
d_{11} &= \frac{b{11}}{b_{10}^2} = n^2\frac{a_{11}}{na_{10}^2} = \frac{na_{11}}{a_{10}^2} = \frac{a_{11}}{n}.
\end{aligned}
\end{equation}
Case 2: $l>1$, that is, eigenvalues other than the first. In this case, $b_{l0} = \frac{a_{l0}}{n} = 0$, 
\begin{equation}\label{eigf2}
\begin{aligned}
\rho_1(M_{\lambda}^{-1}) = d_{l,-1}\lambda + d_{10} + \sum\limits_{r=1}^{\infty}\frac{d_{1r}}{\lambda^r}.
\end{aligned}
\end{equation}
Computing the first two coefficients, we obtain \begin{equation}\label{eig_second}
\begin{aligned}
d_{l,-1} &= \frac{1}{b_{l1}} = \frac{n}{a_{l1}}\\
d_{l0}   &= \frac{b_{l2}}{b_{l1}^2} = \frac{\frac{a_{l2}}{n}+1}{(\frac{a_{l1}}{n})^2} = \frac{n(a_{l2}+n)}{a_{l1}^2}. 
\end{aligned}
\end{equation}
Note that $d_{l0}$ and $d_{l1}$ are finite, as $a_{l1}>0$ because of Theorem \ref{eig_theorem}. Let $D(\lambda)$ be the diagonal matrix with diagonal entries being the eigenvalues of matrix $M_{\lambda}^{-1}$. Therefore, \begin{equation}\label{diag_eig}
D(\lambda) = \diag(\rho_1(M_{\lambda}^{-1}),\rho_2(M_{\lambda}^{-1}),\rho_3(M_{\lambda}^{-1})\cdots,\rho_n(M_{\lambda}^{-1})).
\end{equation}
Let \begin{equation}\label{d_entries}
\begin{aligned}
D_{-1} &= \diag(0,d_{2,-1},d_{3,-1},d_{4,-1},\cdots d_{n,-1}) = \diag(0,\frac{n}{a_{2,1}},\frac{n}{a_{3,1}},\cdots \frac{n}{a_{n,1}})\\
D_{0} &= \diag(d_{10},d_{20},d_{30},\cdots d_{n0}) = \diag(1,\frac{n(a_{22}+n)}{a_{21}^2},\frac{n(a_{32}+n)}{a_{31}^2}\cdots \frac{n(a_{n2}+n)}{a_{n1}^2})\\
D_{1} &= \diag(d_{11},d_{21},d_{31},\cdots,d_{n1}) = \diag(\frac{a_{11}}{n},d_{21},d_{31},\cdots,d_{n1})\\
\vdots\\
D_{r} & = \diag(d_{1r},d_{2r},d_{3r},\cdots,d_{nr})\\
\vdots
\end{aligned}
\end{equation}
Hence, from Equation \ref{eigf2}, we have \begin{equation}\label{deq_ex}
D(\lambda) = \lambda D_{-1} + D_{0} + \frac{D_{1}}{\lambda} + \frac{D_2}{\lambda^2}+\ldots + \frac{D_r}{\lambda^r}+\ldots. 
\end{equation}
 
\section{Interpolation}
\label{proof_inter_sec}

In this section, we prove the interpolation property of the minimizer of the functional as the parameter $\lambda\to\infty$. 
\begin{theorem}\label{interpolation theorem}Denoting the minimizer of the functional $C_{\lambda}(f)$ over $f \in C^0(\mathbb{T}^m)\cap H^k(\mathbb{T}^m)$ as $f_{\lambda}$,
\begin{equation}\label{eq35}\lim\limits_{\lambda \to \infty} f_{\lambda}(\boldsymbol{p}_i) = q_i\end{equation} and  \begin{equation}\label{eq36}\mbox{There exists a function denoted as }f_{\infty} \in C^0(\mathbb{T}^m)\cap H^k(\mathbb{T}^m) \mbox{ such that as }\lambda\to\infty, f_{\lambda} \to f_{\infty}\end{equation} pointwise.
\end{theorem}

\begin{proof}
Let $B^r_i$ be a ball of radius $r$ around points $\boldsymbol{p}_i$, and let $B^r = \bigcup\limits_{i = 1}^nB^r_i$. Assume that $r$ is sufficiently small such that $\bigcap\limits_{i = 1}^nB^r_i$ is a null set.

Let $\mu_i$ be a bump function with support in ball $B_i^r$ and $\mu(\boldsymbol{p}_i) = q_i$. The function $\theta_n(\boldsymbol{x})$ is defined as
\begin{equation}\label{eq37}
\theta_n(\boldsymbol{x}) = \sum\limits_{i = 1}^{n} \mu_i(\boldsymbol{x})    
\end{equation}
Recall that
\begin{equation}\label{eq38}
C_{\lambda}(f) = \lambda^2\sum\limits_{i=1}^{n}(f(\boldsymbol{p}_i)-q_i)^2 + \lambda\|\nabla^kf\|_{L^2(\mathbb{T}^m)}^2 + \|f\|_{L^2(\mathbb{T}^m)}^2
\end{equation}

%Fixing $f$ and varying $\lambda$, we have 
%\begin{equation}\label{eq39}
%\begin{aligned}
%\label{idq}
%\lim\limits_{\lambda\to\infty} \frac{C_{\lambda}(f)}{\lambda^2} &= \sum\limits_{i=1}^{h}(f(\boldsymbol{p}_i)-q_i)^2 + \frac{1}{\lambda}\|\nabla^kf\|_{L^2(\mathbb{T}^m)}^2 + \frac{1}{\lambda^2}\|f\|_{L^2(\mathbb{T}^m)}^2\\
%&= \sum\limits_{i=1}^{h}(f(\boldsymbol{p}_i)-q_i)^2 + O(1/\lambda).
%\end{aligned}
%\end{equation}
Let $f_{\lambda}$ be the minimizer the of $C_{\lambda}(f)$. Then, we have
\begin{equation}\label{eq40}
    C_{\lambda}(f_{\lambda}) \le C_{\lambda}(\theta_n).
\end{equation} 
However, owing to Equation \ref{eq38},
\begin{equation}\label{eq42}
\begin{aligned}
  C_{\lambda}(\theta_n)  &= \lambda^2\sum\limits_{i=1}^{n}(\theta_n(\boldsymbol{p}_i)-q_i)^2+  \lambda\|\nabla^k \theta_n\|_{L^2(\mathbb{T}^m)}^2 + \|\theta_n\|_{L^2(\mathbb{T}^m)}^2\\\\
    &=   \lambda\|\nabla^k \theta_n\|_{L^2(\mathbb{T}^m)}^2 + \|\theta_n\|_{L^2(\mathbb{T}^m)}^2 \mbox{    as                              }( \theta_n(\boldsymbol{p}_i) = q_i )\\
  \end{aligned}
\end{equation}
Hence, using Equations \ref{eq40} and \ref{eq42} 
\begin{equation}\label{deq}
\begin{aligned}
\lambda^2\sum\limits_{i=1}^{n}(f_{\lambda}(\boldsymbol{p}_i)-q_i)^2 + \lambda\|\nabla^kf_{\lambda}\|_{L^2(\mathbb{T}^m)}^2 + \|f_{\lambda}\|_{L^2(\mathbb{T}^m)}^2 &\le \lambda\|\nabla^k \theta_n\|_{L^2(\mathbb{T}^m)}^2 + \|\theta_n\|_{L^2(\mathbb{T}^m)}^2\\
\end{aligned}
\end{equation}As all the three terms on LHS are positive and the RHS also being positive, we have the following three equations
\begin{equation}\label{eq_sqr_err}\begin{aligned}
\lambda^2\sum\limits_{i=1}^{h}(f_{\lambda}(\boldsymbol{p}_i)-q_i)^2 &\le \lambda\|\nabla^k \theta_n\|_{L^2(\mathbb{T}^m)}^2 + \|\theta_n\|_{L^2(\mathbb{T}^m)}^2\\
\implies \sum\limits_{i=1}^{n}(f_{\lambda}(\boldsymbol{p}_i)-q_i)^2 &\le \frac{1}{\lambda}\|\nabla^k \theta_n\|_{L^2(\mathbb{T}^m)}^2 + \frac{1}{\lambda^2}\|\theta_n\|_{L^2(\mathbb{T}^m)}^2\\
\implies (f_{\lambda}(\boldsymbol{p}_i)-q_i)^2 &\le \frac{1}{\lambda}\|\nabla^k \theta_n\|_{L^2(\mathbb{T}^m)}^2 + \frac{1}{\lambda^2}\|\theta_n\|_{L^2(\mathbb{T}^m)}^2 \mbox{         }i = 1,2\ldots n . \\
\implies \left|f_{\lambda}(\boldsymbol{p}_i)-q_i\right| &\le \sqrt{\frac{1}{\lambda}\|\nabla^k \theta_n\|_{L^2(\mathbb{T}^m)}^2 + \frac{1}{\lambda^2}\|\theta_n\|_{L^2(\mathbb{T}^m)}^2} \mbox{        }i = 1,2\ldots n  \\
&= O(1/\sqrt{\lambda}) \mbox{      as $\theta_n$ and $n$ are fixed and $\lambda\to\infty$}\\
\implies \lim\limits_{\lambda\to\infty}f_{\lambda}(\boldsymbol{p}_i) &= q_i \mbox{,               }i = 1,2\ldots n  .\\
\end{aligned}
\end{equation}
\begin{equation}\label{eq44}\begin{aligned}
    \lambda\|\nabla^k f_{\lambda}\|^2_{L^2(\mathbb{T}^m)} &\le \lambda\|\nabla^k \theta_n\|_{L^2(\mathbb{T}^m)}^2 + \|\theta_n\|_{L^2(\mathbb{T}^m)}^2\\
    \implies \|\nabla^k f_{\lambda}\|^2_{L^2(\mathbb{T}^m)} &\le |\nabla^k \theta_n\|_{L^2(\mathbb{T}^m)}^2 + \frac{1}{\lambda}\|\theta_n\|_{L^2(\mathbb{T}^m)}^2\\
    &= O(1)  \mbox{               as $\theta_n$ and $n$ are fixed and $\lambda\to\infty$ }.\\
    \end{aligned}
\end{equation}
\begin{equation}\label{eq45}\begin{aligned}
\| f_{\lambda}\|^2_{L^2(\mathbb{T}^m)} &\le \lambda\|\nabla^k \theta_n\|_{L^2(\mathbb{T}^m)}^2 + \|\theta_n\|_{L^2(\mathbb{T}^m)}^2\\
    \implies \lim\limits_{\lambda\to\infty}\frac{1}{\lambda^2}\|f_{\lambda}\|_{L^2(\mathbb{T}^m)} &= 0.
    \end{aligned}
\end{equation}

%Using Equations \ref{iq1}, \ref{iq2}, \ref{iq3}, it can be shown
%\ 
%that as $\lambda\to\infty$, $f_{\lambda} \to f_{\infty}$ pointwise. $f_{\infty} \in C^0(\mathbb{T}^m)\bigcap H^k(\mathbb{T}^m)$ 

 Let $v_1(\lambda),v_2(\lambda),v_3(\lambda),...v_n(\lambda)$ be the normalized eigenvectors of the matrix $G_{\lambda}$ corresponding to the eigenvalues $\rho_1(G_{\lambda}),\rho_2(G_{\lambda}),\rho_3(G_{\lambda}),\ldots\rho_n(G_{\lambda})$. From linear algebra, we know that they are also the eigenvectors of the matrix $M_{\lambda}^{-1}$. Let us form an $n\times n$ matrix, the columns of which are these eigenvectors denoted by $E(\lambda) = [v_1(\lambda),v_2(\lambda),v_3(\lambda),...v_n(\lambda)].$
 
 Using Equation \ref{diag_eig}, we have
 \begin{equation}\label{eq50}
 M_{\lambda}^{-1} = E(\lambda)D(\lambda)E(\lambda)^{-1}. 
 \end{equation}
 Therefore, by substituting this in the expression for $\boldsymbol{c}$, we obtain 
 \begin{equation}\label{eq51}
 \boldsymbol{c} = E(\lambda)D(\lambda)E(\lambda)^{-1}L.
 \end{equation}
 
 Let \begin{equation}\label{rx}
 R_{\lambda}(\boldsymbol{x}) = [g_{\lambda}(\boldsymbol{x}-p_1),g_{\lambda}(\boldsymbol{x}-p_2),g_{\lambda}(\boldsymbol{x}-p_3),...g_{\lambda}(\boldsymbol{x}-p_N)]^T,
 \end{equation}
 
 and \begin{equation}\label{srx}
 \begin{aligned}
 S_0 &= 1_{n\times 1}\\
 S_r(\boldsymbol{x}) &= [s_r(\boldsymbol{x}-p_1),s_r(\boldsymbol{x}-p_2),s_r(\boldsymbol{x}-p_3),\ldots s_r(\boldsymbol{x}-p_n)]^T, \mbox{    } r = 1,2,3,\ldots. 
 \end{aligned}
 \end{equation}
 From Equations \ref{gasymp} and \ref{srx}, we have
 \begin{equation}\label{equ}
 \begin{aligned}
 % R_{\lambda}(\boldsymbol{x}) &= S_0 + \frac{S_1(\boldsymbol{x})}{\lambda} + \frac{S_2(\boldsymbol{x})}{\lambda^2} + \ldots + \frac{S_r(\boldsymbol{x})}{\lambda^r} + \ldots\\
  R_{\lambda}(\boldsymbol{x}) & = S_0 + \sum\limits_{r=1}^{\infty}\frac{(-1)^{r+1}}{\lambda^r}S_r(\boldsymbol{x})\\
                               %&=  T_0 + \sum\limits_{r=1}^{\infty}\frac{(-1)^{r+1}}{\lambda^r}S_r(\boldsymbol{x})\\
                               %&\mbox{          (as         } S_0 = T_0 = 1_{n\times n} \mbox{  )}.
  \end{aligned}
 \end{equation}
 
 Thus, \begin{equation}\label{eq52}f_{\lambda}(\boldsymbol{x}) = c(\lambda)^TR_{\lambda}(\boldsymbol{x})\end{equation}
 \begin{equation}\label{eq55}f_{\lambda}(\boldsymbol{x}) = \left\{E(\lambda)D(\lambda)E(\lambda)^{-1}L \right\}^T R_{\lambda}(\boldsymbol{x}) \end{equation}
Assuming $\lambda>1$(as required by Theorem \ref{eig_theorem}) and using Equations \ref{deq_ex} and \ref{equ},
\begin{equation}\label{eq55_ex}f_{\lambda}(\boldsymbol{x}) = \left\{E(\lambda)\left(\sum_{i=-1}^{\infty}D_{i}\lambda^{-i}\right)E(\lambda)^{-1}L \right\}^T \left( S_0 + \sum\limits_{r=1}^{\infty}\frac{(-1)^{r+1}}{\lambda^r}S_r(\boldsymbol{x})  \right). \end{equation}
 
 \begin{equation}\label{eq56}f_{\lambda}(\boldsymbol{x}) = \left\{L^T  (E(\lambda)^{-1})^T\left(\sum_{i=-1}^{\infty}D_{i}^T\lambda^{-i}\right)E(\lambda)^{T} \right\} \left( S_0 + \sum\limits_{r=1}^{\infty}\frac{(-1)^{r+1}}{\lambda^r}S_r(\boldsymbol{x})  \right). \end{equation}
 
 \begin{equation}\label{eq57}f_{\lambda}(\boldsymbol{x}) = \left\{L^T  E(\lambda)\left(\sum_{i=-1}^{\infty}D_{i}^T\lambda^{-i}\right)E(\lambda)^{T} \right\} \left( S_0 + \sum\limits_{r=1}^{\infty}\frac{(-1)^{r+1}}{\lambda^r}S_r(\boldsymbol{x}) \right). \end{equation}
 
 \begin{equation}\label{eq58}f_{\lambda}(\boldsymbol{x}) = \left\{  \sum_{i=-1}^{\infty}\left(\lambda^{-i}L^TE(\lambda)D_{i}^TE(\lambda)^{T}\right) \right\} \left(S_0 + \sum\limits_{r=1}^{\infty}\frac{(-1)^{r+1}}{\lambda^r}S_r(\boldsymbol{x})  \right). \end{equation}
 
 Let $J_i(\lambda) = E(\lambda)D_{i}^TE(\lambda)^{T}.$
 
 Thus,
 
 \begin{equation}\label{eq59}f_{\lambda}(\boldsymbol{x}) = \left\{  \sum_{i=-1}^{\infty}\left(\lambda^{-i}L^TJ_i(\lambda)\right) \right\} \left( S_0 + \sum\limits_{r=1}^{\infty}\frac{(-1)^{r+1}}{\lambda^r}S_r(\boldsymbol{x})  \right). \end{equation}

 \begin{equation}\label{eq60}
 \begin{aligned}
     f_{\lambda}(\boldsymbol{x}) &=    \lambda\left( L^TJ_{-1}(\lambda)S_0\right) + \left(L^TJ_{-1}(\lambda)S_1(\boldsymbol{x})\right)\\ 
                    &+  \sum_{j=2}^{\infty}\left(L^TJ_{-1}(\lambda)S_{j}(\boldsymbol{x})\lambda^{1-j}  \right) \\
                    &+ \left(L^TJ_{0}(\lambda)S_0\right)\\
                    &+ \sum_{j=1}^{\infty}\left(L^TJ_0(\lambda) S_{j}(\boldsymbol{x})\lambda^{-j}  \right) \\
     \end{aligned}
\end{equation}
 
Thus, for $\lambda>1$% (as required by Theorem \ref{eig_theorem})
 
\begin{align}
\label{eq61}
%\begin{align}
      f_{\lambda}(\boldsymbol{x}) = \lambda\left( L^TJ_{-1}(\lambda)S_0\right) + \left(L^TJ_{-1}(\lambda)S_1(\boldsymbol{x})\right)
      + \left(L^TJ_{0}(\lambda)S_0\right) + O(\frac{1}{\lambda}).
\end{align}

 %\end{equation}
As observed in Equation \ref{gasymp}, the matrix $G_{\lambda}$ is an analytic perturbation of the matrix $S_0 = 1_{n\times n}$ with the perturbation parameter $\epsilon = \frac{1}{\lambda}$. Hence, using Theorem 2.1 in \cite{JMLR:v18:16-140}, there exists a rotational orthogonal matrix $R$ such that

\begin{equation}
    \lim\limits_{\epsilon\to 0}E(\epsilon) = RE_{0} \mbox{ under the norm } \|.\|_{max}.
\end{equation}
Here, $E_0$ is the eigenvector matrix of the matrix $S_0$
as $\epsilon = \frac{1}{\lambda}$; we have 
\begin{equation}
    \lim\limits_{\lambda\to \infty}E(\lambda) = RE_{0} \mbox{ under the norm } \|.\|_{max}.
\end{equation}
Let $\Sigma$ be any diagonal matrix. We have \begin{equation}
\begin{aligned}
    \lim\limits_{\lambda\to\infty}E(\lambda)\Sigma E(\lambda)^{-1} &= RE_0\Sigma (RE_0)^{-1}\\
    &= RE_0\Sigma E_0^{-1}R^{-1}\\
    &= E_0\Sigma E_0^{-1}.
    \end{aligned}
\end{equation}

Therefore, \begin{equation}
\begin{aligned}
    \lim\limits_{\lambda\to\infty}J_i(\lambda) &= \lim\limits_{\lambda\to\infty}E(\lambda)D_i^TE(\lambda) \\
    &= E_0D_i^TE_0.
    \end{aligned}
\end{equation}

Hence, \begin{equation}
    \begin{aligned}
        \lim\limits_{\lambda\to\infty}f_{\lambda}(\boldsymbol{x}) &= \lim\limits_{\lambda\to\infty}\left\{ \lambda\left( L^TJ_{-1}(\lambda)S_0\right) + \left(L^TJ_{-1}(\lambda)S_1(\boldsymbol{x})\right)
      + \left(L^TJ_{0}(\lambda)S_0\right) + O(\frac{1}{\lambda}) \right\} \\
     &=  \lambda L^T\lim\limits_{\lambda\to\infty}J_{-1}(\lambda) S_0 + L^T\lim\limits_{\lambda\to\infty}J_{-1}(\lambda) S_1(\boldsymbol{x}) + L^T\lim\limits_{\lambda\to\infty}J_{0}(\lambda) S_0 + \lim\limits_{\lambda\to\infty}, O(\frac{1}{\lambda}), \\
      &= \lambda L^TE_0D_{-1}^TE_0S_0 + L^TE_0D_{-1}^TE_0S_1(\boldsymbol{x}) + L^TE_0D_0^TE_0S_0 \\
    \end{aligned}
\end{equation}
 
 Denoting $K_i = E_0D_i^TE_0 $, we have
 
 \begin{equation}\label{eq64}\lim\limits_{\lambda\to\infty}f_{\lambda}(\boldsymbol{x}) = \lambda L^TK_{-1}S_0 + L^TK_{-1}S_1(\boldsymbol{x}) + L^TK_0S_0\end{equation} The first term is independent of $x$ and grows linearly with $\lambda$. It is already known that $\lim\limits_{\lambda\to\infty}f_{\lambda}(\boldsymbol{p}_i) = q_i$; we should have \begin{equation}\label{eq65} L^TK_{-1}S_0 = 0.\end{equation}

As $ s_i \in H^k(\mathbb{T}^m)\bigcap C^0(\mathbb{T}^m)), i = 1,2,3\ldots $, we have $
 L^TK_{-1}S_1(\boldsymbol{x}) \in  H^k(\mathbb{T}^m)\bigcap C^0(\mathbb{T}^m))$. The third term, $L^TK_0S_0$, is constant.
 
 Hence, \begin{equation}\label{eq66} \lim\limits_{\lambda\to\infty}f_{\lambda}(\boldsymbol{x}) = L^TK_{-1}S_1(\boldsymbol{x}) + L^TK_0S_0\end{equation}
 
 Denoting \begin{equation}
 \label{inter}
     f_{\infty}(\boldsymbol{x}) = L^TK_{-1}S_1(\boldsymbol{x}) + L^TK_0S_0
 \end{equation}
Hence, as the parameter $\lambda\to \infty$, the minimizer $f_{\lambda}$ converges pointwise to the function $f_{\infty} \in  H^k(\mathbb{T}^m)\bigcap C^0(\mathbb{T}^m)$ and $f_{\infty}$ interpolates the data $(\boldsymbol{p}_i,q_i)$
  \end{proof}
  
\section{Approximate Interpolation}
\label{approx_inter_sec}

If we observe the final expression for $f_{\infty}$ as in Equation \ref{inter}, we note that there is no closed-form expression or any directly evident methods to compute the coefficient matrices, and it does not have a closed-form expression either. Therefore, this final expression is not useful for computation, and it only serves as proof that there exists an interpolant $f_{\infty} \in S$. However, if $\lambda$ is finite, then we have an expression for $f_{\lambda}$ given in Equation \ref{exp}. All we need to compute the coefficients vector $c$, which is given in Equation \ref{eq28}. However, perfect interpolation of data is not achieved when $\lambda$ is finite. Equation \ref{eq_sqr_err} gives an estimate on the interpolation error at the data points in terms of the interpolation parameter $\lambda$. Stating this again 
\begin{equation}
    \label{int_error_est}
    \begin{aligned}
    \left|f_{\lambda}(\boldsymbol{p}_i)-q_i\right| &\le \sqrt{\frac{1}{\lambda}\|\nabla^k \theta_n\|_{L^2(\mathbb{T}^m)}^2 + \frac{1}{\lambda^2}\|\theta_n\|_{L^2(\mathbb{T}^m)}^2}  \mbox{             }i = 1,2\ldots n  \\
    &= O(1/\sqrt{\lambda}) \mbox{     as we keep the data and $n$ fixed and vary $\lambda\to\infty$}
    \end{aligned}
\end{equation}

If we keep the data fixed and there by the number of data points also fixed, and vary $\lambda$, the interpolation error on the data points stays within $O(1/\sqrt{\lambda})$. Therefore as $\lambda$ increases, the interpolation error on the data points goes to zero. But there is a practical computational problem that comes into play when $\lambda$ is increased beyond a certain point and it is described in Section \ref{condition_number_sec}. The only step in computing $f_{\lambda}$ is computing $c$ using Equation \ref{eq28}, which involves inverting the matrix $M = \frac{G_{\lambda}}{n} + \frac{I}{\lambda^2}$, which we call the interpolation matrix. In the next section, we analyze the difficulty of inverting the interpolation matrix $M$.  

\section{Approximating a Sobolev Function}
\label{approx_proof_sec}

In Section \ref{proof_inter_sec}, we have proven that, as $\lambda\to\infty$, the minimizer $f_{\lambda}$ converges point-wise to a function $f_{\infty}\in C^0(\mathbb{T}^m)\cap H^k(\mathbb{T}^m)$ and that the function $f_{\infty}$ interpolates the data points perfectly. In Section \ref{approx_inter_sec}, we mentioned that we do not have any direct methods to compute $f_{\infty}$. Therefore, we proposed an approximate interpolation method where we choose $f_{\lambda}$, the minimizer corresponding to a finite $\lambda$, as the interpolant and shown that it approximately interpolates the data; in other words, the interpolation error on the data points is within $O(1/\sqrt{\lambda})$. This implies the following: $$|f_{\lambda}(\boldsymbol{p}_i)-\psi(\boldsymbol{p}_i)| = \epsilon(\lambda) = O(\frac{1}{\sqrt{\lambda}})\mbox{   i = 1,2,3,\ldots n}.$$ 

In this section, we prove that this type of approximate scattered data interpolation method has the approximation property for Sobolev functions of the type $C^0(\Omega )\cap H^k(\Omega)$, where $\Omega$ is a bounded Lipschitz domain. \\ %let $f^n_{\lambda}$ be the minimizer of the functional associated with this data, as described in Section \ref{mnmz}.
\begin{definition}\label{theorem_defs}
Let $\Omega \subset (0,1)^m$ be a closed, bounded, Lipschitz domain and $\psi:\Omega\to\mathbb{R}$ be a function in $C^0(\Omega)\cap H^k(\Omega) $. $\psi$ is a function that requires approximation. Let $D$ be a countable dense subset of $\Omega$. The scattered data constitute a set of $n$ distinct points $\{\boldsymbol{p}_i/\boldsymbol{p}_i\in D, i=1,2,\ldots n\}$ chosen from $D$ with no assumptions on their geometry and the corresponding values of $\psi$ evaluated at those points $\psi(\boldsymbol{p}_i)$. Lets define a sequence of sets $E_1,E_2,E_3\ldots$ where $E_n = \{\boldsymbol{p}_i/\boldsymbol{p}_i\in D, i = 1,2,3,..n\}.$ We define the functional in Equation \ref{eq7} of Section \ref{mnmz} using data points in the set $E_n$ and add the tag $n$ to all the notations associated with this functional, as we vary the number of data points $n$ in our analysis in this section. Therefore, the functional in Section \ref{proof_inter_sec} is denoted as $C^n_{\lambda}(f)$, the minimizer $f_{\lambda}$ in Section \ref{proof_inter_sec} is denoted as $f^n_{\lambda}$ and the matrix $G_{\lambda}$ as $G^n_{\lambda}$ and the coefficients vector $\boldsymbol{c}=[c_1,c_2,\ldots c_n]^T$ as $\boldsymbol{c}^n=[c^n_1,c^n_2,\ldots c^n_n]^T$. Note that the suffix $n$ is not a power, but only a notation that the parameter is associated with the functional defined over the set of data points $E_n$.\\
Define the mesh norm of the data points set $E_n$ over the domain $\Omega$ as \begin{equation}\label{mesh_norm}
\zeta_n = \sup\limits_{\boldsymbol{x}\in\Omega}\inf\limits_{\boldsymbol{p}\in E_n}\|\boldsymbol{x}-\boldsymbol{p}\|_2
\end{equation}
Finally, let $L\psi$ be the Sobolev extension of $\psi:\Omega\to\mathbb{R}$ to the Torus $\mathbb{T}^m$. Which means $L\psi:\mathbb{T}^m\to\mathbb{R}$, $L\psi\in C^0(\mathbb{T}^m)\cap H^k(\mathbb{T}^m)$ and $L\psi(\boldsymbol{x}) = \psi(\boldsymbol{x})\forall \boldsymbol{x} \in \Omega$. The existence of the function $L\psi$ is made possible due to the Sobolev extension theorem \cite{evans1998partial,MR2424078}.
\end{definition}
\begin{theorem}\label{Sobolev_approximation_estimate}
With definitions and notations as described in \ref{theorem_defs}, there exist constants $K_0$ and $K_1$ which are independent of the function $\psi$, such that for sufficiently large $n$
%\begin{equation}
   %\mbox{ as }\lambda\to\infty \mbox{   }
\begin{equation}\label{final_theorem_Sobolev}
\| f^n_{\lambda}(\boldsymbol{x}) - \psi(\boldsymbol{x})\|_{L^{\infty}(\Omega)} \le  K_0\zeta_n^{\alpha}\lambda\epsilon(\lambda) + K_1\sqrt{\epsilon(\lambda)} + \epsilon(\lambda).\\
\end{equation}
Where $$\epsilon(\lambda) = \frac{1}{\lambda}\|\nabla^kL\psi\|_{L^2(\mathbb{T}^m)}^2 + \frac{1}{\lambda^2}\|L\psi\|_{L^2(\mathbb{T}^m)}^2$$
   %   \limsup\limits_{n\to\infty}\|f^n_{\lambda}-\psi\|_{L^{\infty}(\Omega)} = \nu(\lambda) = O(1/\sqrt{\lambda})
%\end{equation}
\end{theorem}

\begin{proof}

As $f^n_{\lambda}$ is the minimizer of the functional $C^n_{\lambda}(f)$ in $C^0(\mathbb{T}^m)\cap H^k((\mathbb{T}^m))$, we have 
\begin{equation}\label{eqap2}
\begin{aligned}
C^n_{\lambda}(f^n_{\lambda}) &\le C^n_{\lambda}(\psi) \\
\implies  \frac{\lambda^2}{n}\sum\limits_{i=1}^{n}(f^n_{\lambda}(\boldsymbol{p}_i)-\psi(\boldsymbol{p}_i))^2 + \lambda\|\nabla^kf^n_{\lambda}\|_{L^2(\mathbb{T}^m)}^2 + \|f^n_{\lambda}\|_{L^2(\mathbb{T}^m)}^2 &\le \\ \frac{\lambda^2}{n}\sum\limits_{i=1}^{n}(L\psi(\boldsymbol{p}_i)-\psi(\boldsymbol{p}_i))^2 + \lambda\|\nabla^kL\psi\|_{L^2(\mathbb{T}^m)}^2 + \|L\psi\|_{L^2(\mathbb{T}^m)}^2
\\
\implies \frac{\lambda^2}{n}\sum\limits_{i=1}^{n}(f^n_{\lambda}(\boldsymbol{p}_i)-\psi(\boldsymbol{p}_i))^2 + \lambda\|\nabla^kf^n_{\lambda}\|_{L^2(\mathbb{T}^m)}^2 + \|f^n_{\lambda}\|_{L^2(\mathbb{T}^m)}^2  &\le \lambda\|\nabla^kL\psi\|_{L^2(\mathbb{T}^m)}^2 + \|L\psi\|_{L^2(\mathbb{T}^m)}^2
\end{aligned}
\end{equation}

Since all terms in the LHS of the above inequality are positive we have   \begin{equation}
\label{eq_rate_inter}
\begin{aligned}
\left( \frac{\lambda^2}{n}\sum\limits_{i=1}^{n}(f^n_{\lambda}(\boldsymbol{p}_i)-\psi(\boldsymbol{p}_i))^2 \right) &\le \lambda\|\nabla^kL\psi\|_{L^2(\mathbb{T}^m)}^2 + \|L\psi\|_{L^2(\mathbb{T}^m)}^2\\
\left( \lambda\|\nabla^kf^n_{\lambda}\|_{L^2(\mathbb{T}^m)}^2 \right) &\le \lambda\|\nabla^kL\psi\|_{L^2(\mathbb{T}^m)}^2 + \|L\psi\|_{L^2(\mathbb{T}^m)}^2\\
\left( \|f^n_{\lambda}\|_{L^2(\mathbb{T}^m)}^2 \right) &\le \lambda\|\nabla^kL\psi\|_{L^2(\mathbb{T}^m)}^2 + \|L\psi\|_{L^2(\mathbb{T}^m)}^2.\\
\end{aligned}
\end{equation}
which implies \begin{equation}\label{eq_rate}
\begin{aligned}
\left( \frac{1}{n}\sum\limits_{i=1}^{n}(f^n_{\lambda}(\boldsymbol{p}_i)-\psi(\boldsymbol{p}_i))^2 \right) &\le \frac{1}{\lambda}\|\nabla^kL\psi\|_{L^2(\mathbb{T}^m)}^2 + \frac{1}{\lambda^2}\|L\psi\|_{L^2(\mathbb{T}^m)}^2 \\
\end{aligned}
\end{equation}and
\begin{equation}\label{eq_rate2}
\begin{aligned}
\left( \|\nabla^kf^n_{\lambda}\|_{L^2(\mathbb{T}^m)}^2 \right) &\le \|\nabla^kL\psi\|_{L^2(\mathbb{T}^m)}^2 +\frac{1}{\lambda} \|L\psi\|_{L^2(\mathbb{T}^m)}^2\\
\left( \|f^n_{\lambda}\|_{L^2(\mathbb{T}^m)}^2 \right) &\le \lambda\|\nabla^kL\psi\|_{L^2(\mathbb{T}^m)}^2 + \|L\psi\|_{L^2(\mathbb{T}^m)}^2.\\
\end{aligned}
\end{equation}
Using Morrey's inequality and Equation \ref{eq_rate2} , there exists a $Z\in\mathbb{R}^+$ such that, for all $n\in\mathbb{N},x\in\Omega \mbox{ and }\lambda\in\mathbb{R}^+$ we have \begin{equation}\label{eq_rate3}\begin{aligned}
\|f^n_{\lambda}\|_{L^{\infty}(\Omega)} &\le K_2\|\nabla^kf^n_{\lambda}\|_{L^2(\mathbb{T}^m)}^2  \\
\implies \|f^n_{\lambda}\|_{L^{\infty}(\Omega)} &\le Z\left(\|\nabla^kL\psi\|_{L^2(\mathbb{T}^m)}^2 +\frac{1}{\lambda} \|L\psi\|_{L^2(\mathbb{T}^m)}^2\right)\\
\implies f^n_{\lambda}(\boldsymbol{x}) &\le Z\left(\|\nabla^kL\psi\|_{L^2(\mathbb{T}^m)}^2 +\frac{1}{\lambda} \|L\psi\|_{L^2(\mathbb{T}^m)}^2\right) \forall i\in\{1,2,\ldots n\}\\
\end{aligned}
\end{equation}
Using Equations \ref{eq_rate}, \ref{eq_rate3} and the fact that $\psi$  is a bounded function, there exists a $K_1\in\mathbb{R}^+$ such that for all sufficiently large $n$
\begin{equation}\label{eq_rate4}\begin{aligned}
\left(f^n_{\lambda}(\boldsymbol{p}_i) -L\psi(\boldsymbol{p}_i) \right) ^2 &\le K_1\left(\frac{1}{\lambda}\|\nabla^kL\psi\|_{L^2(\mathbb{T}^m)}^2 +\frac{1}{\lambda^2} \|L\psi\|_{L^2(\mathbb{T}^m)}^2\right) \mbox{   for } i = \{1,2,\ldots n\} \\
\implies \left|f^n_{\lambda}(\boldsymbol{p}_i) -L\psi(\boldsymbol{p}_i) \right| &\le K_1\sqrt{\frac{1}{\lambda}\|\nabla^kL\psi\|_{L^2(\mathbb{T}^m)}^2 +\frac{1}{\lambda^2} \|L\psi\|_{L^2(\mathbb{T}^m)}^2} \mbox{   for } i = \{1,2,\ldots n\} .\\ 
\end{aligned}
\end{equation}
 
We now prove the approximating property for the bounded continuous functions of this approximate interpolation method.

First, from Theorem \ref{interpolation theorem} and specifically Equation \ref{eq_sqr_err}, we know that, for any $\lambda > 1$, there exists an $\epsilon(\lambda)\in\mathbb{R}^+$ such that \begin{equation}\label{first_condition}
    |f^n_{\lambda}(\boldsymbol{x})-\psi(\boldsymbol{x})|\le \epsilon(\lambda)\mbox{  } \forall \mbox{ }x\in E_n
\end{equation} and $$\epsilon(\lambda) = O(1/\sqrt{\lambda}).$$ 
%$\lambda\to\infty\mbox{,   }\epsilon(\lambda) = O(1/\lambda^2)$

For any $\boldsymbol{x}\in\Omega$, denote $h_n(\boldsymbol{x})$ as the closest point in the set $E_n$.

\begin{equation}\label{eq_denote}
\epsilon(\lambda) = \frac{1}{\lambda}\|\nabla^kL\psi\|_{L^2(\mathbb{T}^m)}^2 + \frac{1}{\lambda^2}\|L\psi\|_{L^2(\mathbb{T}^m)}^2 
\end{equation}

For any $\lambda>0$, there exists a an $N$ such that for, all $n>N$, the following statements hold: 
\begin{enumerate}
    \item {\begin{equation}\label{dense}
        \|\boldsymbol{x}-h_n(\boldsymbol{x})\|_2\le \zeta_n \mbox{   }\forall \boldsymbol{x}\in\Omega.
    \end{equation} (follows from $D$ being a countable dense subset of $\Omega$)}\\
    \item{\begin{equation}\label{cond_1}\begin{aligned}
%\frac{1}{n}\sum\limits_{i=1}^{n}(f^n_{\lambda}(h_n(\boldsymbol{x}))-L\psi(h_n(\boldsymbol{x})))^2  &\le \epsilon(\lambda) \\        
      \left|f^n_{\lambda}(h_n(\boldsymbol{x}))-L\psi(h_n(\boldsymbol{x})) \right|&\le K_1\sqrt{\epsilon(\lambda)} \mbox{            }\forall \boldsymbol{x}\in\Omega.
        \end{aligned}
    \end{equation} (follows from Equation \ref{eq_rate4})}\\
    \item{\begin{equation}\label{psi_contn}
        \left|L\psi(h_n(\boldsymbol{x})) - L\psi(\boldsymbol{x})\right| \le \epsilon(\lambda)  \mbox{   }\forall \boldsymbol{x}\in\Omega.
    \end{equation} (follows from continuity of $L\psi$)}\\
\end{enumerate}

Using Morrey's inequality \cite{evans1998partial}, we can deduce that there exists an $\alpha \in (0,1)$ and $K_0\in\mathbb{R}^+$ such that 
\begin{equation}\label{holder}
    \| f^n_{\lambda} \|_{C^{0,\alpha}(\Omega)} \le K_0\| \nabla^k f^n_{\lambda}\|_{L^2(\Omega)}
\end{equation}
Using Equations\ref{eq_rate} and \ref{eq_denote} we have \begin{equation}\label{eq_grad_denote}
\| \nabla^k f^n_{\lambda}\|_{L^2(\Omega)} \le \lambda\epsilon(\lambda)
\end{equation}
and hence
\begin{equation}\label{hcond}
    \| f^n_{\lambda} \|_{C^{0,\alpha}(\Omega)} \le K_0\lambda\epsilon(\lambda).
\end{equation}

Using the definition of Holder continuity, for any $\boldsymbol{x}\in\Omega$, 
\begin{equation}\label{hdef}
    \frac{\left| f^n_{\lambda}(\boldsymbol{x}) - f^n_{\lambda}(h_n(\boldsymbol{x}))   \right|}{\|\boldsymbol{x}-h_n(\boldsymbol{x})\|^{\alpha}_2} \le \| f^n_{\lambda} \|_{C^{0,\alpha}(\Omega)}.
\end{equation}

Using Equations \ref{dense}, \ref{hcond} and \ref{hdef}, we have $ \mbox{   }\forall \boldsymbol{x}\in\Omega$,
\begin{equation}\label{stat3}\begin{aligned}
    \left| f^n_{\lambda}(\boldsymbol{x}) - f^n_{\lambda}(h_n(\boldsymbol{x}))   \right| &\le K_0\lambda\epsilon(\lambda)\|\boldsymbol{x}-h_n(\boldsymbol{x})\|^{\alpha}_2 \\
    \implies \left| f^n_{\lambda}(\boldsymbol{x}) - f^n_{\lambda}(h_n(\boldsymbol{x}))   \right| &\le \zeta_n^{\alpha} K_0\lambda\epsilon(\lambda) \\
    \end{aligned}
\end{equation}

Adding Equations \ref{stat3} and \ref{cond_1}, we obtain, $ \mbox{   }\forall \boldsymbol{x}\in\Omega$ and for all sufficiently large $n$
\begin{equation}\label{inter_2}
    \begin{aligned}
        \left|f^n_{\lambda}(h_n(\boldsymbol{x}))-\psi(h_n(\boldsymbol{x})) \right| + \left| f^n_{\lambda}(\boldsymbol{x}) - f^n_{\lambda}(h_n(\boldsymbol{x})) \right| &\le K_1\sqrt{\epsilon(\lambda)} + \zeta_n^{\alpha}K_0\lambda\epsilon(\lambda)\\
       \implies \left|f^n_{\lambda}(h_n(\boldsymbol{x}))-\psi(h_n(\boldsymbol{x}))  +  f^n_{\lambda}(\boldsymbol{x}) - f^n_{\lambda}(h_n(\boldsymbol{x})) \right| &\le K_1\sqrt{\epsilon(\lambda)} +\zeta_n^{\alpha}K_0\lambda\epsilon(\lambda)\\
       \implies \left| f^n_{\lambda}(\boldsymbol{x}) -L\psi(h_n(\boldsymbol{x})) \right| &\le K_1\sqrt{\epsilon(\lambda)} + \zeta_n^{\alpha}K_0\lambda\epsilon(\lambda).\\
    \end{aligned}
\end{equation}

Again, adding Equations \ref{inter_2} and \ref{psi_contn}, we have $ \mbox{   }\forall \boldsymbol{x}\in\Omega$  and for all sufficiently large $n$
\begin{equation}\label{final_cont}
\begin{aligned}
    \left| f^n_{\lambda}(\boldsymbol{x}) -L\psi(h_n(\boldsymbol{x})) \right| + \left|L\psi(h_n(\boldsymbol{x})) - L\psi(\boldsymbol{x})\right| &\le  K_1\sqrt{\epsilon(\lambda)} + \zeta_n^{\alpha}K_0\lambda\epsilon(\lambda) + \epsilon(\lambda) \\
    \implies \left| f^n_{\lambda}(\boldsymbol{x}) -L\psi(h_n(\boldsymbol{x})) + L\psi(h_n(\boldsymbol{x})) - L\psi(\boldsymbol{x})\right| &\le  K_1\sqrt{\epsilon(\lambda)} +  \epsilon(\lambda) + \zeta_n^{\alpha}K_0\lambda\epsilon(\lambda)\\
    \implies \left| f^n_{\lambda}(\boldsymbol{x}) - L\psi(\boldsymbol{x})\right| &\le  K_1\sqrt{\epsilon(\lambda)} +  \epsilon(\lambda) + \zeta_n^{\alpha}K_0\lambda\epsilon(\lambda) .
\end{aligned}
\end{equation}

As Equation \ref{final_cont} holds for all $x\in\Omega$, we and for all sufficiently large $n$ can say that 
\begin{equation}\label{conclude_pre}
\| f^n_{\lambda}(\boldsymbol{x}) - L\psi(\boldsymbol{x})\|_{L^{\infty}(\Omega)} \le   K_0\zeta_n^{\alpha}\lambda\epsilon(\lambda) + K_1\sqrt{\epsilon(\lambda)} + \epsilon(\lambda).
\end{equation}

As $L\psi = \psi$ on $\Omega$, we finally have 

\begin{equation}\label{conclude}
\| f^n_{\lambda}(\boldsymbol{x}) - \psi(\boldsymbol{x})\|_{L^{\infty}(\Omega)} \le   K_0\zeta_n^{\alpha}\lambda\epsilon(\lambda) + K_1\sqrt{\epsilon(\lambda)} + \epsilon(\lambda).
\end{equation}

\end{proof}

The approximation property of the approximate interpolation method can be expressed as follows:
\begin{remark}
\begin{equation}
    \limsup\limits_{n\to\infty}\|f^n_{\lambda}-\psi\|_{L^{\infty}(\Omega)} = O(1/\sqrt{\lambda}).
\end{equation}
\end{remark}
\begin{proof}
As the Equation \ref{conclude} holds for all sufficiently large $n$, we can say 
\begin{equation}\label{eq_limsup}\begin{aligned}
    \limsup\limits_{n\to\infty}\| f^n_{\lambda}(\boldsymbol{x}) - \psi(\boldsymbol{x})\|_{L^{\infty}(\Omega)} =  \lim\limits_{n\to\infty}\left(K_1\sqrt{\epsilon(\lambda)} + \epsilon(\lambda) + \zeta_n^{\alpha}K_0\lambda\epsilon(\lambda)\right)\\
    \implies \limsup\limits_{n\to\infty}\| f^n_{\lambda}(\boldsymbol{x}) - \psi(\boldsymbol{x})\|_{L^{\infty}(\Omega)} =  K_1\sqrt{\epsilon(\lambda)} + \epsilon(\lambda) + K_0\lambda\epsilon(\lambda)\lim\limits_{n\to\infty}\zeta_n^{\alpha}\\
    \end{aligned}
\end{equation}As the set $D$ is dense in $\Omega$ , $\lim\limits_{n\to\infty}\zeta_n^{\alpha} = 0$. So we have \begin{equation}\label{eq_final_limsup}\begin{aligned}
\limsup\limits_{n\to\infty}\| f^n_{\lambda}(\boldsymbol{x}) - \psi(\boldsymbol{x})\|_{L^{\infty}(\Omega)} &= K_1\sqrt{\epsilon(\lambda)} + \epsilon(\lambda)\\
&= O(1/\sqrt{\lambda}).
\end{aligned}
\end{equation}
\end{proof}
Note that the RHS of Equation \ref{eq_final_limsup} is independent of the number of data points $n$. Thus, by choosing $\lambda$ small enough, it is possible to recover $\psi$ to any desired accuracy, as the data points become dense($n\to\infty$). As the value of $\lambda$ increases the approximation error goes to zero. The reader may note that there is a practical difficulty of computing the approximant $f^n_{\lambda}$ as $\lambda$ increases, which is discussed in Section \ref{condition_number_sec}.

\section{Condition Number of the Interpolation Matrix}
\label{condition_number_sec}
Let $M = (\frac{G_{\lambda}}{n}+\frac{I_n}{\lambda^2})$ be called the interpolation matrix; we must invert this matrix in order to compute $c$ and there by compute the approximate interpolating function $f_{\lambda}$. First, the interpolation matrix $M$ is positive definite, as we have shown in Theorem \ref{eig_theorem}, and the matrix $G_{\lambda}$ is positive definite. Let us derive a bound on the condition number of matrix $M$. First, let the maximum and minimum eigenvalues of the matrix $\frac{G_{\lambda}}{n}$ be $\rho_{max}$ and $\rho_{min}$. Let $\kappa(M)$ denote the condition number of matrix $M$. Thus, the condition number of the matrix $M$ is given as

\begin{equation}
    \kappa(M) = \frac{\rho_{max}+\frac{1}{\lambda^2}}{\rho_{min}+\frac{1}{\lambda^2}}.
\end{equation}

As the matrix $G_{\lambda}$ is positive definite $\rho_{min} > 0$, we obtain
\begin{equation}
    \kappa(M) \le \frac{\rho_{max}+\frac{1}{\lambda^2}}{0+\frac{1}{\lambda^2}},
\end{equation};thus, 

\begin{equation}
    \kappa(A) \le \lambda^2\rho_{max}+1.
\end{equation}

However, $\rho_{max} \le \Tr[\frac{G_{\lambda}}{n}]$; thus, $\rho_{max} \le \frac{1}{n}\Tr[G_{\lambda}]$.
\begin{equation}
\label{cmd}
    \kappa(M) \le \frac{\lambda^2}{n}\Tr[G_{\lambda}]+1.
\end{equation}

We know that $\Tr[G_{\lambda}] = ng_{\lambda}(0)$. Further, \begin{equation}
    g_{\lambda}(0) =  \sum_{\boldsymbol{l}\in\mathbb{Z}^m} \frac{1}{1+\lambda\|\boldsymbol{l}\|_{2k}^{2k}}. 
\end{equation} 
Hence, $$\Tr[G_{\lambda}] = \sum_{\boldsymbol{l}\in\mathbb{Z}^m} \frac{n}{1+\lambda\|\boldsymbol{l}\|_{2k}^{2k}}.$$
Substituting this in Equation \ref{cmd}, we obtain
\begin{equation}
\label{fcmd}
    \kappa(M) \le  1 + \sum_{\boldsymbol{l}\in\mathbb{Z}^m} \frac{\lambda^2}{1+\lambda\|\boldsymbol{l}\|_{2k}^{2k}}.
\end{equation}

The bound on the condition number depends on the parameter $\lambda$ as $O(\lambda)$. Therefore, if we want a higher accuracy of interpolation, we need a higher $\lambda$ that consequently increases the bound on the condition number of the interpolation matrix, making the computations more difficult. Thus, $\lambda$ is a trade-off between the accuracy of interpolation and the ease of computation.

However, the upper bound on the condition number of the interpolation matrix is completely independent of the position of the data points or the number of data points. Therefore, when $\lambda$ is fixed, although the data points become dense in the domain $\Omega$, the condition number of the interpolation matrix remains bounded, which is in stark contrast with the radial basis function interpolation using this plate spline-type functions.

\chapter{Interpolation using Trigonometric Polynomials}\label{chap_3}

\section{Functions of Bounded Variation}

\begin{definition}
Total Variation: Given a function $f$ of the form $f:\Omega\to\mathbb{R}$, $\Omega\subset\mathbb{R}^m$, the total variation is of the function $f$ is denoted as $V_{\Omega}(f)$ and is defined as \begin{equation}
    V_{\Omega}(f) = \int_{\Omega}|Df|
\end{equation} where $Df$ is the distributional/weak derivative of the function $f$.
\end{definition}

\begin{definition}
For the context of this paper, a class of functions called functions of bounded variation, denoted as $BV(\mathbb{T}^m)$, is defined as the set of all functions of the form $f:\mathbb{T}^m\to\mathbb{R}$ which has the following properties. 
\begin{enumerate}
    \item The total variation $V_{\mathbb{T}^m}(f) = \int_{\mathbb{T}^m}|Df|$ is finite.
    \item The function $f$ does not have removable discontinuities(Note that this condition is not imposed in most of the books, but we do this specifically for the context of this paper).
\end{enumerate}

\end{definition}

\section{Fourier projection operator}

\begin{definition}
Given any two vectors $\boldsymbol{a} = (a_1,a_2\ldots a_m)\in\mathbb{R}^m$ and $\boldsymbol{b} = (b_1,b_2\ldots b_m)\in\mathbb{R}^m$, the relation $\boldsymbol{a}\le \boldsymbol{b}\implies a_i\le b_i,\mbox{ for }i = 1,2\ldots m$.
\end{definition}

\begin{definition}
Given a $\boldsymbol{\omega}\in\mathbb{W}^m$, the space of trigonometric polynomials $TP_{\boldsymbol{\omega}}$ is defined as the set of all trigonometric polynomials with degree $\boldsymbol{r}\in\mathbb{W}^m$ such that $\boldsymbol{r}\le\boldsymbol{\omega}$.
\end{definition}

\begin{definition}
Given any $\boldsymbol{\omega}\in\mathbb{W}^m$ The Fourier projection operator of the form $P_{\boldsymbol{\omega}}:BV(\mathbb{T}^m)\to TP_{\boldsymbol{\omega}}$ is defined as, for any given $f\in BV(\mathbb{T}^m)$, \begin{equation}\label{P_def}
P_{\boldsymbol{\omega}}f(\boldsymbol{x}) = \sum_{\boldsymbol{l}\in\mathbb{Z}^m \land -\boldsymbol{\omega}\le\boldsymbol{l}\le \boldsymbol{\omega} } \hat{f}_{\boldsymbol{l}} e^{2\pi i\boldsymbol{l}\cdot\boldsymbol{x}}.
\end{equation} where $\hat{f}_{\boldsymbol{l}}$, $\boldsymbol{l}\in\mathbb{Z}^m$ are the Fourier series coefficients of the function $f$.

The opertaor $\bar{P}_{\boldsymbol{\omega}}$ is defined as $\bar{P}_{\boldsymbol{\omega}}f = f-P_{\boldsymbol{\omega}}f$ 
\end{definition}
\subsection{Some properties of the projection operator}
We state some properties of the operator $P_{\boldsymbol{\omega}}$ which can be easily derived from the properties of the Fourier Series. For any $u\in BV(\mathbb{T}^m)$
\begin{enumerate}
       \item $ \|u\|_{L^2(\mathbb{T}^m)} = \|P_{\boldsymbol{\omega}}u\|_{L^2(\mathbb{T}^m)} + \|\bar{P}_{\boldsymbol{\omega}}u\|_{L^2{\mathbb{T}^m}}$.
       \item $P_{\boldsymbol{\omega}}(u_1+u_2) = P_{\boldsymbol{\omega}}u_1 + P_{\boldsymbol{\omega}}u_2$.
       \item If $u\in TP_{\boldsymbol{\omega}}$, then $P_{\boldsymbol{\omega}}u = u$.
\end{enumerate}

\section{Minimization problem}
\label{mnmz2}

We take the functional in Equation \ref{eq7} but define it over the space of trigonometric polynomials of degree less than or equal to $\boldsymbol{\omega}$ denoted as $TP_{\boldsymbol{\omega}}$. It is given below as

\begin{equation}\label{eq8t}D_{\lambda}(u) = \frac{\lambda^2}{n}\sum\limits_{i=1}^{n}(u(\boldsymbol{p}_i)-q_i)^2 + \lambda\|\nabla^ku\|_{L^2(\mathbb{T}^m)}^2 + \|u\|_{L^2(\mathbb{T}^m)}^2. \end{equation}
where $\boldsymbol{p}_i\in(0,1)^m$,$k,m \in \mathbb{N}, k>\frac{m}{2}, 
\lambda \in \mathbb{R}^+ \mbox{ and } f \in TP_{\boldsymbol{\omega}}$. 
The minimization problem is the functional $D_{\lambda}^n(u)$ is to be minimized in the space of trigonometric polynomials $u\in TP_{\boldsymbol{\omega}}$.
%\end{equation}

\begin{theorem}
The functional $D_{\lambda}^n(u)$ has a unique minimizer in the space of trigonometric polynomials $TP_{\boldsymbol{\omega}}$ 
\end{theorem}
\begin{proof}
In Appendix \ref{unique_minimizer} the functional $D_{\lambda}^n(u)$ is shown to have a unique minimizer in the space $C^0(\mathbb{T}^m)\cap H^k(\mathbb{T}^m)$. As the space $TP_{\boldsymbol{\omega}}$ is a linear open subspace of $C^0(\mathbb{T}^m)\cap H^k(\mathbb{T}^m)$, and the functional $D_{\lambda}^n(u)$ being convex, it follows that $D_{\lambda}^n(u)$ has a unique minimizer in the space $TP_{\boldsymbol{\omega}}$.
\end{proof}

%%%%%%%%%%%%%%%%%%%%%%%%%%%%%%%%%%%%%%%%%%%%%%%%%%%%%%% E-L%%%%%%%%%%%%%%%%%%%%%%%%%%%%%%%%%%%%%%%%

\subsection{Euler--Lagrange (E--L) Equation }

We now derive the Euler--Lagrange (E--L) equation of the minimization problem posed in the previous section and show that it is a linear weak PDE with some global terms.

We minimize in $TP_{\boldsymbol{\omega}}$, the functional \begin{equation}\label{eq8t2}D_{\lambda}(u) = \frac{\lambda^2}{n}\sum\limits_{i=1}^{n}(u(\boldsymbol{p}_i)-q_i)^2 + \lambda\|\nabla^ku\|_{L^2(\mathbb{T}^m)}^2 + \|u\|_{L^2(\mathbb{T}^m)}^2. \end{equation}

We derive the Euler--Lagrange equation for the above problem by steps for each term separately.
For any $\phi \in TP_{\boldsymbol{\omega}}$.

\begin{equation}\label{eq9b}
\begin{aligned}
\frac{d}{ds}|_{s=0} \|u(\boldsymbol{x})+s\phi(\boldsymbol{x})\|_{L^2(\mathbb{T}^m)}^2 &= \frac{d}{ds}|_{s=0} \int_{\mathbb{T}^m} \left|u(\boldsymbol{x})+s\phi(\boldsymbol{x})\right|^{2} \mathop{}\!\mathrm{d}^m\boldsymbol{x}\\
&\stackrel{*}{=} 2\int_{\mathbb{T}^m} \phi(\boldsymbol{x}) u(\boldsymbol{x})\mathop{}\!\mathrm{d}^m\boldsymbol{x}, 
\end{aligned}
\end{equation} where $*$ can be justified by using the dominated convergence theorem

\begin{equation}\label{eq10b}
\begin{aligned}
 \frac{d}{ds}|_{s=0} \lambda\| \nabla^k(u(\boldsymbol{x})+s\phi(\boldsymbol{x}))\|_{L^2(\mathbb{T}^m)} &= \frac{d}{ds}|_{s=0} \int_{\mathbb{T}^m} \lambda\left|\nabla^k u(\boldsymbol{x})+s \nabla^k \phi(\boldsymbol{x})\right|^2\mathop{}\!\mathrm{d}^m\boldsymbol{x}\\
&= 2\lambda\int_{\mathbb{T}^m} \nabla^k \phi(\boldsymbol{x}) \cdot \nabla^k u(\boldsymbol{x})\mathop{}\!\mathrm{d}^m\boldsymbol{x}
\end{aligned}
\end{equation}

\begin{equation}\label{eq11b}\frac{d}{ds}|_{s=0} \frac{\lambda^2}{n}\sum\limits_{i = 1}^n|u(\boldsymbol{p}_i)+s\phi(\boldsymbol{p}_i)-q_i|^2 =  -\frac{2\lambda^2}{n}\sum\limits_{i = 1}^n (q_i-u(\boldsymbol{p}_i))\phi(\boldsymbol{p}_i)
\end{equation}

and by combining all terms, we obtain the following PDE as the Euler-Lagrange equation for the minimization problem.

\begin{equation}\label{eq12b}
\begin{aligned}
%\label{E-L}
-\frac{\lambda^2}{n}\sum\limits_{i = 1}^{n}(q_i-u(\boldsymbol{p}_i))\phi(\boldsymbol{p}_i) + \lambda\int_{\mathbb{T}^m}\nabla^k\phi(\boldsymbol{x})\cdot \nabla^k u(\boldsymbol{x}) \mathop{}\!\mathrm{d}^m\boldsymbol{x}+  \int_{\mathbb{T}^m} \phi(\boldsymbol{x}) u(\boldsymbol{x})\mathop{}\!\mathrm{d}^m\boldsymbol{x}  = 0 \\
\forall \phi \in  TP_{\boldsymbol{\omega}}
\end{aligned}
\end{equation}

%%%%%%%%%%%%%%%%%%%%%%E-L end%%%%%%%%%%%%%%%%%%%%%%%%%%%%%%%%%%%%%%%%%%%%%%%%%%%%%%%%%%%%%%%%%%%%%%%%%%%%%

%%%%%%%%%%%%%%%%%%%%%%%% Expression to Minimizer %%%%%%%%%%%%%%%%%%%%%%%%%%%%%%%%%%%%%%%%%%%%%%%%%%%%%%%%

\begin{theorem}\label{exp_min_TP}
The solution to the PDE in Equation \ref{eq12b} is $u_{\lambda}$, which is given as
\begin{equation}
    \label{exp_th2}  
  u_{\lambda}(\boldsymbol{x}) = \sum\limits_{i=1}^n \frac{c_i}{n}w_{\lambda}(\boldsymbol{x}-\boldsymbol{p}_i),
 \end{equation}
 where
 \begin{equation}\label{eq22_th2}
%\label{greenf}
    w_{\lambda}(\boldsymbol{x}) = P_{\boldsymbol{\omega}}g_{\lambda}(x) = \sum_{\boldsymbol{l}\in\mathbb{Z}^m \land -\boldsymbol{\omega} \le \boldsymbol{l}\le \boldsymbol{\omega} } \frac{1}{1+\lambda\|\boldsymbol{l}\|_{2k}^{2k}} \cos{(2\pi\boldsymbol{l}\cdot\boldsymbol{x})}.
\end{equation}
$\pmb{c} = [c_1,c_2,...c_{N}]^T$ is given as \begin{equation}\label{eq28_th2}
%\label{iq2}
 \pmb{c} = (\frac{1}{n}W_{\lambda}+\frac{1}{\lambda^2}I)^{-1}L,\end{equation} where the matrix $W_{\lambda}$ is given as 
\begin{equation}\label{eq29_th2}
% \label{iqd}
  W_{\lambda} = [\gamma_{ij}(\lambda)]_{n\times n},\gamma_{ij}(\lambda) = w_{\lambda}(\boldsymbol{p}_i-\boldsymbol{p}_j)\end{equation} and $$L = [q_1,q_2,\ldots q_n]^T.$$ 
\end{theorem}
The above theorem can be proved on similar lines of the proof of Theorem \ref{E-L_solution}. However the proof of this theorem is given in Appendix \ref{A2}.
%%%%%%%%%%%%%%%%%%%%%%%%%%%%%%%%%%%%%%%%%%%%%%%%%%%%%%%%%%%%%%%%%%%%%%%%%%%%%%%%%%%%%%%%%%%%%%%%%%%%%%%%%

%\label{proof_inter_sec}
\begin{theorem}
\label{interpolation theorem 2}
Assuming the data and the number of data points $n$ are fixed, and $u_{\lambda}$ denoting the minimizer of the functional $D_{\lambda}(u)$ over the set $TP_{\boldsymbol{\omega}}$, \begin{equation}\label{eq35p}\lim\limits_{\lambda \to \infty} u_{\lambda}(\boldsymbol{p}_i) = q_i + O(\frac{1}{\|\boldsymbol{\omega}\|_2})\end{equation} and  \begin{equation}\label{eq36p}\mbox{There exists a function denoted as }u_{\infty} \in TP_{\boldsymbol{\omega}} \mbox{ such that as }\lambda\to\infty, u_{\lambda} \to u_{\infty}\end{equation} pointwise.
\end{theorem}

\begin{proof}
%Let $V(\boldsymbol{\omega}) = \mathbb{Z}^m \bigcap \bigotimes\limits_{j=1}^m(0,\omega_j)$

Consider the Dirichlet function \begin{equation}
    D_{\boldsymbol{\omega}}(\boldsymbol{x}) = \frac{1}{\omega_1\omega_2\ldots\omega_m}\sum\limits_{\boldsymbol{r}\in \mathbb{Z}^m\land 0\le \boldsymbol{r}\le \boldsymbol{\omega}} \cos{(2\pi \boldsymbol{r}\cdot \boldsymbol{x})} 
\end{equation} where $\boldsymbol{\omega} = (\omega_1,\omega_2\ldots \omega_m)$.
It has the following properties. \begin{equation}\label{Dir_1}
    D_{\boldsymbol{\omega}}(0) = 1
\end{equation} and \begin{equation}\label{Dir_2}
    D_{\boldsymbol{\omega}}(\boldsymbol{x}) = O(\frac{1}{\|\boldsymbol{\omega}\|_2}),\mbox{   }\boldsymbol{x}\ne 0
\end{equation}

Consider the function \begin{equation}\label{Dir_3}
    \Gamma_n(\boldsymbol{x}) = \sum\limits_{i=1}^n q_i D_{\boldsymbol{\omega}}(\boldsymbol{x}-\boldsymbol{p}_i)
\end{equation}
As $u_{\lambda}$ is the minimizer of the functional, we have \begin{equation}\label{Dir_4}\begin{aligned}
    A_{\lambda}(u_{\lambda}) &\le A_{\lambda}(\Gamma_n)\\
    \implies \lambda^2\sum_{i=1}^n(u_{\lambda}(\boldsymbol{p}_i)-q_i)^2 &\le A_{\lambda}(\Gamma_n)\\
    \implies \lambda^2(u_{\lambda}(\boldsymbol{p}_i)-q_i)^2 &\le A_{\lambda}(\Gamma_n),\mbox{     }i=1,2\ldots n\\
    \implies (u_{\lambda}(\boldsymbol{p}_i)-q_i)^2 &\le \sum_{i=1}^n(\Gamma_n(x)-q_i)^2 + \frac{1}{\lambda}\|\nabla^k\Gamma_n\|_{L^2(\mathbb{T}^m)} + \frac{1}{\lambda^2}\|\Gamma_n\|_{L^2(\mathbb{T}^m)},\\ &\mbox{     }i=1,2\ldots n\\
    \implies \lim\limits_{\lambda\to\infty}(u_{\lambda}(\boldsymbol{p}_i)-q_i)^2 &\le \sum_{i=1}^n(\Gamma_n(x)-q_i)^2,\mbox{     }i=1,2\ldots n\\
    \end{aligned}
\end{equation}
Using Equations \ref{Dir_1}, \ref{Dir_2}, \ref{Dir_3} and \ref{Dir_4}
\begin{equation}
\begin{aligned}
    \lim\limits_{\lambda\to\infty}(u_{\lambda}(\boldsymbol{p}_i)-q_i)^2 &\le \sum\limits_{i=1}^n(q_i-q_i+(n-1)O(\frac{1}{\|\boldsymbol{\omega}\|_2}))^2,\mbox{     }i=1,2\ldots n\\
    \implies \lim\limits_{\lambda\to\infty}(u_{\lambda}(\boldsymbol{p}_i)-q_i)^2 &\le \sum\limits_{i=1}^n((n-1)O(\frac{1}{\|\boldsymbol{\omega}\|_2}))^2,\mbox{     }i=1,2\ldots n
    \end{aligned}
\end{equation}
As the data and number of data points $n$ is constant, we have from above Equation
\begin{equation}
\begin{aligned}
 \lim\limits_{\lambda\to\infty}(u_{\lambda}(\boldsymbol{p}_i)-q_i)^2 &= O(\frac{1}{\|\boldsymbol{\omega}\|_2^2})),\mbox{     }i=1,2\ldots n\\
 \implies \lim\limits_{\lambda\to\infty}u_{\lambda}(\boldsymbol{p}_i)-q_i &= O(\frac{1}{\|\boldsymbol{\omega}\|_2})),\mbox{     }i=1,2\ldots n\\
 \implies \lim\limits_{\lambda\to\infty}u_{\lambda}(\boldsymbol{p}_i) &= q_i + O(\frac{1}{\|\boldsymbol{\omega}\|_2})),\mbox{     }i=1,2\ldots n\\
\end{aligned}    
\end{equation}
This completes the proof of the first statement of the Theorem.
\end{proof}
The second statement can easily be proved on similar lines as the proof of the second statement of Theorem \ref{interpolation theorem} using same asymptotic expansions as in Section \ref{sec_asymp_exp_interp_matrix} which can easily shown to be valid in case of expressions for interpolation matrix derived in Theorem \ref{exp_min_TP} .

\begin{remark}
If each of data points $\boldsymbol{p}_i$ coincide with a point on a uniform rectangular grid of spacing $\frac{1}{\omega_j}$ in the $j^{th}$ coordinate axis, i.e $\boldsymbol{p_i} = (\frac{n_i^1}{\omega_i},\frac{n_i^2}{\omega_2}\ldots \frac{n_i^m}{\omega_m}),\mbox{   }n_i^r\in\{0,1,2\ldots\omega_r\}, r= 1,2\ldots m \mbox{   and   } i = 1,2\ldots n$, then $u_{\infty}$ interpolates the data perfectly.
\end{remark}
\begin{proof}
If each of the data points coincides with a point on a uni-from rectangular grid, then \begin{equation}\label{grid_null}
    D_{\boldsymbol{\omega}}(\boldsymbol{p}_i) = 0, \mbox{    }i = 1,2\ldots n
\end{equation}
Using Equations \ref{Dir_3} and \ref{grid_null} we have
\begin{equation}\label{coincide_grid}
    \Gamma_n(\boldsymbol{p}_i) = q_i,\mbox{  } i = 1,2\ldots n.
\end{equation}
There by Using Equation \ref{Dir_4} and \ref{coincide_grid} we have
\begin{equation}\label{interp_grid}\begin{aligned}
    \lim\limits_{\lambda\to\infty}(u_{\lambda}(\boldsymbol{p}_i)-q_i)^2 &\le \sum_{i=1}^n(\Gamma_n(x)-q_i)^2,\mbox{     }i=1,2\ldots n\\
    &= 0,\mbox{     }i=1,2\ldots n\\
    \implies u_{\infty}(\boldsymbol{p}_i) &= q_i,\mbox{     }i=1,2\ldots n\\
    \end{aligned}
\end{equation}

%%\Gamma_n(\boldsymbol{p}_i) = q_i
\end{proof}

\chapter{Approximation of a Multivariate BV Function from its Scattered Data}\label{chap_4}

\section{Approximation of a BV Function}
In this section we show that given any function of bounded variation of the form $\psi:\mathbb{T}^m\to\mathbb{R}$, we can approximate it in the $L^2$-norm from its scattered data.

\begin{definition}\label{theorem_defs2}
Let $\psi : \mathbb{T}^m\to\mathbb{R}$ be a BV function that requires approximation. It is assumed that the total variation (in the Vitali sense) $V_{\mathbb{T}^m}(\psi)$ is finite but non zero. Let $D$ be a countable dense subset of $(0,1)^m$ excluding points of discontinuity of $\psi$. The scattered data constitute a set of $n$ distinct points $\{\boldsymbol{p}_i/\boldsymbol{p}_i\in D, i=1,2,\ldots n\}$ chosen from $D$ with no assumptions on their geometry and the corresponding values of $\psi$ evaluated at those points $\psi(\boldsymbol{p}_i)$. Lets define a sequence of sets $E_1,E_2,E_3\ldots$ where $E_n = \{\boldsymbol{p}_i/\boldsymbol{p}_i\in D, i = 1,2,3,..n\}$.We define the functional in Equation \ref{eq8t} of Section \ref{mnmz2} using data points in the set $E_n$ as which is given as
\begin{equation}\label{functional_redef}
    D_{\lambda}^n(u) = \frac{\lambda^2}{n}\sum\limits_{i=1}^{n}(u(\boldsymbol{p}_i)-\psi(\boldsymbol{p}_i))^2 + \lambda\|\nabla^k u\|_{L^2(\mathbb{T}^m)}^2 + \|u\|_{L^2(\mathbb{T}^m)}^2, \end{equation}
where $k,m \in \mathbb{N}, k>\frac{m}{2}, 
\lambda \in \mathbb{R}^+ \mbox{ and } u \in TP_{\boldsymbol{\omega}}$.\\
As we vary the number of data points $n$ and also $\boldsymbol{\omega}$ in our analysis in this section, we therefore, the functional $D_{\lambda}(u)$ is denoted as $D^n_{\lambda,\boldsymbol{\omega}}(u)$ and the minimizer $u_{\lambda}$ is denoted as $u^n_{\lambda,\boldsymbol{\omega}}$. Note that the suffix $n$ is not a power, but only a notation that the parameter is associated with the functional defined over the set of data points $E_n$.\\
%Define the mesh norm of the data points set $E_n$ over the domain $\mathbb{T}^m$ as \begin{equation}\label{mesh_norm2}
%\zeta_n = \sup\limits_{\boldsymbol{x}\in(0,1)^m}\inf\limits_{\boldsymbol{p}\in E_n}\|\boldsymbol{x}-\boldsymbol{p}\|_2
%\end{equation}
%\end{definition}

Define denote discrepancy measure $\zeta_n$ of the point set $E_n=\{\boldsymbol{p}_1,\boldsymbol{p}_2...\boldsymbol{p}_n\}$ in the domain $(0,1)^m$ as \begin{equation}\label{Disc}\zeta_n = D_n^*(\{\boldsymbol{p}_1,\boldsymbol{p}_2...\boldsymbol{p}_n\})
\end{equation} The definition of star discrepancy $D_n^*$ is assumed to be as defined in the book \cite{kuipers2012uniform}.
\end{definition}
As the points are taken from the set $D$ which is a countable dense subset of $(0,1)^m$, as $n\to\infty$, from \cite{kuipers2012uniform}, we have the asymptotic $$\zeta_n = O(\frac{(\ln {n})^m}{n}) $$

\begin{theorem}
Considering the definitions in \ref{theorem_defs2}, If we vary the parameter $\lambda$ with the number of data points $n$ as $\lambda = \zeta^{-\beta}_n,\mbox{   where }\beta>0$ and vary $\boldsymbol{\omega} = (\omega_1,\omega_2\ldots\omega_m)$ as $\omega_i = \kappa_i\zeta_n^{-\alpha},i = 1,2\ldots m$, where $\kappa_i$ are a positive constants, and additionally if the following conditions are assumed 

\begin{equation}
     \begin{aligned}
     \alpha &> 0\\
     \beta &> 0\\
     k&>\frac{m}{2}\\%&< k\\ 
      1&<\frac{\alpha}{\beta} <(2k-1)\\
     \end{aligned}
 \end{equation}
%\begin{equation}
%     \begin{aligned}
%       0 &< \alpha < \frac{1}{2k-1}\\ 
%       0 &< \beta < \frac{2k-1}{k}\\
%        k&>\max(\frac{3}{2},\frac{m}{2}) , k\in \mathbb{N}%&< k
%        \end{aligned}
% \end{equation}
 then
\begin{equation}
    \label{th_bv}
    \lim\limits_{n\to\infty}\|u_{\lambda,\boldsymbol{\omega}}^n - \psi\|_{L^2(\mathbb{T}^m)} = 0
\end{equation}
\end{theorem}

\begin{proof}
%We first show that $u_{\lambda,\boldsymbol{\omega}}^n(x)$ is bounded $\forall \boldsymbol{x}\in\mathbb{T}^m$ as $n\to\infty$.
%%%%%%%%%%%%%%%%%%%%%%%%%%%%%%%%%%%%%%%%%%%%%%%%%%%%%%%%%%%%%%%%%%%%%%%%%%%%%%%%%%%%%%%%%%%%%%%%%%%
As $u_{\lambda,\boldsymbol{\omega}}^n$ is the minimizer of the functional $D^n_{\lambda,\boldsymbol{\omega}}(f)$ in $TP_{\boldsymbol{\omega}}$, we have 
\begin{equation}\label{eqap3}
\begin{aligned}
D^n_{\lambda,\boldsymbol{\omega}}(u_{\lambda,\boldsymbol{\omega}}^n) &\le D^n_{\lambda,\boldsymbol{\omega}}(P_{\boldsymbol{\omega}}\psi) \\
\end{aligned}
\end{equation}
Then using the expression for the functional as in Equation \ref{functional_redef}, we have
\begin{equation}\label{eqap}
\begin{aligned}
  \frac{\lambda^2}{n}\sum\limits_{i=1}^{n}(u_{\lambda,\boldsymbol{\omega}}^n(\boldsymbol{p}_i)-\psi(\boldsymbol{p}_i))^2 + \lambda\|\nabla^ku_{\lambda,\boldsymbol{\omega}}^n\|_{L^2(\mathbb{T}^m)}^2 + \|u_{\lambda,\boldsymbol{\omega}}^n\|_{L^2(\mathbb{T}^m)}^2 &\le \\ \frac{\lambda^2}{n}\sum\limits_{i=1}^{n}(P_{\boldsymbol{\omega}}\psi(\boldsymbol{p}_i)-\psi(\boldsymbol{p}_i))^2 + \lambda\|\nabla^kP_{\boldsymbol{\omega}}\psi\|_{L^2(\mathbb{T}^m)}^2 + \|P_{\boldsymbol{\omega}}\psi\|_{L^2(\mathbb{T}^m)}^2
\\
%\implies \frac{\lambda^2}{n}\sum\limits_{i=1}^{n}(u_{\lambda,\boldsymbol{\omega}}^n(\boldsymbol{p}_i)-\psi(\boldsymbol{p}_i))^2 + \lambda\|\nabla^ku_{\lambda,\boldsymbol{\omega}}^n\|_{L^2(\mathbb{T}^m)}^2 + \|u_{\lambda,\boldsymbol{\omega}}^n\|_{L^2(\mathbb{T}^m)}^2  &\le \lambda\|\nabla^kP_{\boldsymbol{\omega}}\psi\|_{L^2(\mathbb{T}^m)}^2 + \|P_{\boldsymbol{\omega}}\psi\|_{L^2(\mathbb{T}^m)}^2
\end{aligned}
\end{equation}

As all terms in the LHS of the above inequality are positive we have   %\begin{equation}
%\label{eq_rate_inter}
%\left( \frac{\lambda^2}{n}\sum\limits_{i=1}^{n}(u_{\lambda,\boldsymbol{\omega}}^n(\boldsymbol{p}_i)-\psi(\boldsymbol{p}_i))^2 \right) &\le \frac{\lambda^2}{n}\sum\limits_{i=1}^{n}(P_{\boldsymbol{\omega}}\psi(\boldsymbol{p}_i)-\psi(\boldsymbol{p}_i))^2 + \lambda\|\nabla^kP_{\boldsymbol{\omega}}\psi\|_{L^2(\mathbb{T}^m)}^2 + \|P_{\boldsymbol{\omega}}\psi\|_{L^2(\mathbb{T}^m)}^2\\
%\end{equation}
\begin{equation}\label{EQ_4_6}
\begin{aligned}
\left( \lambda\|\nabla^ku_{\lambda,\boldsymbol{\omega}}^n\|_{L^2(\mathbb{T}^m)}^2 \right) &\le \frac{\lambda^2}{n}\sum\limits_{i=1}^{n}(P_{\boldsymbol{\omega}}\psi(\boldsymbol{p}_i)-\psi(\boldsymbol{p}_i))^2 + \lambda\|\nabla^kP_{\boldsymbol{\omega}}\psi\|_{L^2(\mathbb{T}^m)}^2 + \|P_{\boldsymbol{\omega}}\psi\|_{L^2(\mathbb{T}^m)}^2 \\
\implies  \left( \|\nabla^ku_{\lambda,\boldsymbol{\omega}}^n\|_{L^2(\mathbb{T}^m)}^2 \right) &\le \frac{\lambda}{n}\sum\limits_{i=1}^{n}(P_{\boldsymbol{\omega}}\psi(\boldsymbol{p}_i)-\psi(\boldsymbol{p}_i))^2 + \|\nabla^kP_{\boldsymbol{\omega}}\psi\|_{L^2(\mathbb{T}^m)}^2 + \frac{1}{\lambda}\|P_{\boldsymbol{\omega}}\psi\|_{L^2(\mathbb{T}^m)}^2
%\left( \|u_{\lambda,\boldsymbol{\omega}}^n\|_{L^2(\mathbb{T}^m)}^2 \right) &\le \frac{\lambda^2}{n}\sum\limits_{i=1}^{n}(P_{\boldsymbol{\omega}}\psi(\boldsymbol{p}_i)-\psi(\boldsymbol{p}_i))^2 +\lambda\|\nabla^kP_{\boldsymbol{\omega}}\psi\|_{L^2(\mathbb{T}^m)}^2 + \|P_{\boldsymbol{\omega}}\psi\|_{L^2(\mathbb{T}^m)}^2.\\
\end{aligned}
\end{equation}
Using the Koksma-Hlawka inequality \cite{kuipers2012uniform} for the convergence of the Riemann integral, as $n$ grows, we have the asymptotic

    \begin{equation}\label{IQ_1}\begin{aligned}
    \left|\frac{1}{n}\sum\limits_{i=1}^{n}(P_{\boldsymbol{\omega}}\psi(\boldsymbol{p}_i)-\psi(\boldsymbol{p}_i))^2  - \|P_{\boldsymbol{\omega}}\psi-\psi\|^2_{L^2(\mathbb{T}^m)} \right| &\le  \zeta_n V_{\mathbb{T}^m}((P_{\boldsymbol{\omega}}\psi-\psi)^2)
    \end{aligned}\end{equation}
    Which implies the following equations
    
    \begin{equation}\label{EQ_4_7}\begin{aligned}
    \frac{1}{n}\sum\limits_{i=1}^{n}(P_{\boldsymbol{\omega}}\psi(\boldsymbol{p}_i)-\psi(\boldsymbol{p}_i))^2  &\le \|P_{\boldsymbol{\omega}}\psi-\psi\|^2_{L^2(\mathbb{T}^m)} + \zeta_n V_{\mathbb{T}^m}((P_{\boldsymbol{\omega}}\psi-\psi)^2)
     \end{aligned}
\end{equation} and
    \begin{equation}\label{IQ_2}\begin{aligned}
    \frac{1}{n}\sum\limits_{i=1}^{n}(P_{\boldsymbol{\omega}}\psi(\boldsymbol{p}_i)-\psi(\boldsymbol{p}_i))^2  &\ge \|P_{\boldsymbol{\omega}}\psi-\psi\|^2_{L^2(\mathbb{T}^m)} - \zeta_n V_{\mathbb{T}^m}((P_{\boldsymbol{\omega}}\psi-\psi)^2) \end{aligned}
\end{equation}

    % \int_{\mathbb{T}^m}\nabla\left((P_{\boldsymbol{\omega}}\psi(\boldsymbol{x})-\psi(\boldsymbol{x}))^2\right)\mathrm{d}^m\boldsymbol{x}\\
    %&\sim \|P_{\boldsymbol{\omega}}\psi-\psi\|^2_{L^2(\mathbb{T}^m)} + 2K_0\zeta_n \int_{\mathbb{T}^m}\left((P_{\boldsymbol{\omega}}\psi(\boldsymbol{x})-\psi(\boldsymbol{x}))\right)\cdot\left((\nabla P_{\boldsymbol{\omega}}\psi(\boldsymbol{x})-\nabla\psi(\boldsymbol{x}))\right)\mathrm{d}^m\boldsymbol{x}\\
    %&\sim \|P_{\boldsymbol{\omega}}\psi-\psi\|^2_{L^2(\mathbb{T}^m)} + 2\zeta_n K_0 \|P_{\boldsymbol{\omega}}\psi-\psi\|_{L^1(\mathbb{T}^m)}\|\nabla P_{\boldsymbol{\omega}}\psi-\nabla \psi\|_{L^1(\mathbb{T}^m)}\\
    %&\le \|P_{\boldsymbol{\omega}}\psi-\psi\|^2_{L^2(\mathbb{T}^m)} + 2\zeta_n K V_{\mathbb{T}^m}\left(P_{\boldsymbol{\omega}}\psi-\psi\right)
   
 %where $K_0$ is a positive constant independent of $u_{\lambda,\boldsymbol{\omega}}^n$.

Using Equations \ref{EQ_4_6} and \ref{EQ_4_7} \begin{equation}\label{grad_growth}
\begin{aligned}
    \|\nabla^ku_{\lambda,\boldsymbol{\omega}}^n\|_{L^2(\mathbb{T}^m)}^2 &\le \lambda\|P_{\boldsymbol{\omega}}\psi-\psi\|^2_{L^2(\mathbb{T}^m)} + \lambda\zeta_n V_{\mathbb{T}^m}((P_{\boldsymbol{\omega}}\psi-\psi)^2)\\
    &+ \|\nabla^kP_{\boldsymbol{\omega}}\psi\|_{L^2(\mathbb{T}^m)}^2 + \frac{1}{\lambda}\|P_{\boldsymbol{\omega}}\psi\|_{L^2(\mathbb{T}^m)}^2
\end{aligned}
\end{equation}
Again as all the terms in the Equation \ref{eqap} are positive, we have 
\begin{equation}\label{eq_sun}
    \frac{1}{n}\sum\limits_{i=1}^{n}(u_{\lambda,\boldsymbol{\omega}}^n(\boldsymbol{p}_i)-\psi(\boldsymbol{p}_i))^2 \le \\ \frac{1}{n}\sum\limits_{i=1}^{n}(P_{\boldsymbol{\omega}}\psi(\boldsymbol{p}_i)-\psi(\boldsymbol{p}_i))^2 + \frac{1}{\lambda}\|\nabla^kP_{\boldsymbol{\omega}}\psi\|_{L^2(\mathbb{T}^m)}^2 + \frac{1}{\lambda^2}\|P_{\boldsymbol{\omega}}\psi\|_{L^2(\mathbb{T}^m)}^2
\end{equation}

Again using the Koksma-Hlawka inequality \cite{kuipers2012uniform} for the convergence of the Riemann integral, as $n$ grows, we have   

\begin{equation}\label{IQ_3}
    \left|\frac{1}{n}\sum\limits_{i=1}^{n}(u_{\lambda,\boldsymbol{\omega}}^n(\boldsymbol{p}_i)-\psi(\boldsymbol{p}_i))^2 - \|u_{\lambda,\boldsymbol{\omega}}^n - \psi\|^2_{L^2(\mathbb{T}^m)} \right| \le  \zeta_n V_{\mathbb{T}^m}((u_{\lambda,\boldsymbol{\omega}}^n - \psi)^2) \end{equation}
which implies the following equations

\begin{equation}\label{eq_sun2}
    \frac{1}{n}\sum\limits_{i=1}^{n}(u_{\lambda,\boldsymbol{\omega}}^n(\boldsymbol{p}_i)-\psi(\boldsymbol{p}_i))^2 \le \|u_{\lambda,\boldsymbol{\omega}}^n - \psi\|^2_{L^2(\mathbb{T}^m)} + \zeta_n  V_{\mathbb{T}^m}((u_{\lambda,\boldsymbol{\omega}}^n - \psi)^2) \end{equation}
    and
\begin{equation}\label{IQ_4}
    \frac{1}{n}\sum\limits_{i=1}^{n}(u_{\lambda,\boldsymbol{\omega}}^n(\boldsymbol{p}_i)-\psi(\boldsymbol{p}_i))^2 \ge \|u_{\lambda,\boldsymbol{\omega}}^n - \psi\|^2_{L^2(\mathbb{T}^m)}  - \zeta_n V_{\mathbb{T}^m}((u_{\lambda,\boldsymbol{\omega}}^n - \psi)^2) \end{equation}

%where $K_1$ is a positive constant independent of $u_{\lambda,\boldsymbol{\omega}}^n$.
%and
%\begin{equation}
%    \frac{1}{n}\sum\limits_{i=1}^{n}(P_{\boldsymbol{\omega}}\psi(\boldsymbol{p}_i)-\psi(\boldsymbol{p}_i))^2 \sim \|P_{\boldsymbol{\omega}}\psi - \psi\|_{L^2(\mathbb{T}^m)} + \zeta_n K V_{\mathbb{T}^m}((P_{\boldsymbol{\omega}}\psi - \psi)^2) \end{equation}
    
Using Equations \ref{IQ_2}, \ref{eq_sun2} and \ref{eq_sun},  
\begin{equation}\label{eq_error}
\begin{aligned}
\|u_{\lambda,\boldsymbol{\omega}}^n - \psi\|^2_{L^2(\mathbb{T}^m)}  -    \|P_{\boldsymbol{\omega}}\psi - \psi\|^2_{L^2(\mathbb{T}^m)} &\le \zeta_n V_{\mathbb{T}^m}((u_{\lambda,\boldsymbol{\omega}}^n - \psi)^2)  + \zeta_n V_{\mathbb{T}^m}((P_{\boldsymbol{\omega}}\psi-\psi)^2)\\
    &+ \frac{1}{\lambda}\|\nabla^kP_{\boldsymbol{\omega}}\psi\|_{L^2(\mathbb{T}^m)}^2 + \frac{1}{\lambda^2}\|P_{\boldsymbol{\omega}}\psi\|_{L^2(\mathbb{T}^m)}^2
\end{aligned}
\end{equation}
Now consider $V_{\mathbb{T}^m}((u_{\lambda,\boldsymbol{\omega}}^n - \psi)^2)$ and writing the expression for the total variation as an intergral of the absolute of the distributional derivative, we have

\begin{equation}\label{moon_mars}
    \begin{aligned}
    V_{\mathbb{T}^m}((u_{\lambda,\boldsymbol{\omega}}^n - \psi)^2) &= \int_{\mathbb{T}^m}\left\|D\left(    (u_{\lambda,\boldsymbol{\omega}}^n(\boldsymbol{x}) - \psi(\boldsymbol{x}))^2 \right)\right\|_2\mathrm{d}^m\boldsymbol{x}\\
    &= \int_{\mathbb{T}^m}\left|u_{\lambda,\boldsymbol{\omega}}^n(\boldsymbol{x}) - \psi(\boldsymbol{x})\right|\left\|(\nabla u_{\lambda,\boldsymbol{\omega}}^n(\boldsymbol{x}) - D\psi(\boldsymbol{x})\right\|_2 \mathrm{d}^m\boldsymbol{x}\\
    &\le \int_{\mathbb{T}^m}\left|u_{\lambda,\boldsymbol{\omega}}^n(\boldsymbol{x}) - \psi(\boldsymbol{x})\right|\mathrm{d}^m\boldsymbol{x}\int_{\mathbb{T}^m}\left\|(\nabla u_{\lambda,\boldsymbol{\omega}}^n(\boldsymbol{x}) - D\psi(\boldsymbol{x})\right\|_2 \mathrm{d}^m\boldsymbol{x}\\
    &\le \int_{\mathbb{T}^m}\left(\left|u_{\lambda,\boldsymbol{\omega}}^n(\boldsymbol{x})\right| + \left|\psi(\boldsymbol{x})\right|\right)\mathrm{d}^m\boldsymbol{x}\int_{\mathbb{T}^m}\left(\left\|(\nabla u_{\lambda,\boldsymbol{\omega}}^n(\boldsymbol{x})\right\|_2 + \left\| D\psi(\boldsymbol{x})\right\|_2\right) \mathrm{d}^m\boldsymbol{x}\\
    &= \left(\|u_{\lambda,\boldsymbol{\omega}}^n\|_{L^1(\mathbb{T}^m)} + \|\psi\|_{L^1(\mathbb{T}^m)}  \right) \left(\|\nabla u_{\lambda,\boldsymbol{\omega}}^n\|_{L^1(\mathbb{T}^m)} + \|D \psi\|_{L^1(\mathbb{T}^m)} \right)\\
    %& \le \int_{\mathbb{T}^m}\left(u_{\lambda,\boldsymbol{\omega}}^n(\boldsymbol{x})\nabla u_{\lambda,\boldsymbol{\omega}}^n(\boldsymbol{x}) - \psi(\boldsymbol{x})\nabla u_{\lambda,\boldsymbol{\omega}}^n(\boldsymbol{x}) - u_{\lambda,\boldsymbol{\omega}}^n(\boldsymbol{x})\nabla\psi(\boldsymbol{x}) -\psi(\boldsymbol{x})\nabla\psi(\boldsymbol{x})\right)\mathrm{d}^m\boldsymbol{x}\\
    %&\le \|u_{\lambda,\boldsymbol{\omega}}^n\|_{L^1(\mathbb{T}^m)}\|\nabla u_{\lambda,\boldsymbol{\omega}}^n\|_{L^1(\mathbb{T}^m)} -\|\psi\|_{L^1(\mathbb{T}^m)}\|\nabla u_{\lambda,\boldsymbol{\omega}}^n\|_{L^1(\mathbb{T}^m)}\\ &+ \|u_{\lambda,\boldsymbol{\omega}}^n\|_{L^1(\mathbb{T}^m)}\|\nabla \psi\|_{L^1(\mathbb{T}^m)}  -\|\psi\|_{L^1(\mathbb{T}^m)}\|\nabla \psi\|_{L^1(\mathbb{T}^m)}\\ &\le  \|u_{\lambda,\boldsymbol{\omega}}^n\|_{L^1(\mathbb{T}^m)}\|\nabla u_{\lambda,\boldsymbol{\omega}}^n\|_{L^1(\mathbb{T}^m)} +\|\psi\|_{L^1(\mathbb{T}^m)}\|\nabla u_{\lambda,\boldsymbol{\omega}}^n\|_{L^1(\mathbb{T}^m)}\\ &+ \|u_{\lambda,\boldsymbol{\omega}}^n\|_{L^1(\mathbb{T}^m)}\|\nabla \psi\|_{L^1(\mathbb{T}^m)}  +\|\psi\|_{L^1(\mathbb{T}^m)}\|\nabla \psi\|_{L^1(\mathbb{T}^m)}
    \end{aligned}
\end{equation}
(Note that, as $u_{\lambda,\boldsymbol{\omega}}^n\in TP_{\boldsymbol{\omega}}$, it is smooth making both the distributional derivative and the gradient being the same and hence we have written $Du_{\lambda,\boldsymbol{\omega}}^n$ as $\nabla u_{\lambda,\boldsymbol{\omega}}^n$. For the function $\psi\in BV(\mathbb{T}^m)$, $D\psi$ is the distributional/weak derivative).

As $u_{\lambda,\boldsymbol{\omega}}^n$ is smooth, there exists a positive constant $K_4$ independent of $u_{\lambda,\boldsymbol{\omega}}^n$ such that \begin{equation}\begin{aligned}\label{moon_1}
    \|u_{\lambda,\boldsymbol{\omega}}^n\|_{L^1(\mathbb{T}^m)} &\le K_4\|u_{\lambda,\boldsymbol{\omega}}^n\|_{L^{\infty}(\mathbb{T}^m)}
\end{aligned}
\end{equation} 

Similarly as $\nabla u_{\lambda,\boldsymbol{\omega}}^n$ is smooth, there exists a positive constant $K_5$ independent of $u_{\lambda,\boldsymbol{\omega}}^n$ such that \begin{equation}\begin{aligned}\label{moon_2}
   \|\nabla u_{\lambda,\boldsymbol{\omega}}^n\|_{L^1(\mathbb{T}^m)} &\le K_5\|\nabla u_{\lambda,\boldsymbol{\omega}}^n\|_{L^{2}(\mathbb{T}^m)}
\end{aligned}
\end{equation} 
Using Morrey's inequality \cite{evans1998partial}, there exists a positive constant $K_6$ independent of $u_{\lambda,\boldsymbol{\omega}}^n$ such that \begin{equation}\label{mars_1}\begin{aligned}
    \|u_{\lambda,\boldsymbol{\omega}}^n\|_{L^{\infty}(\mathbb{T}^m)} &\le K_6 \|\nabla^k u_{\lambda,\boldsymbol{\omega}}^n\|_{L^2(\mathbb{T}^m)}\\
\end{aligned}
\end{equation}
Using Friedrich's inequality \cite{zheng2005friedrichs} (a generalization of Poincar{\'e}-Wirtinger inequality \cite{evans1998partial}), there exists a positive constant $K_7$ independent of $u_{\lambda,\boldsymbol{\omega}}^n$ such that \begin{equation}\label{mars_2}\begin{aligned}
    \|\nabla u_{\lambda,\boldsymbol{\omega}}^n\|_{L^{2}(\mathbb{T}^m)} &\le K_7 \|\nabla^k u_{\lambda,\boldsymbol{\omega}}^n\|_{L^2(\mathbb{T}^m)}\\
\end{aligned}
\end{equation}

Combining Equations \ref{moon_1}, \ref{mars_1} we obtain
\begin{equation}\begin{aligned}\label{moon_mars_1}
    \|u_{\lambda,\boldsymbol{\omega}}^n\|_{L^1(\mathbb{T}^m)} &\le K_4K_6 \|\nabla^k u_{\lambda,\boldsymbol{\omega}}^n\|_{L^2(\mathbb{T}^m)}
\end{aligned}
\end{equation}

and Combining Equations \ref{moon_2}, \ref{mars_2} we obtain
\begin{equation}\begin{aligned}\label{moon_mars_2}
    \|\nabla u_{\lambda,\boldsymbol{\omega}}^n\|_{L^1(\mathbb{T}^m)} &\le K_5K_7 \|\nabla^k u_{\lambda,\boldsymbol{\omega}}^n\|_{L^2(\mathbb{T}^m)}
\end{aligned}
\end{equation} 

Using Equations \ref{moon_mars}, \ref{moon_mars_2} and \ref{moon_mars_1} we get
\begin{equation}\label{final_moon_mars}
    \begin{aligned}
    V_{\mathbb{T}^m}((u_{\lambda,\boldsymbol{\omega}}^n - \psi)^2) &\le \left(K_4K_6 \|\nabla^k u_{\lambda,\boldsymbol{\omega}}^n\|_{L^2(\mathbb{T}^m)} + \|\psi\|_{L^1(\mathbb{T}^m)}  \right) \left(K_5K_7 \|\nabla^k u_{\lambda,\boldsymbol{\omega}}^n\|_{L^2(\mathbb{T}^m)} + \|D \psi\|_{L^1(\mathbb{T}^m)} \right)\\
    &= K_8 \|\nabla^k u_{\lambda,\boldsymbol{\omega}}^n\|^2_{L^2(\mathbb{T}^m)} +\left(K_9\|\psi\|_{L^1(\mathbb{T}^m)} + K_{10}\|\nabla \psi\|_{L^1(\mathbb{T}^m)} \right)\|\nabla^k u_{\lambda,\boldsymbol{\omega}}^n\|_{L^2(\mathbb{T}^m)}\\ &+ \|\psi\|_{L^1(\mathbb{T}^m)} \|D \psi\|_{L^1(\mathbb{T}^m)}
    %& \le \int_{\mathbb{T}^m}\left(u_{\lambda,\boldsymbol{\omega}}^n(\boldsymbol{x})\nabla u_{\lambda,\boldsymbol{\omega}}^n(\boldsymbol{x}) - \psi(\boldsymbol{x})\nabla u_{\lambda,\boldsymbol{\omega}}^n(\boldsymbol{x}) - u_{\lambda,\boldsymbol{\omega}}^n(\boldsymbol{x})\nabla\psi(\boldsymbol{x}) -\psi(\boldsymbol{x})\nabla\psi(\boldsymbol{x})\right)\mathrm{d}^m\boldsymbol{x}\\
    %&\le \|u_{\lambda,\boldsymbol{\omega}}^n\|_{L^1(\mathbb{T}^m)}\|\nabla u_{\lambda,\boldsymbol{\omega}}^n\|_{L^1(\mathbb{T}^m)} -\|\psi\|_{L^1(\mathbb{T}^m)}\|\nabla u_{\lambda,\boldsymbol{\omega}}^n\|_{L^1(\mathbb{T}^m)}\\ &+ \|u_{\lambda,\boldsymbol{\omega}}^n\|_{L^1(\mathbb{T}^m)}\|\nabla \psi\|_{L^1(\mathbb{T}^m)}  -\|\psi\|_{L^1(\mathbb{T}^m)}\|\nabla \psi\|_{L^1(\mathbb{T}^m)}\\ &\le  \|u_{\lambda,\boldsymbol{\omega}}^n\|_{L^1(\mathbb{T}^m)}\|\nabla u_{\lambda,\boldsymbol{\omega}}^n\|_{L^1(\mathbb{T}^m)} +\|\psi\|_{L^1(\mathbb{T}^m)}\|\nabla u_{\lambda,\boldsymbol{\omega}}^n\|_{L^1(\mathbb{T}^m)}\\ &+ \|u_{\lambda,\boldsymbol{\omega}}^n\|_{L^1(\mathbb{T}^m)}\|\nabla \psi\|_{L^1(\mathbb{T}^m)}  +\|\psi\|_{L^1(\mathbb{T}^m)}\|\nabla \psi\|_{L^1(\mathbb{T}^m)}
    \end{aligned}
\end{equation}
where $K_8 = K_4K_5K_6K_7$, $K_9 = K_5K_7$ and $K_{10} = K_4 K_5$.

Substituting $\lambda = \zeta^{-\beta}_n$ in Equation \ref{eq_error} we get
\begin{equation}\label{eq_error_2}
\begin{aligned}
\|u_{\lambda,\boldsymbol{\omega}}^n - \psi\|^2_{L^2(\mathbb{T}^m)}  -    \|P_{\boldsymbol{\omega}}\psi - \psi\|^2_{L^2(\mathbb{T}^m)} &\le \zeta_n  V_{\mathbb{T}^m}((u_{\lambda,\boldsymbol{\omega}}^n - \psi)^2)  + \zeta_n V_{\mathbb{T}^m}((P_{\boldsymbol{\omega}}\psi-\psi)^2)\\  &+ \zeta^{\beta}_n\|\nabla^kP_{\boldsymbol{\omega}}\psi\|_{L^2(\mathbb{T}^m)}^2 + \zeta^{2\beta}_n\|P_{\boldsymbol{\omega}}\psi\|_{L^2(\mathbb{T}^m)}^2
\end{aligned}
\end{equation}

As $\boldsymbol{\omega}$ grows, as $\psi$ is a function of bounded variation, we have the asymptotic \cite{Zygmund1959} \begin{equation}\label{f_convergence}
    \|P_{\boldsymbol{\omega}}\psi-\psi\|^2_{L^2(\mathbb{T}^m)} = O(1/\|\omega\|_2)
\end{equation}
Substituting $\|\boldsymbol{\omega}\|_2 = \zeta_n^{-\alpha}$ we get \begin{equation}
    \|P_{\boldsymbol{\omega}}\psi-\psi\|^2_{L^2(\mathbb{T}^m)} = O(\zeta_n^{\alpha})
\end{equation}

From the definition of total variation and from Fourier analysis, we can derive the asymptotic \cite{ Zygmund1959, Tolstov1976, Katznelson_2004} and use of \cite{OpenAI2025ChatGPT}
\begin{equation}\label{tvar_error_asymp}
    \begin{aligned}
      V_{\mathbb{T}^m}((P_{\boldsymbol{\omega}}\psi-\psi)^2) &\le 2\int_{\Omega}|P_{\boldsymbol{\omega}}\psi-\psi|\mathrm{d}^m\boldsymbol{x} V_{\mathbb{T}^m}(P_{\boldsymbol{\omega}}\psi-\psi)\\
      &= O(1/\|\boldsymbol{\omega}\|_2) \Theta(1)\\
      &= O(1/\|\boldsymbol{\omega}\|_2)\\
      &= O(\zeta_n^{\alpha})( \mbox{ after substituting  }\|\boldsymbol{\omega}\|_2 = \zeta_n^{-\alpha})
    \end{aligned}
\end{equation} 

Using derivative as a Fourier multiplier operator and using Plancheral theorem , we have the asymptotic

\begin{equation}\label{kgrad_asymp}
\begin{aligned}
    \|\nabla^kP_{\boldsymbol{\omega}}\psi\|_{L^2(\mathbb{T}^m)}^2  &= O(\|\boldsymbol{\omega}\|_2^{2k-1})\\
    &= O(\zeta_n^{-\alpha(2k-1)})( \mbox{ after substituting  }\|\boldsymbol{\omega}\|_2 = \zeta_n^{-\alpha})
\end{aligned}
\end{equation} 

As $\psi$ is a BV function, we have the asymptotic 
\begin{equation}\label{norm_psum_asymp}
    \|P_{\boldsymbol{\omega}}\psi\|_{L^2(\mathbb{T}^m)}^2 = O(1)
\end{equation}

Therefore using Equations \ref{grad_growth}, \ref{f_convergence}, \ref{tvar_error_asymp}, \ref{kgrad_asymp} and \ref{norm_psum_asymp}, \begin{equation}\label{kgrad_norm_asymp}
\begin{aligned}
\|\nabla^ku_{\lambda,\boldsymbol{\omega}}^n\|_{L^2(\mathbb{T}^m)}^2 &\le \lambda\|P_{\boldsymbol{\omega}}\psi-\psi\|^2_{L^2(\mathbb{T}^m)} + \lambda\zeta_n V_{\mathbb{T}^m}((P_{\boldsymbol{\omega}}\psi-\psi)^2)\\
    &+ \|\nabla^kP_{\boldsymbol{\omega}}\psi\|_{L^2(\mathbb{T}^m)}^2 + \frac{1}{\lambda}\|P_{\boldsymbol{\omega}}\psi\|_{L^2(\mathbb{T}^m)}^2\\
    &= \lambda O(\zeta_n^{\alpha}) + \lambda  O(\zeta_n^{1+\alpha}) + O(\zeta_n^{-\alpha(2k-1)}) + \frac{1}{\lambda}O(1)\\
    &= O(\zeta_n^{\alpha-\beta}) + O(\zeta_n^{1+\alpha-\beta}) + O(\zeta_n^{-\alpha(2k-1)}) + O(\zeta_n^{\beta})\mbox{  after substituting }\lambda = \zeta_n^{-\beta}\\
    &= O(\zeta_n^{\gamma}) \mbox{  where   } \gamma = \min(\alpha-\beta,1+\alpha-\beta,-\alpha(2k-1),\beta)
\end{aligned}
\end{equation}

Using Equations \ref{final_moon_mars} and \ref{kgrad_norm_asymp},
\begin{equation}\label{func_error_var_asymp}
\begin{aligned}
V_{\mathbb{T}^m}((u_{\lambda,\boldsymbol{\omega}}^n - \psi)^2) &= O(\|\nabla^ku_{\lambda,\boldsymbol{\omega}}^n\|_{L^2(\mathbb{T}^m)}^2)\\
& = O(\zeta_n^{\gamma}) \mbox{  where   } \gamma = \min(\alpha-\beta,1+\alpha-\beta,-\alpha(2k-1),\beta)
\end{aligned}
\end{equation}
Equation \ref{func_error_var_asymp} implies the asymptotic
\begin{equation}\label{func_error_var_asymp_2}
\begin{aligned}
\zeta_n  V_{\mathbb{T}^m}((u_{\lambda,\boldsymbol{\omega}}^n - \psi)^2) = O(\zeta_n^{\gamma+1}) \mbox{  where   } \gamma = \min(\alpha-\beta,1+\alpha-\beta,-\alpha(2k-1),\beta)
\end{aligned}
\end{equation}
Equation \ref{tvar_error_asymp} implies the asymptotic
\begin{equation}\label{tvar_error_asymp_2}
    \begin{aligned}
      \zeta_n  V_{\mathbb{T}^m}((P_{\boldsymbol{\omega}}\psi-\psi)^2) &= \zeta_n O(\zeta_n^{\alpha})\\
      &= O(\zeta_n^{1+\alpha})
      \end{aligned}
\end{equation}

Equation \ref{kgrad_asymp} implies the asymptotic
\begin{equation}\label{kgrad_asymp_2}
    \begin{aligned}
      \zeta^{\beta}_n\|\nabla^kP_{\boldsymbol{\omega}}\psi\|_{L^2(\mathbb{T}^m)}^2 &= \zeta_n^{\beta} O(\zeta_n^{-\alpha(2k-1)})\\
      &= O(\zeta_n^{\beta-\alpha(2k-1)})
      \end{aligned}
\end{equation}
Equation \ref{norm_psum_asymp} implies the asymptotic
\begin{equation}\label{norm_psum_asymp_2}
    \zeta^{2\beta}_n\|P_{\boldsymbol{\omega}}\psi\|_{L^2(\mathbb{T}^m)}^2 = O(\zeta_n^{2\beta})
\end{equation} 

Using Equations \ref{eq_error_2}, \ref{func_error_var_asymp_2}, \ref{tvar_error_asymp_2}, \ref{kgrad_asymp_2} and \ref{norm_psum_asymp_2} we get

\begin{equation}
    \begin{aligned}
      \|u_{\lambda,\boldsymbol{\omega}}^n - \psi\|^2_{L^2(\mathbb{T}^m)}  -    \|P_{\boldsymbol{\omega}}\psi - \psi\|^2_{L^2(\mathbb{T}^m)} &= O(\zeta_n^r)\\
      \implies \|u_{\lambda,\boldsymbol{\omega}}^n - \psi\|^2_{L^2(\mathbb{T}^m)} &= O(\zeta_n^r) + O(\zeta_n^{\alpha})\\
      &= O(\zeta_n^r) \mbox{ as  }\alpha > 0
    \end{aligned}
\end{equation}
 where $$r =  \gamma+1 = \min(\min(1+\alpha-\beta,2+\alpha-\beta,1-\alpha(2k-1), 1+\beta),1+\alpha,\beta-\alpha(2k-1),2\beta) $$
 
 Making $r>0$ and also noting the previous assumptions, $\alpha>0$, $\beta>0$ and $k>\frac{m}{2}$, we get the final conditions on $\alpha,\beta,k$ that are required for convergence as below
 %\begin{equation}
  %   \begin{aligned}
  %     0 &< \alpha < \frac{1}{2k-1}\\ 
  %     0 &< \beta < \frac{2k-1}{k}\\
  %      k&>\max(\frac{3}{2},\frac{m}{2}) , k\in \mathbb{N}%&< k
  %      \end{aligned}
 %\end{equation}
 
 \begin{equation}
     \begin{aligned}
     \alpha&>0\\
     \beta&>0\\
      1&<\frac{\alpha}{\beta} <(2k-1)\\
        k&>\frac{m}{2}, k\in \mathbb{N}%&< k
        \end{aligned}
 \end{equation}

 Under these conditions, $$\lim\limits_{n\to\infty}\|u_{\lambda,\boldsymbol{\omega}}^n - \psi\|^2_{L^2(\mathbb{T}^m)} = \lim_{n\to\infty}O(\zeta_n^r) = 0$$
 
 \end{proof}
 
 Thus the BV function $\psi$ is approximated from its scattered data.

%\\ \\ \\ \\ \\ \\ \\ \\ \\ \\ \\ \\ \\ \\ \\ \\ \\ \\
%%%%%%%%%%%%%%%%%%%%%%%%%%%%%%%%%%%%%%%%%%%%%%%%%%%%%%%%%%%%%%%%%%%%%%%%
\section{Compliance with Ethical Standards}

\subsection{Funding}
This study was not funded by any grant or organization.
\subsection{Conflict of Interest}
The author Rajesh Dachiraju declares that he has no conflict of interest.
\subsection{Ethical approval}
This article does not contain any studies with human participants or animals performed by any of the authors.

%\\ \\ \\ \\ \\ \\ \\ \\ \\ \\ \\ \\ \\ \\ \\ \\ \\ \\
%%%%%%%%%%%%%%%%%%%%%%%%%%%%%%%%%%%%%%%%%%%%%%%%%%%%%%%%%%%%%%%%%%%%%%%%

\bibliographystyle{amsplain}
\bibliography{refs}

\begin{appendix}

%\appendix{A1}
\chapter{Existence and Uniqueness of the Minimizer}\label{unique_minimizer}

\begin{definition}
Define the set of functions $S = C^0(\mathbb{T}^m)\cap H^k(\mathbb{T}^m)$
\end{definition}
\begin{definition}
Denote $\Omega = \mathbb{T}^m$
\end{definition}

\begin{definition}
Define the norm $\|.\|_{T^k(\mathbb{T}^m)}$ as
\begin{equation}
\label{tkdef}
 \|f\|_{T^k(\Omega)}^2 =     \|f\|_{L^2(\Omega)}^2 + \lambda \|\nabla^k f\|_{L^2(\Omega)}^2   
\end{equation}
 and $\lambda$ is a positive real constant.
\end{definition}

\section{Minimization Problem}
$\forall f \in S$, minimize the functional
\begin{equation}
\label{eq_edit_3}
C_{\lambda}(f) = \|f\|_{T^k(\Omega)}^2 + \frac{\lambda^2}{n}\sum\limits_{i = 1}^n(f(\boldsymbol{p_i})-a_i)^2
\end{equation}%$$\cancel{\|f\|_{T^k(\Omega)}^2 =     \|f\|_{L^2(\Omega)}^2 + \|\nabla^k f\|_{L^2(\Omega)}^2}

\begin{theorem}
\label{th_eq}
For this particular set $S$, the norm $\|.\|_{T^k(\Omega)}$ is equivalent to the Sobolev norm $\|.\|_{H^k(\Omega)}$.
\end{theorem}

\begin{proof}
% text of proof
As the norms $\|.\|_{T^k(\Omega)}$ for different $\lambda \in \mathbb{R}^+$ are equivalent, for this proof we consider only $\lambda = 1$.
Let $\boldsymbol{l} = (l_1,l_2,l_3,..l_m)\in \mathbb{Z}^m$ and $\alpha$ a multi-index. Let $u_{\boldsymbol{l}}$ be the Fourier series coefficients of $u\in S$, we have

\begin{equation}
    ||u||_{H^k(\Omega)}^2 = ||u||_{L^2(\Omega)}^2 + \sum_{|\alpha| =  k} ||D^{\alpha} u||_{L^2(\Omega)}^2.
\end{equation}
By Plancherel's theorem
\begin{equation}
\sum_{|\alpha| =  k} ||D^{\alpha} u||_{L^2(\Omega)}^2 =   \sum_{|\alpha| =  k} \sum_{\boldsymbol{l} \in \mathbb{Z}^k} ((2\pi)^k \boldsymbol{l^{\alpha}})^{2} |\hat{u}_{\boldsymbol{l}}|^2
=
 \sum_{\boldsymbol{l} \in \mathbb{Z}^k}( |\hat{u}_{\boldsymbol{l}}|^2 \sum_{|\alpha| =  k} ((2\pi)^k \boldsymbol{l^{\alpha}})^{2} )   
\end{equation}
By arithmetic mean-geometric mean inequality, it can be shown that 
\begin{equation}
    \sum_{|\alpha| =  k} ((2\pi)^k \boldsymbol{l^{\alpha}})^{2} \le C_k \sum_{i =  1}^m (2\pi l_i)^{2k}
\end{equation}
with $C_k$ depending only on $k$.
So 
\begin{equation}
    \sum_{|\alpha| =  k} ||D^{\alpha} u||_{L^2(\Omega)}^2 \le C_k \sum_{\boldsymbol{l} \in \mathbb{Z}^k}( |\hat{u}_{\boldsymbol{l}}|^2 \sum_{i = 1}^m (2\pi l_i)^{2k} ) = C_k \sum_{i = 1}^m( \sum_{\boldsymbol{l} \in \mathbb{Z}^k}(2\pi l_i)^{2k}|\hat{u}_{\boldsymbol{l}}|^2 ) 
\end{equation}
Using equation \ref{tkdef} and applying Plancherel's theorem in reverse
\begin{equation}
    \|u\|_{L^2(\Omega)}^2 + \sum_{i = 1}^m( \sum_{\boldsymbol{l} \in \mathbb{Z}^k}(2\pi l_i)^{2k}\hat{u}_{\boldsymbol{l}}^2 ) = \|u\|_{T^k(\Omega)}
\end{equation}
Therefore
\begin{equation}
\|u\|_{H^k(\Omega)} \le D_k \|u\|_{T^k(\Omega)}    
\end{equation}
where $D_k$ a constant depending only on $k$.
We can easily observe that $\|u\|_{H^k(\Omega)} \ge \|u\|_{T^k(\Omega)} $. Hence the norms are equivalent.
\end{proof}

\begin{theorem}
\label{th_uconv}
Given that $k>\frac{m}{2}$, If $u \in H^k(\Omega)$, then
\begin{equation}
u\in L^{\infty}(\Omega)
\end{equation}
and 
\begin{equation}
\|u\|_{L^{\infty}(\Omega)} \le K\|u\|_{H^k(\Omega)}
\end{equation}
with $K$  depending only on $k$ and $m$
\end{theorem}

\begin{proof}
% text of proof
Let us express $u$ in terms of its Fourier series coefficients $\hat u_{\boldsymbol{l}},\boldsymbol{l} \in \mathbb{Z}^m$, via the Fourier series expansion and then the trick is to multiply by 1 in disguise, with $\langle \boldsymbol{l}\rangle := \sqrt{1+|\boldsymbol{l}|^2}$  
\begin{equation}
u(\boldsymbol{x}) = \sum\limits_{\boldsymbol{l}\in \mathbb{Z}^m} \hat u_{\boldsymbol{l}}e^{2\pi i \boldsymbol{l}\cdot \boldsymbol{x}}  = \sum\limits_{\boldsymbol{l}\in \mathbb{Z}^m} \hat u_{\boldsymbol{l}} \langle \boldsymbol{l}\rangle^k \langle \boldsymbol{l}\rangle^{-k} e^{2\pi i \boldsymbol{l}\cdot \boldsymbol{x}} 
\end{equation}
by Hölder's inequality,
\begin{equation}
 |u(\boldsymbol{x})| \le \sum\limits_{\boldsymbol{l}\in \mathbb{Z}^m} \left |\hat u_{\boldsymbol{l}}\langle \boldsymbol{l}\rangle^k \right| \langle \boldsymbol{l}\rangle^{-k} \le \sqrt{\sum\limits_{\boldsymbol{l}\in \mathbb{Z}^m}|\hat u_{\boldsymbol{l}}\langle \boldsymbol{l}\rangle^{k}|^2 \sum\limits_{\boldsymbol{l}\in \mathbb{Z}^m}|\langle \boldsymbol{l}\rangle^{-k}|^2}   
\end{equation} 

By Plancherel's Theorem, $\sqrt{\sum\limits_{\boldsymbol{l}\in \mathbb{Z}^m}|\hat u_{\boldsymbol{l}}\langle \boldsymbol{l}\rangle^{k}|^2} = \|u\|_{H^{k}}$ and 
$K = \sqrt{\sum\limits_{\boldsymbol{l}\in \mathbb{Z}^m}|\langle \boldsymbol{l}\rangle^{-k}|}$ is a constant depending only on  $k,n$, which is finite as $k > m/2$. This completes the proof.
\end{proof}

\begin{theorem}
\label{completeness}
Given that $k>\frac{m}{2}$, any sequence in $S$, that converges in the norm $\|.\|_{T^k}$, also converges uniformly to a limit function in $S$.
\end{theorem}

\begin{proof}
% text of proof
Let $\{f_n\}\to f$ under the norm $\|.\|_{T^k}$, then $\|f_n-f\|_{T^k}\to 0$, so $\|f_n-f\|_{H^k} \to 0$, (as $\|.\|_{T^k}$ is equivalent to $\|.\|_{H^k}$ due to Theorem \ref{th_eq}) and hence due to Theorem \ref{th_uconv}, $\|f_n-f\|_{L^{\infty}(\Omega)} \to 0$. So, as this sequence of continuous functions with periodic boundary conditions converges uniformly, the limit function $f$ is also a continuous function with periodic boundary conditions and so $f\in M $. It is evident that $f \in H^k(\Omega)$, so $f\in S$. 
\end{proof}

\begin{theorem}
Given that $k>\frac{m}{2}$, the minimizer of the functional $C(f)$ over the set $S$ exists and is unique.
\end{theorem}

\begin{proof}
% text of proof
Let $\delta$ be the infimum of $C(f)$ over the set $S$. So there exists a sequence $\{f_n\}, f_n \in S$ such that $C(f_n) \to \delta$. Since both terms of $C(f)$ are positive, due to first term, $\{f_n\}$ is Cauchy under the norm $\|.\|_{T^k}$. Due to theorem \ref{completeness}, $S$ is a closed linear subspace of the Hilbert space $H^k$, and with the inner product induced by restriction, is also a Hilbert space in its own right. Hence the sequence $\{f_n\}$ converges to a limit function $g \in S$ under the norm $\|.\|_{T^k}$, which also means
\begin{equation}
\label{eq_edit_1}
\|f_n\|_{T^k} \to \|g\|_{T^k}
\end{equation}
Again due to theorem \ref{completeness}, $\{f_n\} \to g$ pointwise. So \begin{equation}
\label{eq_edit_2}
f_n(\boldsymbol{p_i}) \to g(\boldsymbol{p_i}), i = 1,2,..N
\end{equation}
Using equations \ref{eq_edit_3}, \ref{eq_edit_1}, \ref{eq_edit_2}, we can say that $C_{\lambda}(f_n)\to C{\lambda}(g)$, and therefore $C{\lambda}(g) = \delta$. Hence as $g \in S$, the infimum of $C{\lambda}(f)$ over set $S$ is attained in $S$. Uniqueness follows from the uniform convexity of $L^2$ norm.
\end{proof}
This proves the existence and uniqueness of the solution to the minimization problem.

\chapter{Expression for the minimizer in the space $TP_{\boldsymbol{\omega}}$}\label{A2}

\begin{theorem}
The solution to the PDE in Equation \ref{eq12b} is $u_{\lambda}$, which is given as
\begin{equation}
    \label{exp_th2p}  
  u_{\lambda}(\boldsymbol{x}) = \sum\limits_{i=1}^n \frac{c_i}{n}w_{\lambda}(\boldsymbol{x}-\boldsymbol{p}_i),
 \end{equation}
 where
 \begin{equation}\label{eq22_th2p}
%\label{greenf}
    w_{\lambda}(\boldsymbol{x}) = P_{\boldsymbol{\omega}}g_{\lambda}(x) 
\end{equation}
$\pmb{c} = [c_1,c_2,...c_{N}]^T$ is given as \begin{equation}\label{eq28_thp}
%\label{iq2}
 \pmb{c} = (\frac{1}{n}W_{\lambda}+\frac{1}{\lambda^2}I)^{-1}L,\end{equation} where the matrix $W_{\lambda}$ is given as 
\begin{equation}\label{eq29_th2p}
% \label{iqd}
  W_{\lambda} = [\gamma_{ij}(\lambda)]_{n\times n},\gamma_{ij}(\lambda) = w_{\lambda}(\boldsymbol{p}_i-\boldsymbol{p}_j)\end{equation} and $$L = [q_1,q_2,\ldots q_n]^T.$$ 
\end{theorem}

\begin{proof}

Consider the following PDE equation:

\begin{equation}
\begin{aligned}
\label{sub_ELp}
 -\int_{\mathbb{T}^m} \phi(\boldsymbol{x})\delta(\boldsymbol{x})\mathop{}\!\mathrm{d}^m\boldsymbol{x} + \lambda\int_{\mathbb{T}^m} \nabla^k\phi(\boldsymbol{x})\cdot\nabla^k f(\boldsymbol{x})\mathop{}\!\mathrm{d^m}x &+ \int_{\mathbb{T}^m} \phi(\boldsymbol{x})f(\boldsymbol{x})\mathop{}\!\mathrm{d}^m\boldsymbol{x}\\ &= 0 \forall \phi\in TP_{\boldsymbol{\omega}}.    
\end{aligned}
\end{equation}
      
Let $g$ be its solution. Now, consider the equation 
\begin{equation}
\label{sub_EL_2p}
    \begin{aligned}
        -\frac{c_i}{n}\int_{\mathbb{T}^m} \phi(\boldsymbol{x})\delta(\boldsymbol{x}-\boldsymbol{p}_i)\mathop{}\!\mathrm{d}^m\boldsymbol{x} + \lambda\int_{\mathbb{T}^m} \nabla^k\phi(\boldsymbol{x})\cdot\nabla^k f(\boldsymbol{x})\mathop{}\!\mathrm{d}^m\boldsymbol{x} &+ \int_{\mathbb{T}^m} \phi(\boldsymbol{x})f(\boldsymbol{x})\mathrm{d^m}x\\ &= 0 \forall \phi\in TP_{\boldsymbol{\omega}}.  
    \end{aligned}
\end{equation}

Substituting $f = c_iw(\boldsymbol{x}-\boldsymbol{p}_i)$ in the LHS of the equation \ref{sub_EL_2p} and  denoting it as $J$, we obtain 
\begin{equation}
\label{der_sub_EL_1p}
\begin{aligned}
J(\phi) = -\frac{c_i}{n}\int_{\mathbb{T}^m} \phi(\boldsymbol{x})\delta(\boldsymbol{x}-\boldsymbol{p}_i)\mathop{}\!\mathrm{d}^m\boldsymbol{x} &+ \lambda\int_{\mathbb{T}^m} \nabla^k\phi(\boldsymbol{x})\cdot\nabla^k c_iw(\boldsymbol{x}-\boldsymbol{p}_i)\mathop{}\!\mathrm{d}^m\boldsymbol{x}\\
&+ \int_{\mathbb{T}^m} \phi(\boldsymbol{x})c_iw(\boldsymbol{x}-\boldsymbol{p}_i)\mathop{}\!\mathrm{d}^m\boldsymbol{x}.\\
\end{aligned}
\end{equation}
Substituting $\boldsymbol{t} = \boldsymbol{x}-\boldsymbol{p}_i$, we obtain
\begin{equation}
    \label{der_sub_EL_2p}
    \begin{aligned}
     J(\phi) = -\frac{c_i}{n}\int_{\mathbb{T}^m} \phi(\boldsymbol{t}+\boldsymbol{p}_i)\delta(\boldsymbol{t})\mathop{}\!\mathrm{d}^m\boldsymbol{t} &+ \frac{c_i}{n}\lambda\int_{\mathbb{T}^m} \nabla^k\phi(\boldsymbol{t}+\boldsymbol{p}_i)\cdot\nabla^k w(\boldsymbol{t})\mathop{}\!\mathrm{d}^m\boldsymbol{t}\\
     &+ \frac{c_i}{n}\int_{\mathbb{T}^m} \phi(\boldsymbol{t}+\boldsymbol{p}_i)w(\boldsymbol{t})\mathop{}\!\mathrm{d}^m\boldsymbol{t}.\\
     \end{aligned}
\end{equation}

Let $\theta(\boldsymbol{t}) = \phi(\boldsymbol{t}+\boldsymbol{p}_i)$, so we have 

\begin{equation}
    \label{der_sub_EL_3p}
     J(\phi) = \frac{c_i}{n}\left\{-\int_{\mathbb{T}^m} \theta(\boldsymbol{t})\delta(t)\mathop{}\!\mathrm{d}^m\boldsymbol{t} + \lambda\int_{\mathbb{T}^m} \nabla^k\theta(\boldsymbol{t})\cdot\nabla^k w(\boldsymbol{t})\mathop{}\!\mathrm{d}^m\boldsymbol{t} + \int_{\mathbb{T}^m} \theta(\boldsymbol{t})w(\boldsymbol{t})\mathop{}\!\mathrm{d}^m\boldsymbol{t}\right\}.\\
\end{equation}

For every $\phi \in TP_{\boldsymbol{\omega}}$, we have $\theta \in TP_{\boldsymbol{\omega}}$, and using the fact that $w(t)$ is the solution of the Equation \ref{sub_ELp}, we have

\begin{equation}
    J(\phi) = 0 \forall \phi\in TP_{\boldsymbol{\omega}}.
\end{equation}

Hence, $\frac{c_i}{n}w(\boldsymbol{x}-\boldsymbol{p}_i)$ is a solution to Equation \ref{sub_EL_2p}.  Writing Equation \ref{sub_EL_2p} with different $c_i$, $i = 1,2,3...n$ and substituting $f=\frac{c_i}{n}w(\boldsymbol{x}-\boldsymbol{p}_i)$ in the $i^{th}$ equation (as it the solution of that equation), and adding up all the $n$ equations, we obtain 

\begin{equation}
    \label{der2_sub_ELp}
    \begin{aligned}
     \sum\limits_{i=1}^n\left\{-\frac{c_i}{n}\int_{\mathbb{T}^m} \phi(\boldsymbol{x})\delta(\boldsymbol{x}-\boldsymbol{p}_i)\mathop{}\!\mathrm{d}^m\boldsymbol{x}\right\} &+ \lambda\int_{\mathbb{T}^m}\\ \nabla^k\phi(\boldsymbol{x})\cdot\nabla^k\left( \sum\limits_{i=1}^n \frac{c_i}{n}w(\boldsymbol{x}-\boldsymbol{p}_i)\right)\mathop{}\!\mathrm{d}^m\boldsymbol{x} &+ \int_{\mathbb{T}^m} \phi(\boldsymbol{x})\sum\limits_{i=1}^n \frac{c_i}{n}w(\boldsymbol{x}-\boldsymbol{p}_i)\mathop{}\!\mathrm{d}^m\boldsymbol{x} = 0\\ \forall \phi\in TP_{\boldsymbol{\omega}}.
     \end{aligned}
\end{equation}

Denoting $u_{\lambda} = \sum\limits_{i=1}^n \frac{c_i}{n}w(\boldsymbol{x}-\boldsymbol{p}_i)$ and assuming $c_i = \lambda^2(q_i-f(\boldsymbol{p}_i))$ and noting that $\int_{\mathbb{T}^m}\phi(\boldsymbol{x})\delta(\boldsymbol{x}-\boldsymbol{p}_i)\mathrm{d}x = \phi(\boldsymbol{p}_i)$, we can rewrite Equation \ref{der2_sub_ELp} as 

\begin{equation}
    \label{der3_sub_ELp}
    \begin{aligned}
      -\frac{\lambda^2}{n}\sum\limits_{i=1}^n (q_i-u_{\lambda}(\boldsymbol{p}_i))\phi(\boldsymbol{p}_i) + \lambda\int_{\mathbb{T}^m} \nabla^k\phi(\boldsymbol{x})\cdot\nabla^k u_{\lambda}(\boldsymbol{x})\mathop{}\!\mathrm{d}^m\boldsymbol{x} &+ \int_{\mathbb{T}^m} \phi(\boldsymbol{x})u_{\lambda}(\boldsymbol{x})\mathop{}\!\mathrm{d}^m\boldsymbol{x}\\
     &= 0 \forall \phi\in TP_{\boldsymbol{\omega}},
     \end{aligned}
\end{equation}

which is same as the E--L equation, as in Equation \ref{eq12b}. Hence, $$u_{\lambda}(\boldsymbol{x}) = \sum\limits_{i=1}^n \frac{c_i}{n}w(\boldsymbol{x}-\boldsymbol{p}_i)$$ is the solution of the E-L equation. However, we still have no expression for $c_i$ and $w(\boldsymbol{x})$. To determine $w$, we need to solve Equation \ref{sub_ELp} as $w$ is its solution. Let $\boldsymbol{l} = (l_1,l_2,l_3,..l_m)\in \mathbb{Z}^m$. Let $\hat{w}_{\boldsymbol{l}}$ and $\hat{\phi}_{\boldsymbol{l}}$ be the Fourier series coefficients of $w$ and $\phi$. Using Parseval's theorem, we have the following equations

\begin{equation}\label{eq14p}
\int_{\mathbb{T}^m}\nabla^k\phi(\boldsymbol{x})\cdot \nabla^k w(\boldsymbol{x}) \mathop{}\!\mathrm{d}^m\boldsymbol{x} = \sum\limits_{i = 1}^m( \sum\limits_{\boldsymbol{l} \in \mathbb{Z}^k}(2\pi l_i)^{2k}\hat{w}_{\boldsymbol{l}} \hat{\phi}_{\boldsymbol{l}}).
\end{equation}

\begin{equation}\label{eq15p}
\int_{\mathbb{T}^m}\phi(\boldsymbol{x})\cdot w(\boldsymbol{x}) \mathop{}\!\mathrm{d}^m\boldsymbol{x} = \sum\limits_{\boldsymbol{l} \in \mathbb{Z}^k}\hat{w}_{\boldsymbol{l}} \hat{\phi}_{\boldsymbol{l}}.
\end{equation}

\begin{equation}\label{eq16p}
    \int_{\mathbb{T}^m}\phi(\boldsymbol{x})\delta(\boldsymbol{x})\mathop{}\!\mathrm{d}^m\boldsymbol{x} = \frac{1}{N} \sum\limits_{\boldsymbol{l} \in \mathbb{Z}^k} \hat{\phi}_{\boldsymbol{l}}. 
\end{equation}

Combining these equations in Equation \ref{sub_ELp}, we obtain 

\begin{equation}\label{eq17p}
%\label{peq}
     -\sum\limits_{\boldsymbol{l} \in \mathbb{Z}^k} \hat{\phi}_{\boldsymbol{l}} + \lambda\sum\limits_{i = 1}^m( \sum\limits_{\boldsymbol{l} \in \mathbb{Z}^k}(2\pi l_i)^{2k}\hat{w}_{\boldsymbol{l}} \hat{\phi}_{\boldsymbol{l}}) + \sum\limits_{\boldsymbol{l} \in \mathbb{Z}^k}\hat{w}_{\boldsymbol{l}} \hat{\phi}_{\boldsymbol{l}} = 0 .
\end{equation}

Now consider the function $\theta(x) = \cos{(2\pi\boldsymbol{\eta}\cdot\boldsymbol{x})} + i\sin{(2\pi\boldsymbol{\eta}\cdot\boldsymbol{x})}$ and let $\hat{\theta}_{\boldsymbol{l}}$ be its Fourier series coefficients. Then, by substituting this $\hat{\theta}_{\boldsymbol{l}}$ for $\hat{\phi}_{\boldsymbol{l}}$ in Equation \ref{eq17p}, we obtain 

\begin{equation}\label{eq19p}
    -1 + \lambda\sum\limits_{i = 1}^m (2\pi \eta{_i})^{2k}\hat{w}_{\boldsymbol{\eta}} + \hat{w}_{\boldsymbol{\eta}} = 0, 
\end{equation}

which implies 
\begin{equation}\label{eq20p}
    \hat{w}_{\boldsymbol{\eta}} = \frac{1}{1+2\pi\lambda\|\boldsymbol{\eta}\|_{2k}^{2k}}.  
\end{equation}

Applying this for each of $\boldsymbol{\eta} \le \boldsymbol{\omega}$, we obtain the solution for Equation \ref{sub_ELp} as $w$ whose Fourier series coefficients $\hat{w}_{\boldsymbol{l}}$ are given as \begin{equation}\label{eq21p}\begin{aligned}
    \hat{w}_{\boldsymbol{l}} &=  \frac{1}{1+\lambda\|\boldsymbol{l}\|_{2k}^{2k}}, \forall \boldsymbol{l}\in\mathbb{Z}^m\land-\boldsymbol{\omega}\le\boldsymbol{l}\le\boldsymbol{\omega}\\%\boldsymbol{l}\in\mathbb{Z}^m .
    &= 0, \mbox{ elsewhere, as } \theta\in TP_{\boldsymbol{\omega}}
    \end{aligned}
\end{equation} 

Let us denote this solution as $w_{\lambda}$. Thus, by Fourier series expansion, we obtain 

\begin{equation}\label{eq22p}
%\label{greenf}
    w_{\lambda}(\boldsymbol{x}) =  \sum_{\boldsymbol{l}\in\mathbb{Z}^m \land -\boldsymbol{\omega} \le \boldsymbol{l}\le \boldsymbol{\omega} } \frac{1}{1+\lambda\|\boldsymbol{l}\|_{2k}^{2k}} \cos{(2\pi\boldsymbol{l}\cdot\boldsymbol{x})}.
\end{equation}
Comparing this Equation with Equation \ref{eq22p}, we can note that \begin{equation}
    w_{\lambda} = P_{\boldsymbol{\omega}}g_{\lambda}
\end{equation}

Using $c_i = \lambda^2(q_i-f(\boldsymbol{p}_i))$ and that $u_{\lambda}(\boldsymbol{x}) = \frac{1}{n}\sum\limits_{i = 1}^n c_i w_{\lambda}(\boldsymbol{x}-\boldsymbol{p}_i)$ substituting the values of $u_{\lambda}(\boldsymbol{p}_i)$ from the later expression in the former equation, we obtain $n$ equations in $n$ unknowns $c_i$. Thus, we can solve for the $c_i$. Further, we obtain a matrix expression for $\pmb{c} = [c_1,c_2,...c_{n}]^T$ and is given as \begin{equation}\label{eq28p}
%\label{iq2}
 \pmb{c} = (\frac{1}{n}W_{\lambda}+\frac{1}{\lambda^2}I)^{-1}L,\end{equation} where the matrix $W_{\lambda}$ is given as 
\begin{equation}\label{eq29p}
% \label{iqd}
  W_{\lambda} = [\gamma_{ij}(\lambda)]_{n\times n},\gamma_{ij}(\lambda) = w_{\lambda}(\boldsymbol{p}_i-\boldsymbol{p}_j)
  \end{equation} and $L = [q_1,q_2,\ldots q_n]^T$. As the existence and uniqueness of the minimizer were already established, it is safe to assume that the matrix $\frac{1}{n}W_{\lambda}+\frac{1}{\lambda^2}I$ is invertible, allowing us to determine the unique minimizer of the functional as \begin{equation}
    \label{expp}  
  u_{\lambda}(\boldsymbol{x}) = \sum\limits_{i=1}^n \frac{c_i}{n}w_{\lambda}(\boldsymbol{x}-\boldsymbol{p}_i).
 \end{equation}
 \end{proof}
\end{appendix}
%\end{appendices}
%\end{appendix}
\end{document}